\documentclass[11pt]{amsart}

\usepackage{amsmath,amsthm,amssymb}
\usepackage{enumerate}
\usepackage{float}
\usepackage[all,cmtip]{xy}
\usepackage[utf8]{inputenc}
\usepackage{tikz}
\usetikzlibrary{arrows,shapes,trees}
\usepackage{graphicx}
\usepackage{hyperref}
\usepackage{amssymb}
\usepackage{amsfonts}
\usepackage{amsthm}
\usepackage{amsmath}
\usepackage{epstopdf}
\usepackage{mathrsfs}

\title[Minimality of strong foliations of Anosov diffeomorphisms]{Minimality of strong foliations of Anosov and partially hyperbolic diffeomorphisms}
\author[A. Avila, S. Crovisier, A. Wilkinson]{A. Avila, S. Crovisier and A. Wilkinson}
\date{\today}

\thanks{A.A. was supported by the SNSF.
S.C. was partially supported by the ERC project 692925 NUHGD.
A.W. was supported by  was supported by NSF grants DMS-1900411 and DMS-2154796.
The authors are grateful to the Institut Henri Poincar\'e for its hospitality
and would like to thank A. Eskin, R. Potrie and Z. Zhang for helpful conversations and comments.
}

\address{Artur Avila \newline
\rm Institut f\"ur Mathematik,
Universit\"at Z\"urich\newline
Winterthurerstrasse 190,
CH-8057 Z\"urich,
Switzerland
\newline \& IMPA, Estrada Dona Castorina 110, Rio de Janeiro, Brazil.}
\address{Sylvain Crovisier \newline
\rm CNRS - Laboratoire de Math{\'e}matiques d'Orsay, UMR 8628\newline
Universit{\'e} Paris-Saclay, 91405 Orsay Cedex, France.}
\address{Amie Wilkinson \newline
\rm Department of Mathematics,
University of Chicago, \newline 5734 S. University Avenue Chicago, Illinois 60637, USA.}

\theoremstyle{plain}
\newtheorem*{theorem*}{Theorem}

\newtheorem{lemma}{Lemma}[section]

\newtheorem{corollary}[lemma]{Corollary}
\newtheorem{proposition}[lemma]{Proposition}

\newtheorem{problem}{Problem}

\newtheorem*{corollary*}{Corollary}
\newtheorem*{claim}{Claim}
\newtheorem{definition}[lemma]{Definition}

\newtheorem{maintheorem}{Theorem} 

\newtheorem{remark}[lemma]{Remark}
\newtheorem*{remarkNoNumber}{Remark}

\newtheoremstyle{vThm*}%
{}{}%
{\itshape}%
{-3pt}{\bfseries}%
{}{ }%
{\thmnote{#3}}%
\theoremstyle{vThm*}
\newtheorem*{nThm*}{}

\def\eps{\varepsilon}

\def\Diff{\operatorname{Diff} }

\def\title{\em}

\def\bar{\overline}

\def\cW{\mathcal{W}}

\def\cA{\mathcal{A}}

\def\cD{\mathcal{D}}
\def\cF{\mathcal{F}}

\def\cF{\mathcal{F}}
\def\cU{\mathcal{U}}

\def\hW{{\widehat{\mathcal{W}}}}

\def\P{\mathcal{P}}

\def\cH{\mathcal{H}}

\def\cP{\mathcal{P}}

\def\hW{\widehat{\mathcal{W}}}

\def\transverse{\,\raise2pt\hbox to1em{\hfil$\top$\hfil}\hskip -1em \hbox
to1em{\hfil$\cap$\hfil}\,} 
\newcommand\R{\mathbb R}
\newcommand\RR{{\mathbb R}}
\newcommand\TT{{\mathbb T}}

\newcommand\ZZ{{\mathbb Z}}
\newcommand\NN{{\mathbb N}}

\newcommand\SL{{\operatorname{SL}}} 
\newlength{\figboxwidth} 

\def\W{\mathcal{W}}

\begin{document}

\maketitle

\begin{abstract}
 We study the topological properties of expanding invariant foliations of $C^{1+}$ diffeomorphisms, in the context of partially hyperbolic diffeomorphisms and laminations with $1$-dimensional center bundle.  

In this first version of the paper, we introduce a property we call {\em s-transversality} of a partially hyperbolic lamination with $1$-dimensional center bundle, which is robust under $C^1$ perturbations. We  prove that under a weak expanding condition on the center bundle (called {\em some hyperbolicity}, or ``SH''),  any s-transverse partially hyperbolic lamination contains a disk tangent to the center-unstable direction (Theorem~\ref{t=s-transPartHyp}).

We obtain several corollaries, among them:  if $f$ is a $C^{1+}$ partially hyperbolic Anosov diffeomorphism with $1$-dimensional expanding center, and the (strong) unstable foliation $\cW^{uu}$ of $f$ is minimal, then  $\cW^{uu}$ is robustly minimal under $C^1$-small perturbations, provided that the stable and strong unstable bundles are not jointly integrable (Theorem~\ref{t=mainopen}).   

Theorem~\ref{t=mainopen} has applications in our upcoming work \cite{A5} with Eskin, Potrie and Zhang, in which we prove that on $\TT^3$, any $C^{1+}$partially hyperbolic Anosov diffeomorphism with $1$-dimensional expanding center has a minimal strong unstable foliation, and has a unique $uu$-Gibbs measure  provided that the stable and strong unstable bundles are not jointly integrable.

In a future work, we address the density (in any $C^r$ topology) of minimality of strong unstable foliations for $C^{1+}$ partially hyperbolic diffeomorphisms with $1$-dimensional center  and the SH property.  Our ultimate goal is to prove that, on any closed manifold, among the $C^r$ ($r>1$) partially  hyperbolic  Anosov diffeomorphisms with $1$-dimensional expanding center, there is a $C^1$-open and $C^r$-dense set of diffeomorphisms with minimal $\cW^{uu}$ foliation.

\end{abstract}

\setcounter{tocdepth}{1}
\tableofcontents 


\section*{Introduction}

Let  $f\colon M\to M$ be a diffeomorphism of a closed, connected Riemannian manifold $M$.  Assume that there exists an  $f$-invariant foliation $\cF^u$ of $M$ with $C^1$ leaves that is {\em expanding}, meaning that there exists $N\in\NN$ such that for all $v \in T\cF^u$ tangent to a leaf of $\cF^u$,  we have $\|Df^N(v)\| \geq 2 \|v\|$ (a given diffeomorphism can have more than one expanding foliation).  Such foliations arise in the study of Anosov and partially hyperbolic diffeomorphisms as well  as in more general scenarios. The dynamics of such foliations have fundamental implications for the dynamics of the underlying diffeomorphisms.  This paper addresses the  dynamics of expanding foliations.

The basic building block  for the topological dynamics of such a foliation $\cF^u$ is an  {\em $\cF^u$-lamination}, which is a (non-necessarily invariant) compact set consisting of a union of leaves of $\cF^u$. The $\cF^u$-lamination is \emph{minimal} if every leaf is dense in the lamination; it is \emph{dynamically minimal} if it is invariant and the orbit of every of its leaf is dense in the lamination.

The fundamental problem to address first is then the following.
\begin{problem}\label{p=orbit closure} Classify the $\cF^u$-laminations for a fixed $\cF^u$.  Is there a unique (non-empty) $\cF^u$-lamination? More generally, when are the minimal $\cF^u$-laminations smooth submanifolds of $M$?
\end{problem}
We remark that  if $f$ has an attractor $\Lambda$, then $\Lambda$ is an $\cF^u$-lamination, but it is also possible to have nontrivial $\cF^u$-laminations without an attractor. For example, the automorphism $f_0 \colon (x,y,z)\mapsto (2x + y, x+ y, z)$ on $\TT^3 = \TT^2\times \TT$ has an expanding foliation $\cF^u$ tangent to the expanding eigenspace, and for any $z_0\in\TT$, 
the set $\TT^2\times \{z_0\}$ is an $\cF^u$-lamination  (which is not an attractor, as $1$ is an eigenvalue of $f_0$).  Notice that for this example, the minimal (under inclusion) $\cF^u$-laminations are smooth submanifolds of $\TT^3$.  

This illustrates a more general phenomenon in homogeneous dynamics.  If $G$ is a connected Lie group and $\Gamma\subset G$ is a cocompact lattice, then any automorphism $A\colon G \to  G$ preserving $\Gamma$ and $a\in G$ determine an affine diffeomorphism $f$ of $M=G/\Gamma$ via $f(g\Gamma) = a\cdot A(g)\Gamma$.   If  ${\mathfrak e}^u$ is the Lie subalgebra  generated by  any collection of expanding  generalized eigenspaces  of the induced map   $\hat f_0\colon \mathfrak g\to \mathfrak g$, then the foliation  $\cF^u$ tangent to  ${\mathfrak e}^u$  is expanding.  Any such subalgebra ${\mathfrak e}^u$ is unipotent, and the work of Dani, Margulis, Ratner, Starkov and Shah 
\cite{Dani, DaniMargulis, Margulis1, Margulis2, Margulis3, Ratner4, Shah, Starkov1, Starkov2}  implies that all  minimal $\cF^u$-laminations are of the form $H/\Gamma$, where $H<G$ is a closed subgroup: in particular, they are submanifolds. Thus Problem~\ref{p=orbit closure} may also  be seen as proposing a nonlinear version of this  homogeneous orbit closure theorem for the dynamics of expanding foliations.

From the perspective of the topological dynamics of $f$, Problem~\ref{p=orbit closure}  has fundamental implications; for example {\em the minimality of an $f$-invariant  foliation  $\cF^u$ implies that  $f$ itself is topologically mixing}.
\medskip

Next, studying $\cF^u$ from a measure-theoretic angle,  fundamental objects are the {\em $\cF^u$-states}, which are $f$-invariant probability measures whose disintegration along the leaves of $\cF^u$ is absolutely continuous with respect to (leafwise) volume.  $\cF^u$-states were first defined by Pesin and Sinai \cite{pesinsinai}, where they proved that $\cF^u$-states always exist. The properties of  $\cF^u$-states (also called u-Gibbs measures) for $f$ are key to understanding the dynamics of $f$; for example, if there is a unique $\cF^u$-state $\mu$, then it is automatically a {\em physical measure} (i.e.,  there is a full volume set  $B\subseteq M$ such that for $x\in B$ and any continuous $\psi$, $\lim_{n\to\infty} \frac1n \sum_0^{n-1} \psi(f^j(x)) = \int \psi\, d\mu $).  Further statistical implications of the uniqueness of $\cF^u$-states are derived in \cite{Dolgustates}.

Problem~\ref{p=orbit closure}  has the following measure-theoretic analogue.
\begin{problem}\label{p=measure classification} Classify the $\cF^u$-states for a fixed $\cF^u$.  Is there a unique $\cF^u$-state? When are  the  ergodic $\cF^u$-states  the smooth invariant measures supported on smooth submanifolds?
\end{problem}
In the homogeneous setting, this problem was solved by Ratner, in her measure classification theorems \cite{Ratner1, Ratner2, Ratner3}, which imply in particular that any $\cF^u$-state is the induced Lebesgue-Haar measure on a homogeneous submanifold.  These measure classification results give the classication of orbit closures above (as one can by averaging  build invariant measures supported on any  orbit closure), thus solving Problem~\ref{p=orbit closure} in this context.
As with Ratner's approach in the homogeneous setting,  solutions to Problem~\ref{p=measure classification} can be used to answer Problem~\ref{p=orbit closure}.

  In the other direction, as demonstrated in our work below, solving Problem~\ref{p=orbit closure}  can be also  a key step toward solving Problem~\ref{p=measure classification}, as the support of any $\cF^{u}$-state is an $\cF^{u}$-lamination.  In the case where $\cF^u$ is the unstable foliation of a $C^2$  transitive Anosov diffeomorphism $f$, we have a complete understanding of its dynamics, going back to the work of Sinai, Ruelle and Bowen: the foliation $\cF^u$ is minimal, and moreover there is a unique $\cF^u$-state (the SRB measure for $f$): when $f$ preserves volume, this measure is volume.  Note that by replacing $f$ by $f^{-1}$, one can similarly classify the dynamics of stable foliations.
\medskip

Beyond Anosov diffeomorphisms, the next natural setting  is when  $f\colon M\to M$ is a partially hyperbolic diffeomorphism with  $Df$-invariant splitting $TM = E^{uu}\oplus E^c\oplus E^{ss}$, and $\cF^u = \cW^{uu}$ is tangent to the (strong) unstable bundle $E^{uu}$. 
The dynamics of $\cW^{uu}$ for $f$ partially hyperbolic have been studied extensively, especially in recent years, focusing largely, but not exclusively, on the case where $\dim E^c = 1$, see e.g.,  \cite{pesinsinai, BDU, Dolgustates, PujSam,  HHU07,GMK, NH, NOH,  GYYZ,  katz, CDP, HUY22, CR, ALOS, EPZ}.   
Denote by $\cP\cH^r(M)$ the set of all $C^r$, partially hyperbolic diffeomorphisms of $M$, and by  $\cP^k\cH^r(M)$ the set of all $f\in \cP\cH^r(M)$ with $\dim E^c = k$.  If the superscript $r$ is not specified, we mean $r=\infty$.

The topological questions in Problem~\ref{p=orbit closure} have been partially  addressed in the case where $\dim E^c=1$, in the works \cite{BDU, HHU07}, where they show that  $C^1$ open and -densely in  $\cP^1\cH(M)$,  at least one of the two foliations $\cW^{uu}$, $\cW^{ss}$ is minimal.  Some of these ideas were generalized by Pujals and Sambarino \cite{PujSam}, who showed (for any $\dim E^c$) that a  weak expansion property called SH (``some hyperbolicity") in $E^c$  (which we describe in detail below) implies that if $\cW^{ss}$ is minimal, then it is {\em robustly minimal} (meaning the $\cW^{ss}$ foliation for any $C^1$ close diffeomorphism is minimal).  The question of whether robust  minimality of {\em both} foliations should hold densely in  $\cP^1\cH(M)$, even when $\dim(M)=3$,  was not settled.

Indeed, a particularly simple situation where this question has remained open, mentioned in several works  (e.g.  \cite[Problem 1.10]{BDU}, \cite[Question 1.1]{NH}, \cite[Conjecture 1.2]{GMK}) is the case where $f\colon \TT^3\to \TT^3$ is both Anosov and partially hyperbolic, with $E^c$ uniformly expanded.   An example is the family  $f_{\varepsilon}(x, y, z)=  (2 x+y+\varepsilon \sin (2 \pi x), x+2 y+z+\varepsilon \sin (2 \pi x), y+z)$ on $\TT^3$, for  $\varepsilon$ small (see \cite{GMK}).

 In this setting on $\TT^3$, the minimality of the foliations $\cW^u$ and $\cW^{ss}$ tangent to $E^u:=E^{uu}\oplus E^c$ and $E^{ss}$ follows from the Anosov property, but the existence of a single example with  $\cW^{uu}$ robustly minimal remained unknown.
 In numerical experiments on  $f_\varepsilon$ and other families, Gogolev, Maimon and Kolmogorov \cite{GMK} found  convincing evidence for  minimality of $\cW^{uu}$ for small $\varepsilon\neq 0$.    They also found numerical evidence of the uniqueness of $\cW^{uu}$-states for these examples.
\medskip

In this paper, we consider first the {\em Anosov diffeomorphisms with expanding $1$-d center}, that is,  the  $f\in \cP^1\cH(M)$ such that $E^c$ is uniformly expanded.  This is a $C^1$ open class, and examples exist in all dimensions $\geq 3$, including these examples in  \cite{GMK}.  For  a perturbation $f$ of {\em any} affine Anosov diffeomorphism on $\TT^3$, either $f$ or $f^{-1}$ belongs to  $\cP^1\cH(\TT^3)$.  In this context, an ultimate aim of a further version of this paper  is to establish the following.
\begin{maintheorem}\label{t=ACWmain} For any closed connected manifold $M$ and $r>1$,  in the space $\cA^r(M)$ of $C^r$  transitive Anosov diffeomorphisms of $M$ with expanding 1-d center, there is a $C^1$ open and $C^r$ dense set with minimal foliation $\mathcal{W}^{u u}$.
\end{maintheorem}
A proof of such a result naturally breaks into two parts:  openness and density of minimality.  The density will be treated in a subsequent version of this paper, where we plan to construct  a residual set in $\cA^r$ on which minimality of $\cW^{uu}$ holds, using a probabilistic perturbative argument in blenders, reminiscent of our methods in \cite{A1}.  

In this version of the paper (Part 1 of the argument), we establish a criterion for openness using a property called {\em s-transversality}, which we introduce here.  In the case where $f$ is partially hyperbolic, if  $\cW^{uu}$ is minimal,  then s-transversality holds  if and only if $E^{uu}\oplus E^{ss}$ is not integrable.   
In our first main result, employing a novel ``topological drift" argument, we show that  s-transversality, which is  $C^1$ robust, implies the minimality of $\cW^{uu}$,  giving the desired openness criterion for $\cW^{uu}$ minimality.   In the context of Anosov diffeomorphisms, we then obtain the following.

\begin{maintheorem}\label{t=mainopen} Let $f\in \cA^r(M)$, for some $r>1$.  Suppose that the foliation  $\cW^{uu}$ is minimal and $E^{uu}\oplus E^s$ is not integrable.  Then $\cW^{uu}$ is $C^1$ robustly minimal in $\cA^r(M)$. 
\end{maintheorem} 

In this (first) version of our paper, we establish Theorem~\ref{t=mainopen} as a consequence of a more general result (Theorem~\ref{t=s-transPartHyp}) we state below, which applies in particular to all $f\in\cP^1\cH(M)$ satisfying the SH property in \cite{PujSam} under additional hypotheses, and even to partially hyperbolic laminations in the absence of a global partially hyperbolic structure.  This latter case uses a generalization of some of the fake foliation techniques introduced in \cite{BW}.

We remark that the measure-theoretic questions in Problem~\ref{p=measure classification} have also been studied in the partially hyperbolic setting.  Based on arguments going back to Benoist and Quint in random dynamics \cite{BQ,EM, BRH, EL}, Katz \cite{katz} established a criterion for uniqueness of $\cW^{uu}$-states for  $C^\infty$ diffeomorphisms in dimension $3$ (here $r=\infty$ is essential: for an alternate  approach for finite $k$, see \cite{ALOS}). Katz's  condition, known as ``QNI" (quantitiative non-integrability) can be defined  for  any ergodic probability measure $\mu$  with exponents $\lambda_1 > \lambda_2 > 0 > \lambda_3$  such that $\mu$ is absolutely continuous along Pesin manifolds for the exponent $\lambda_1$; when QNI holds, the disintegration of $\mu$ along  the {\em full} Pesin unstable manifolds (for $\lambda_1, \lambda_2$) must also be absolutely continuous (thus, in this nonuniform setting, ``QNI $\implies$ SRB"). 
 
 The challenge in applying  Katz's criterion in the ($3$-dimensional) partially hyperbolic setting is to translate the QNI condition into a workable geometric condition on the foliations $\cW^{uu}$ and $\cW^{ss}$.  This has been carried out recently by Eskin, Potrie, and Zhang  \cite{EPZ}, who showed in particular that if $f$ is smooth and partially hyperbolic in dimension $3$ and $\mu$ is an ergodic $\cW^{uu}$-state {\em with full support}, then   $f$ satisfies the QNI condition if and only if  the bundle $E^{uu}\oplus E^{ss}$ is not $\ell$-integrable for some order $\ell$.

Our results on s-transversality,  combined with the work  in \cite{katz, EPZ}, shed complete light on the questions in  \cite{BDU, NH, GMK} about  Anosov diffeomorphisms in $\cP^1\cH(\TT^3)$.  In a forthcoming work with Eskin, Potrie and Zhang we show that for {\em  every} Anosov $f\in \cP^1\cH(\TT^3)$, if $E^{uu}\oplus E^{ss}$ is $\ell$-integrable at any order $\ell$, then it is integrable, and then the  Franks-Manning homeomorphism $h$ conjugating to its linearization $f_\ast$ carries all of the invariant foliations of $f$ to the (affine) invariant foliations for $f_\ast$.  Since these affine foliations are all minimal, we obtain in~\cite{A5}:
\emph{Let $f\in \cA^\infty(\TT^3)$.  Then the foliation $\mathcal{W}^{u u}$ is minimal, and either
$E^{uu}\oplus E^{ss}$ is integrable, or
there is a unique $\cW^{uu}$-state.}

Our results on the density of minimality of $\cW^{uu}$ will be addressed in a subsequent version of this paper. 

\section{s-transversality and statements of the main results}

Suppose $f$ is a $C^{1+}$ diffeomorphism of a Riemannian manifold $M$
and $\Lambda\subseteq M$ is a partially hyperbolic, $f$-invariant set
with a splitting
\[T_\Lambda M = E^{uu}\oplus E^c\oplus E^s;\quad \hbox{dim}(E^c)=1.\]
We assume that $\Lambda$ is a uu-lamination, i.e. it is saturated by $\cW^{uu}$ leaves.

In order to state our main results, we will need to define a notion of non-joint integrability of the laminations $\cW^{uu}$ and $\cW^s$, one which we call  s-transversality.  In the case where $\Lambda=M$ (i.e. where $f\colon M\to M$ is partially hyperbolic), s-transversality is implied by the pair of conditions that $\cW^{uu}$ is minimal and  the bundle $E^{uu}\oplus E^s$ is not integrable (see Proposition~\ref{p=stransversecriterion}).  In the general case, there is subtlety  involved in defining the condition correctly.

\begin{figure}[h]
\begin{center}
\includegraphics[scale=.2]{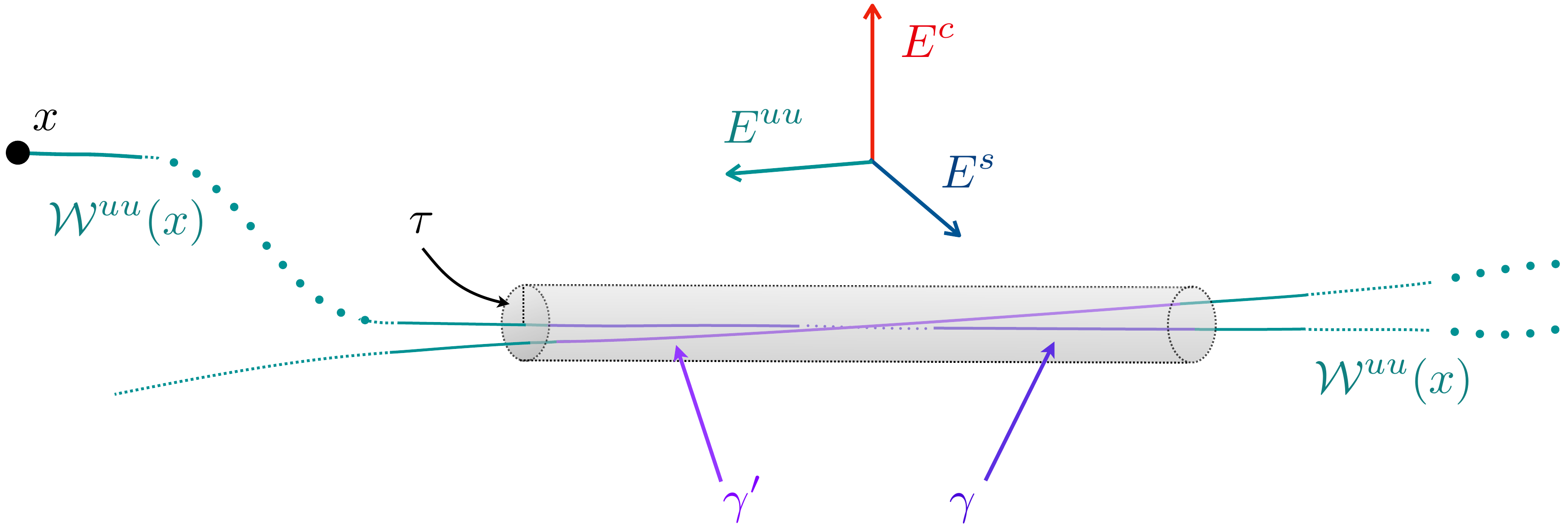}
\caption{Informal definition of s-transversality. }
\label{f=informals-transversality}
\end{center}
\end{figure}

\subsection{The notion of ``self-transversality'' of a foliation}\label{ss.brushes}
The informal definition of s-transversality of $\Lambda$  that we will require  to state our main result is this:  there is a small constant $\tau>0$ such that for any $x\in\Lambda$, there exist paths $\gamma,\gamma'$ in $\cW^{uu}(x)$, at Hausdorff distance within $\tau$ of each other, that cross each other when viewed transversally to the $E^s$ direction, as in Figure~\ref{f=informals-transversality}.

While Figure~\ref{f=informals-transversality} depicts a type of  ``smoothly transverse'' crossing of the two paths, requiring such a nice crossing is too restrictive to be useful. For reasons we will explain, we will a priori need to allow for  paths $\gamma, \gamma'$ lying on a common stable manifold for a large subinterval of their parametrization,  so that the crossing can only be seen by comparing the relative position of the  endpoints of $\gamma, \gamma'$, as depicted in  Figure~\ref{f=accuratecrossing}.

For our definition to be maximally applicable, we need it to be independent of the scale  $\tau$; that is, if it holds for some $\tau$, it should hold for all $\tau'\in(0,\tau)$. To establish such scale-independence requires applying  $f^n$, $n>0$ to paths  (see Proposition~\ref{p.s-transverse}  below), which means (given that the crossings are not assumed to be transverse)  that  we must  allow  the paths $\gamma,\gamma'$ be  arbitrarily long in our definition. 
Thus the  property of $\gamma,\gamma'$ crossing each other ``transverse to the $E^s$ direction" cannot be localized by considering short paths. Even working locally, it is a delicate matter how one defines the relative position of two paths ``transverse to the $E^s$ direction;"  for example,  two nearby curves cannot  be compared by just using linear projections ``approximately in the stable direction,"  as  such projections do not behave well under composition with $f$.

\begin{figure}[h]
\begin{center}
\includegraphics[scale=.25]{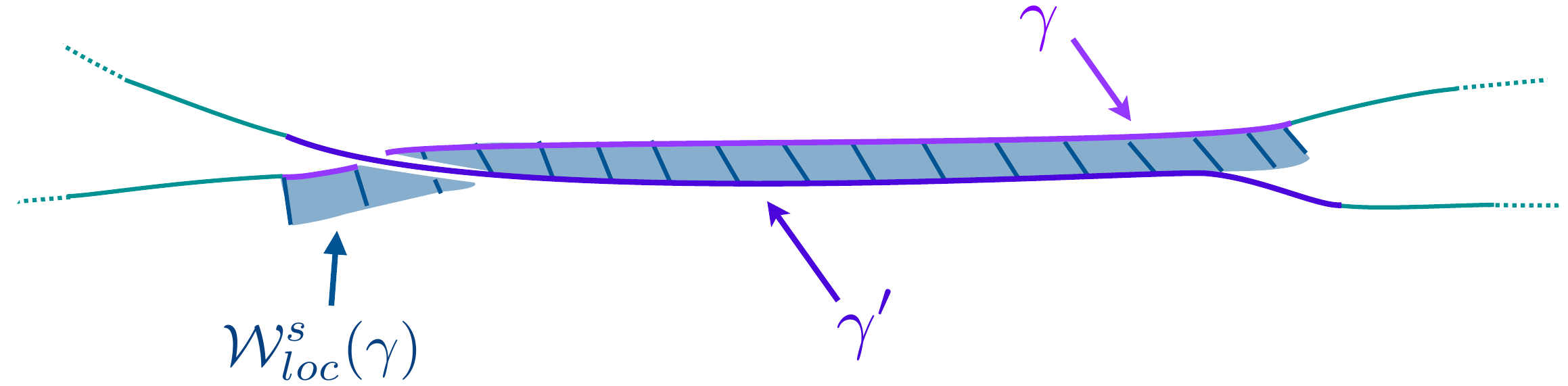}
\caption{ }
\label{f=accuratecrossing}
\end{center}
\end{figure}

Under additional hypotheses, there is a relatively straightforward way to define this crossing property.
 For example, in the case where $f$ is globally partially hyperbolic (so $\Lambda=M$) and dynamically coherent (meaning that the bundles $E^{uu}\oplus E^c$ and $E^c\oplus E^s$ are integrable --  see Section~\ref{s=simplified proof} for definitions), one can fix a smooth disk  $T$ containing $\gamma([0,1])$ and tangent to $E^{uu}\oplus E^c$   and then project $\gamma'$ onto $T$ using the globally-defined $\W^s$-holonomy.  Fixing an orientation on $T$, one can then compare the relative positions of the endpoints of $\gamma$ and $\gamma'$ along leaves of the ($1$-dimensional) center  subfoliation $\cW^c$ of $T$ (which is preserved by $\W^s$ holonomy).  If the relative positions are opposite at the two ends, we say that  $\gamma$ and $\gamma'$ cross (transverse to $\W^s$).  Since the stable foliation is invariant, one can iterate crossing paths to obtain crossing paths, and one can prove that crossing at scale $\tau$ implies crossing at all scales.  Even without dynamical coherence, if $\Lambda = M$, the $1$-dimensionality of  $E^c$ allows for a similar definition of crossing (see Proposition~\ref{p=stransversecriterion}).

We would like to imitate as closely as possible this definition when  $\Lambda\neq M$ is a general $f$-invariant uu-lamination.  In this setting, there is no global invariant  $\W^s$ foliation (not even in a neighborhood of $\Lambda$): local stable manifolds $\W^s_{loc}(x)$ are defined only for $x\in \Lambda$.  Our definition will be based on the picture in Figure~\ref{f=brush}: there the local stable manifold of $\gamma$ divides a $\tau$-neighborhood of $\gamma$ into two components,  and the endpoints of $\gamma'$ lie in different components.  
Thus to define s-transversality, we need to make rigorous this notion of ``two components." We start locally, 
using the notion of {\em brushes}, 
which are  canonically-defined objects built from local stable and unstable manifolds of points in $\Lambda$.

\begin{figure}[h]
\begin{center}
\includegraphics[scale=.2]{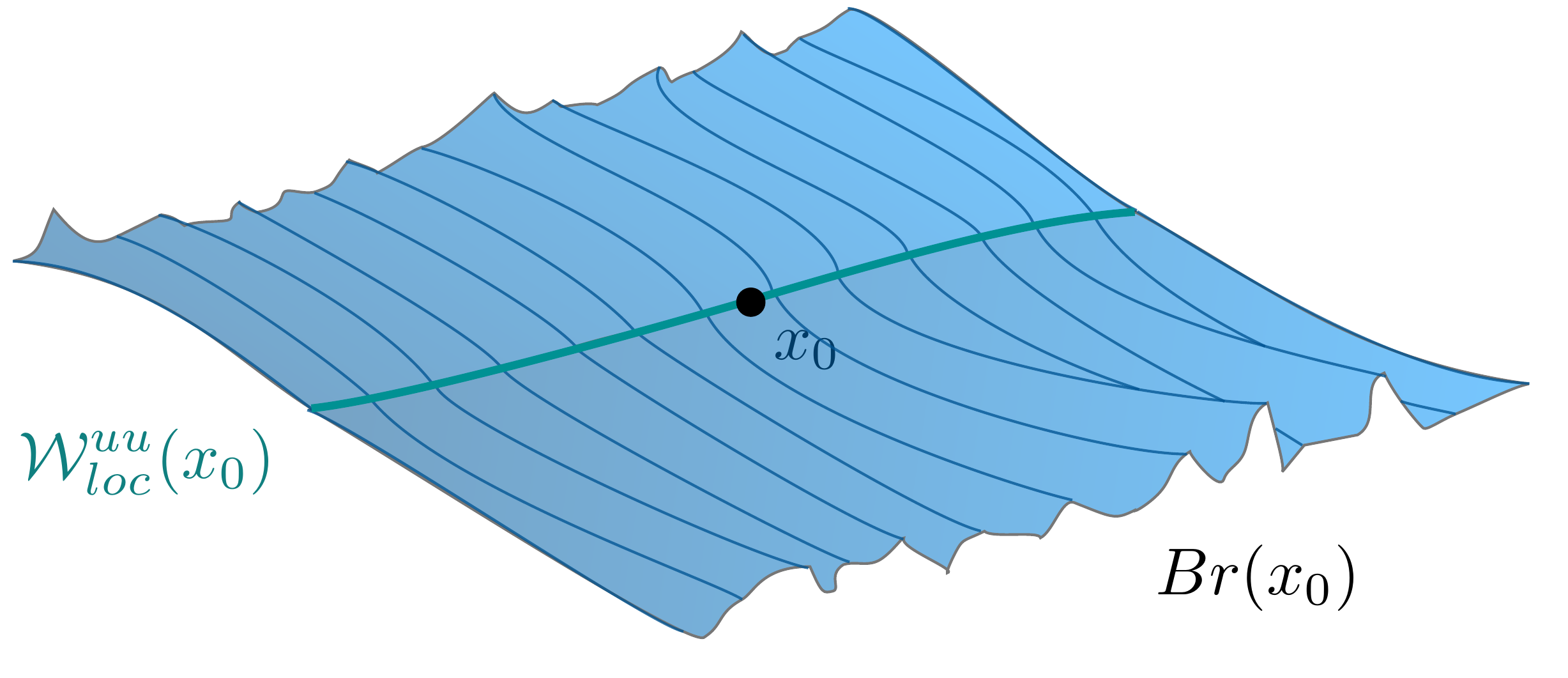}
\caption{A brush. }
\label{f=brush}
\end{center}
\end{figure}

\subsection{Brushes}\label{ss.brush}
Let us fix a small scale $\eps_0>0$.
For $x\in \Lambda$, we set
\[\W^{uu}_{loc}(x):=\W^{uu}(x, {\varepsilon_0}),\; \hbox{ and } \W^{s}_{loc}(x):=\W^{s}(x, {\varepsilon_0}).\]
The {\em brush} through $x$ is defined by
\[Br(x):= \bigcup_{y\in \cW^{uu}_{loc} (x)}\cW^{s}_{loc}(y).\]
It is a codimension-1 topological submanifold which in general is {\em  not} $C^1$.

For $\eps \ll \eps_0$ and $x\in \Lambda$, we consider the set
\[U_{\eps}(x):=B(x,\eps)\setminus Br(x).\]
 If $Br(x)$ {\em were} a $C^1$, codimension-$1$ submanifold (transverse to $E^c$), then it would divide each $U_\eps(x)$ into exactly two path-connected components.  We would like to prove something similar for 
the topological manifold $Br(x)$.
\footnote{In fact we can prove that the brush $Br(x)$ is a locally flat embedded disc (in the sense of~\cite{brown}), hence separates $M$ locally in two components.}

We first expand slightly the definition of path-component of $U_{\eps}(x)$ to allow for paths that leave the neighborhood.
Thus we say that  points $y,y'\in U_{\eps}(x)$
are  {\em in the same component} if they can be joined by a continuous path $\gamma_{y,y'}$ in $U_{\varepsilon_0/C_0}(x)$,  where $C_0>1$ is an explicit constant that only depends on the angle between the bundles
$E^{uu},E^c,E^s$.

\begin{figure}[h]
\begin{center}
\includegraphics[scale=.2]{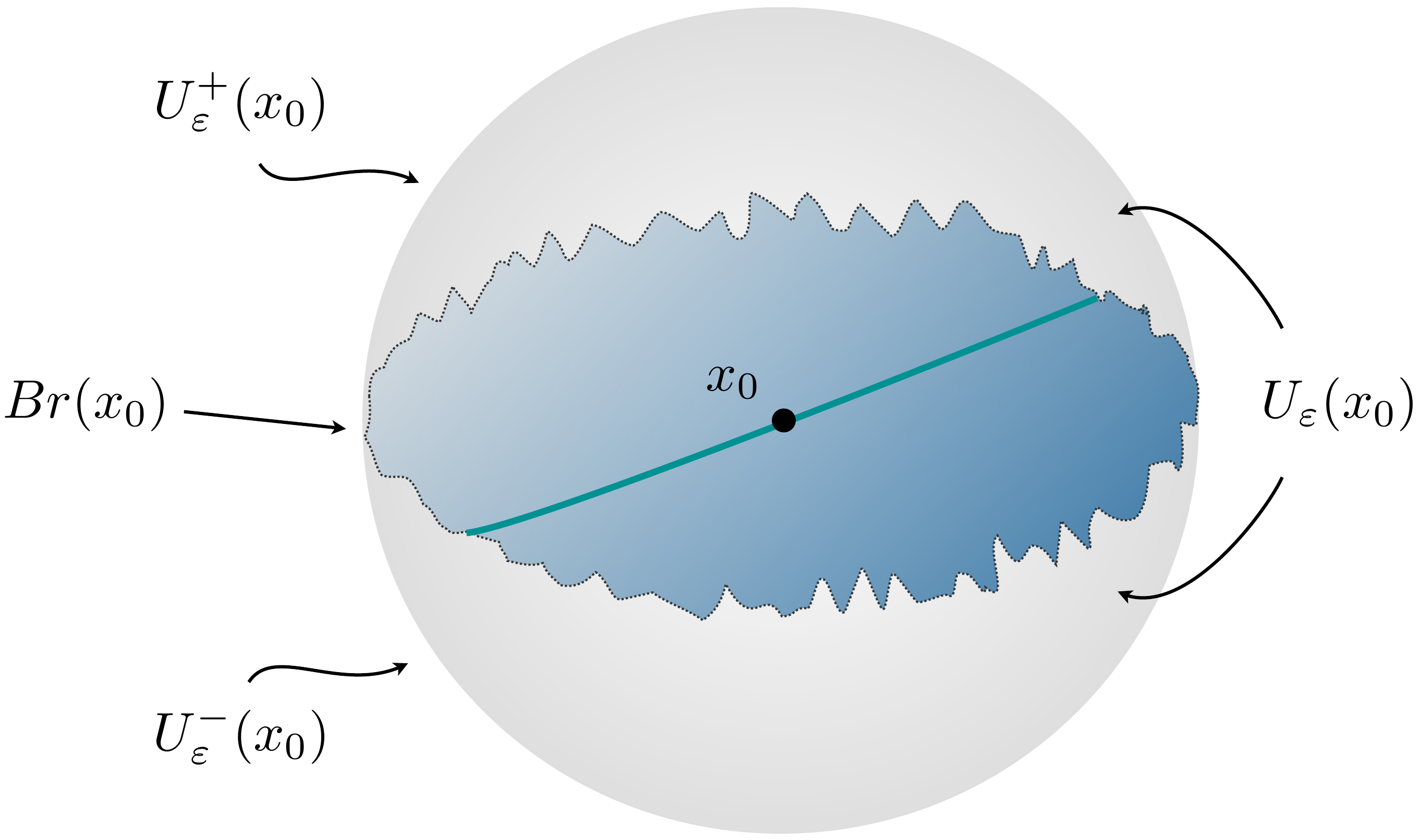}
\caption{The set $U_\eps(x_0)$. }
\label{f=Ueps}
\end{center}
\end{figure}

In Lemma~\ref{l=twocomponents}, we prove that for sufficiently small $\eps>0$,  the set $U_{\eps}(x)$ has exactly two components, $U^\pm_{\eps}(x)$, as in Figure~\ref{f=Ueps}; moreover they do not depend on the choice of $\varepsilon_0$. We say that
$x,y\in U^\pm_{\eps}(x)$ {\em  lie on the same side of $Br(x)$} if they are in the same component of $U_{\eps}(x)$. Since brushes extend canonically along any $\W^{uu}$-leaf, the notion of being on the same side of $Br(x)$ (and on different sides) can then be locally continued along any path  $\gamma$  lying in the $\cW^{uu}(x)$ leaf (see  Lemma~\ref{l.lift}).  Thus  for a path $\gamma$ in a leaf $\cW^{uu}(x)$ there is a well-defined notion of points in $y\in U_\eps(\gamma(0))$ and $y'\in U_\eps(\gamma(1))$  being on the same (or different) side(s) of $Br(\gamma(0))$.  We say that $y,y'$ are {\em on the same (or different) side(s) of $\cW^{uu}(x)$ relative to $\gamma$}.

It is this notion of being on opposite sides along a path that is key to  defining  s-transversality.

\subsection{S-transversality}\label{s=transversality} We can now state the formal definition:

\begin{definition}\label{d.transverse}  
The uu-lamination $\Lambda$ is  {\em s-transverse}
if for any $\tau>0$ small enough and for any $x\in \Lambda$, there are paths $\gamma,\gamma'\colon [0,1]\to \cW^{uu}(x)$
such that:
\begin{enumerate}
\item $d(\gamma(t),\gamma'(t))<\tau$ for each $t\in [0,1]$,
\item $\gamma'(0)$ and $\gamma'(1)$ belong to $U_\tau(\gamma(0))$
and $U_\tau(\gamma(1))$ respectively, but are on different sides of $\cW^{uu}(x)$
relative to $\gamma$. 
\end{enumerate}
It is \emph{locally s-transverse} if $\gamma,\gamma'$ can be chosen with diameter smaller than $\tau$.
\end{definition}

\begin{remark}\label{r.single tau}
In practice, it is enough to check the definition of s-transversa\-lity for a single sufficiently small value of $\tau$ (see Lemma~\ref{l=single tau}).
\end{remark}

In order to illustrate this definition,
we return to the case where  $f$ is a dynamically coherent,  partially hyperbolic diffeomorphism
and  the center bundle $E^c$ is orientable,  with orientation preserved by $Df$.
Note that for any $x,y$ with $d(x,y)<\tau$, the stable plaque $\cW^s_{loc}(y)$
intersects $\cW^{cu}_{loc}(x)$ at a unique point $z$ and that
$\cW^c_{loc}(z)$ meets $\cW^{uu}_{loc}(x)$.
It is easy to see that two points $y,y'$ are in the same component of $U_\tau(x)$
if the center arc joining $\cW^s_{loc}(y)\cap \cW^{cu}_{loc}(x)$
to $W^{uu}_{loc}(x)$ and the center arc joining $\cW^s_{loc}(y')\cap \cW^{cu}_{loc}(x)$
to $W^{uu}_{loc}(x)$ have the same orientation.

In this case, the uu-lamination $\Lambda$ is s-transverse if for $\tau>0$ small
and for any $x\in \Lambda$ there exists a ``hexagonal loop"
consisting of a concatenation of $6$ nontrivial $C^1$ arcs $\gamma_1,\cdots,\gamma_6$ with the following properties (see Figure~\ref{f=hexagon}):
\begin{enumerate}
\item The images of $\gamma_1,\gamma_4$ lie in $\W^{uu}(x)$. For all $t$,
$d(\gamma_1(t),\gamma_4(1-t))<\tau$.
\item The images of $\gamma_2,\gamma_6$ lie in $\W^s_{loc}(\gamma_1(1)),\W^s_{loc}(\gamma_1(0))$ respectively.
\item  $\gamma_3,\gamma_5$ have their images contained in $\W^c_{loc}(\gamma_4(0)),\W^c_{loc}(\gamma_4(1))$ respectively and have opposite orientations.
\end{enumerate}
\begin{figure}[h]
\begin{center}
\includegraphics[scale=.2]{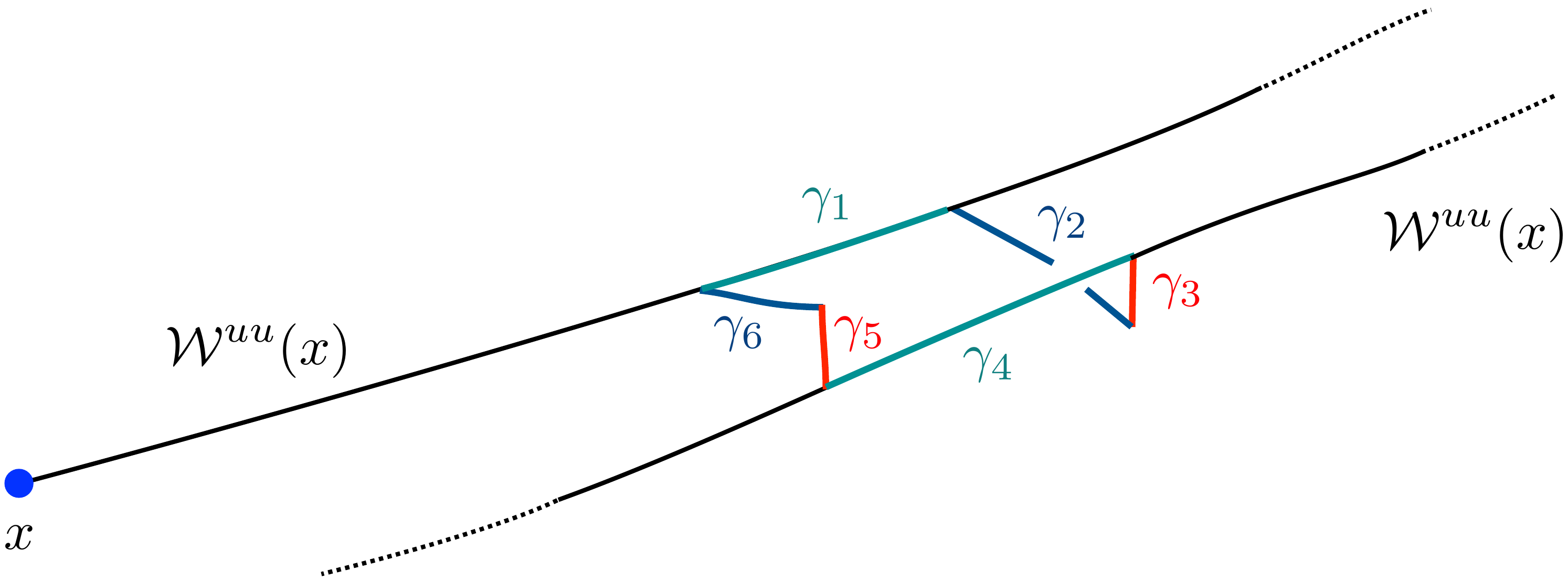}
\caption{Hexagonal loop attached to $\cW^{uu}(x)$.}
\label{f=hexagon}
\end{center}
\end{figure}
This notion of a hexagon gives rise to a simple criterion for checking  s-transversality in the global partially hyperbolic case.

We now  state some properties of s-transversality that will be proved in Section~\ref{ss=cleanup}.
The s-transversality is  $C^1$ robust:

\begin{proposition}\label{p.s-transverse0}
If $\Lambda$ is an invariant s-transverse uu-lamination for a diffeomorphism $f$,
then there are neighborhoods $U$ of $\Lambda$ in $M$ and $\cU$ of $f$ in the space of $C^1$ diffeomorphisms
such that for any $g\in\cU$, any invariant uu-lamination $\Lambda_g\subset U$ is also s-transverse.
\end{proposition}

The next criterion for s-transversality shows in particular that, if $\W^{uu}$ is minimal and s-transverse, then it is locally s-transverse.

\begin{proposition}\label{p=stransversecriterion}  Let $f$ be a partially hyperbolic diffeomorphism with $1$-dimensional center, and assume that the foliation $\W^{uu}$ is
dynamically  minimal.  If  $E^{uu}\oplus E^s$ is not integrable, then  $\W^{uu}$ is locally s-transverse.
\end{proposition}

A weaker notion of ``non joint integrability" has been used in~\cite{CPS} in order to prove the finiteness of attractors
for globally partially hyperbolic diffeomorphisms. The same idea gives the following structure on dynamically minimal
uu-laminations.

\begin{proposition}\label{p=structure-minimal}
If $\Omega$ is a non-invariant minimal s-transverse uu-lami\-na\-tion for a diffeomorphism $f$,
then it is fixed by some iterate $f^N$, $N\geq 1$.

If $\Lambda$ is an invariant dynamically minimal s-transverse uu-lami\-na\-tion, then
it decomposes as $\Lambda= \Omega\sqcup f(\Omega) \sqcup\cdots \sqcup f^{N-1}(\Omega)$,
where $\Omega$ is a minimal uu-lamination. The sets $f^i(\Omega)$ are the connected components of $\Lambda$.
\end{proposition}

\subsection{The SH  (``some hyperbolicity")  property} Another concept we will use in the statements of our main results is that of {\em some hyperbolicity} in  the center bundle $E^c$ of a partially hyperbolic $uu$-lamination  $\Lambda$, which was introduced by Pujals and Sambarino~\cite{PujSam}.  This gives a natural weakening of the hypothesis that $\Lambda$ is  hyperbolic with  a dominated splitting of the unstable bundle $E^{u} = E^{uu}\oplus E^c$.   It has a topological flavor, and is different from the measure-theoretic ``mostly expanding" property defined in~\cite{ABV}.

\begin{definition}\label{d=SH} Let $\Lambda$  be a partially hyperbolic  uu-lamination. 
We say that  $\Lambda$ has the {\em SH  property}  if there exist $R> 0$, $C>0$ and $\lambda_{SH}>1$ such that for every $x\in \Lambda$, there exists $y\in \W^{uu}(x,R)$  such that, for all $n,\ell \geq 1$,
\[\|Df^{-n}\vert_{E^c(f^{\ell+n} x)}\| \leq C\lambda_{SH}^n.
\]
\end{definition}
Note that 
the constant $R$ can be chosen as small as one likes in the definition, at the expense of decreasing $C$.

The SH property is $C^1$ robust (i.e., every  $g$ sufficiently $C^1$ close to $f$ has the SH property), see Proposition~\ref{p=SH laminations open} below.
For (globally) partially hyperbolic diffeomorphisms with the SH property,
if the foliation $\cW^{s}$ is minimal, then  it is $C^1$ robustly minimal (\cite{PujSam} and Proposition~\ref{p=ss minimal} below).
Note that the property SH is not symmetric, hence it does not give the minimality of $\cW^{uu}$ for an Anosov diffeomorphism without further assumptions.

\begin{examples} Besides partially hyperbolic Anosov diffeomorphisms, examples of SH uu-laminations can be constructed by deforming such systems: this is the case of Shub's and Ma\~n\'e's examples, see~\cite{PujSam}.
If one considers an hyperbolic attractor with a one-dimensional expanding center bundle, one can perform similar deformations
and obtain a partially hyperbolic uu-lamination which satisfies the SH property and is proper.
\end{examples}

\subsection{A robust criterion for minimality}
Our main results in Part 1 of this paper concern the property of a uu-lamination containing a {\em $cu$-disk}.  By a $cu$-disk in a uu-lamination $\Lambda$, we mean a $C^1$ embedded disk in $\Lambda$ that is everywhere tangent to the distribution $E^{uu}\oplus E^c$.  For example, in the case that $\Lambda$ is hyperbolic with $1$-dimensional expanding center, a $cu$-disk is merely a disk in $\Lambda$ tangent to the (full) unstable lamination for $f$.
The existence of a $cu$ disk has strong implications for the dynamics of $\Lambda$, including the existence of a weak attractor when $f$ has the SH property, or the minimality of $\cW^{uu}$ (under additional assumptions).  Our central result in this part of this paper is:

\begin{maintheorem}\label{t=s-transPartHyp}
Let  $f\colon M\to M$ be a $C^{1+}$ diffeomorphism and $\Lambda$ an s-transver\-se partially hyperbolic uu-lamination with  $1$-dimensional center and satisfying the SH property.   Then  $ \Lambda$ contains a $cu$-disk.

More precisely, there is a hyperbolic set $K_0$
with unstable bundle  $E^{uu}\oplus E^{c}$ such that $\Lambda$ contains the unstable manifolds of $K_0$.
\end{maintheorem}

\begin{corollary}\label{c=s-transPartHyp}
In the setting of Theorem~\ref{t=s-transPartHyp},
there exists an invariant uu-lamination $\Lambda'\subset \Lambda$ which is a transitive {\em weak attractor}:
$\omega(x)\subset \Lambda'$ for all $x$ in a nonempty open subset of $M$
and $\omega(x)=\Lambda'$ for some $x\in \Lambda$.

If $f$ is transitive and $M$ connected, then $\Lambda=M$ and $\cW^{uu}$ is minimal.
\end{corollary}
\begin{proof}
Let $\Lambda'\subset \Lambda$ be a dynamically minimal $f$-invariant uu-sublamination of $\Lambda$ (meaning that $\Lambda'$ contains no proper invariant sublamination: such an $\Lambda'$  exists by Zorn's lemma). Theorem~\ref{t=s-transPartHyp} implies that $\Lambda'$ contains a disc $D$ tangent to $E^{uu}\oplus E^c$.
Saturating $D$ by local stable manifolds, we obtain an open set $U\subset M$. Then every $x\in U$ lies in $\cW^s(y)$, for some $y\in \Lambda'$, and so $f^n(x)$ accumulates on a subset of $\Lambda'$.  Thus $\Lambda'$ is a weak attractor.
Since any dynamically minimal uu-lamination is a transitive set, $\Lambda'$ is transitive.

If $f$ is transitive, then for $x$ in a dense G$_\delta$ subset of $M$ we have both
$\omega(x)\subset \Lambda'$ and $\omega(x)=M$, so $M=\Lambda=\Lambda'$.
Proposition~\ref{p=structure-minimal} then implies $M= \Omega\sqcup f(\Omega) \sqcup\cdots \sqcup f^{N-1}(\Omega)$
where $\Omega$ is a minimal uu-lamination. If $M$ is connected, it then follows that $N=1$, and hence $\cW^{uu}$ is
a minimal foliation.
\end{proof}

When $\Lambda$ is a hyperbolic set, the SH property is the same as  $\Lambda$ having an expanding center;
if the center is contracting, then $\W^{uu}$ is just the hyperbolic unstable lamination of $\Lambda$.
This set is an attractor, it has finitely many connected components and $\cW^u$ is minimal on each of them.
Note that a hyperbolic transitive weak attractor $\Lambda'$  is an attractor, i.e. $\omega(x)\subset \Lambda'$ for all $x$ in a neighborhood of $\Lambda$. Combining with Proposition~\ref{p=structure-minimal} we thus get:

\begin{corollary}\label{c=s-transHyp}
Let $\Lambda$ be a uniformly hyperbolic, partially hyperbolic uu-lamination for a $C^{1+}$ diffeomorphism $f\colon M\to M$ with  $1$-dimensional   center. 
If $\Lambda$ is s-transverse, then it contains a hyperbolic transitive attractor $\Lambda'$, and each connected component of $\Lambda'$ is a minimal uu-lamination.
\end{corollary}

From the robustness of the SH property (Propositions~\ref{p=SH laminations open}),
Pujals-Sambari\-no's robustness of  $\cW^s$-minimality (Proposition~\ref{p=ss minimal}),
the criterion for s-transver\-sality (Propositions~\ref{p=stransversecriterion}),
and the robustness of the s-transversality (Proposition~\ref{p.s-transverse0}),
Corollary~\ref{c=s-transPartHyp} implies following, which gives Theorem~\ref{t=mainopen}:

\begin{corollary}\label{c=s-transPartHyp3}
Suppose  $f\colon M\to M$ is $C^{1+}$, partially hyperbolic with  $1$-dimensio\-nal center, has the SH property and a minimal $\W^s$ foliation. If $\cW^{uu}$ is minimal and $E^{uu}\oplus E^s$ not integrable,  then  $\cW^{uu}$ is $C^1$ robustly minimal: any
$C^{1+}$ diffeomorphism $C^1$ close to $f$ has a minimal $\cW^{uu}$ foliation.

In particular, if $f$ is an Anosov diffeomorphism with $1$-dimensional expanding center such that $\cW^{uu}$ is minimal and $E^{uu}\oplus E^s$ is not  integrable, then $\cW^{uu}$ is $C^1$ robustly minimal  among $C^{1+}$ diffeomorphisms.
\end{corollary}

\subsection{Genericity of minimality}
In a future version of this work, we will state further results concerning the genericity of $\cW^{uu}$-minimality among (conservative and dissipative) Anosov diffeomorphisms with $1$-dimensional expanding center and related results about partially hyperbolic  laminations  whose $1$-dimensional center has the SH property.


\section{Preliminaries on partial hyperbolicity}\label{s=notations}

\subsection{Notations}\label{ss.notation}
Given a map $f\colon X\to X$, a function $\xi\colon X\to \R_{>0}$ and  $n\geq 1$,  we write $\xi_n(x):=  \xi(f^{n-1}(x))\cdots \xi(f(x))\cdot \xi(x)$, or $\xi_{f,n}(x)$ when we want to emphasize the map $f$.

For $a,b \neq 0$ and $\Delta>1$, we write $a\asymp_\Delta b$ to mean that $a/b\in [\Delta^{-1}, \Delta]$  (in particular if  $a\asymp_\Delta b$, then $a$ and $b$ have the same sign).

\subsection{Partial hyperbolicity}
\subsubsection{Tangent splittings}
Let $ A\subset M$ be a compact set that is invariant under a $C^1$-diffeomorphism $f$.
An invariant splitting of its tangent space $T_A M =E\oplus F$ is \emph{dominated}
if there is $N\geq 1$ such that for any $x\in A$ and any unit vectors $u\in E_x$, $v\in F_x$,
$$\|Df^N(v)\|\leq \tfrac 1 2 \|Df^N(u)\|.$$
The bundle $F$ is \emph{uniformly contracted} by $Df$ if $\|Df^N|_F\|\leq \tfrac 1 2$ for some $N\geq 1$
and $E$ is uniformly expanded if it is uniformly contracted by $Df^{-1}$.

A uniformly expanded bundle $E$ will usually be denoted by $E^u$ (or $E^{uu}$); to each point $x\in A$ we  associate a well-defined strong unstable leaf $\cW^u(x)$ tangent to $E^u$
(also denoted by $\cW^{uu}(x)$ when $E=E^{uu}$). To emphasize the diffeomorphism  $f$, 
we will also write $\cW^u_f(x)$.

For $R>0$, the ball of radius $R$ centered at a point $x\in A$ inside $\cW^u(x)$ 
will be denoted by $\cW^u_R(x)$. Having fixed a small constant $\varepsilon_0>0$,
we will also write $\cW^u_{loc}(x):=W^u_{\varepsilon_0}(x)$ and call it the
\emph{local strong unstable manifold of $x$}.
Finally we set $\cW^u_{loc}(A)=\cup_{x\in A} \cW^u_{loc}(x)$.

Similarly, when $F$ is uniformly contracted we will usually denote it by $E^{s}$ (or $E^{ss}$)
and write  $\cW^{s}(x)$ (or $\cW^{ss}(x)$) for the associated leaves.

\subsubsection{Cone fields}
We associate to a dominated splitting $T_AM =E\oplus F$ fixed
cone fields $\mathcal{C}^E$, $\mathcal{C}^F$
i.e. small neighborhoods of the bundles $E$ and $F$ inside the corresponding tangent Grassmannian bundles.
The domination property implies that
these cone fields can be chosen \emph{invariant}:
for any $x\in  A$, the closure of the image $Df(\mathcal{C}^E_x)$ is
contained in the interior of $\mathcal{C}^E_{f(x)}$, whereas the closure of the
backward image $Df(\mathcal{C}^F_x)$ is
contained in the interior of $\mathcal{C}^F_{f^{-1}(x)}$.

\subsubsection{Partial hyperbolicity}
A compact, $f$-invariant set $A$ is \emph{hyperbolic} if it admits a dominated splitting $T_AM =E^u\oplus E^s$ such that
$E^u,E^s$ are respectively uniformly expanded and contracted.
A compact, $f$-invariant set $A$ is \emph{partially hyperbolic} if it admits a dominated splitting $T_AM =E^{u}\oplus E^c\oplus E^{s}$ such that
$E^u,E^{s}$ are non-trivial and respectively uniformly expanded and contracted.

An \emph{unstable lamination} is a  (not necessarily $f$-invariant) set $\Lambda$ contained in a partially hyperbolic set $A$ such that $\Lambda$ contains all the strong unstable leaves
$\cW^u(x)$ for $x\in \Lambda$.

\subsubsection{Quantitative partial hyperbolicity}\label{ss.partial-hyperbolicity}
Partial hyperbolicity of $f$ implies  that there exist a smooth Riemannian structure on $M$, a neighborhood $U$ of $\Lambda$, a $C^1$-neighborhood $\cU$ of $f$ and continuous functions
\[\kappa, \lambda, \mu^-,\mu^+ \colon U\to (0,\infty), \text{ with } \kappa<1<\lambda
\text{ and }
\kappa < \mu^-<\mu^+ < \lambda\] 
such that for any diffeomorphism $g\in \cU$,  we have that any $g$-invariant compact set $A\subset U$
has a (unique) partially hyperbolic splitting $T_UM =  E^u\oplus E^c\oplus E^{s}$ into subbundles which have the same dimensions as for the partially hyperbolic splitting on $\Lambda$ for $f$
and such that for any $x\in A$:
$$\sup_{v\in E^{s} (x), \|v\|=1} \|D_xg(v) \| < \kappa(x),\qquad  \lambda(x)
<\inf_{v\in  E^u(x), \|v\|=1} \|D_xg(v) \|,$$
$$\mu^-(x)< \|D_xg|_{E^c}\|<\mu^+(x).$$
We also write $E^u_g, E^c_g, E^{s}_g$ in order to emphasize the dependence on $g$.

\subsection{Plaque families}

We introduce a system of locally invariant plaque families which will be useful when considering partially hyperbolic systems that are not dynamically coherent.

\begin{figure}[h]
\begin{center}
\includegraphics[scale=.21]{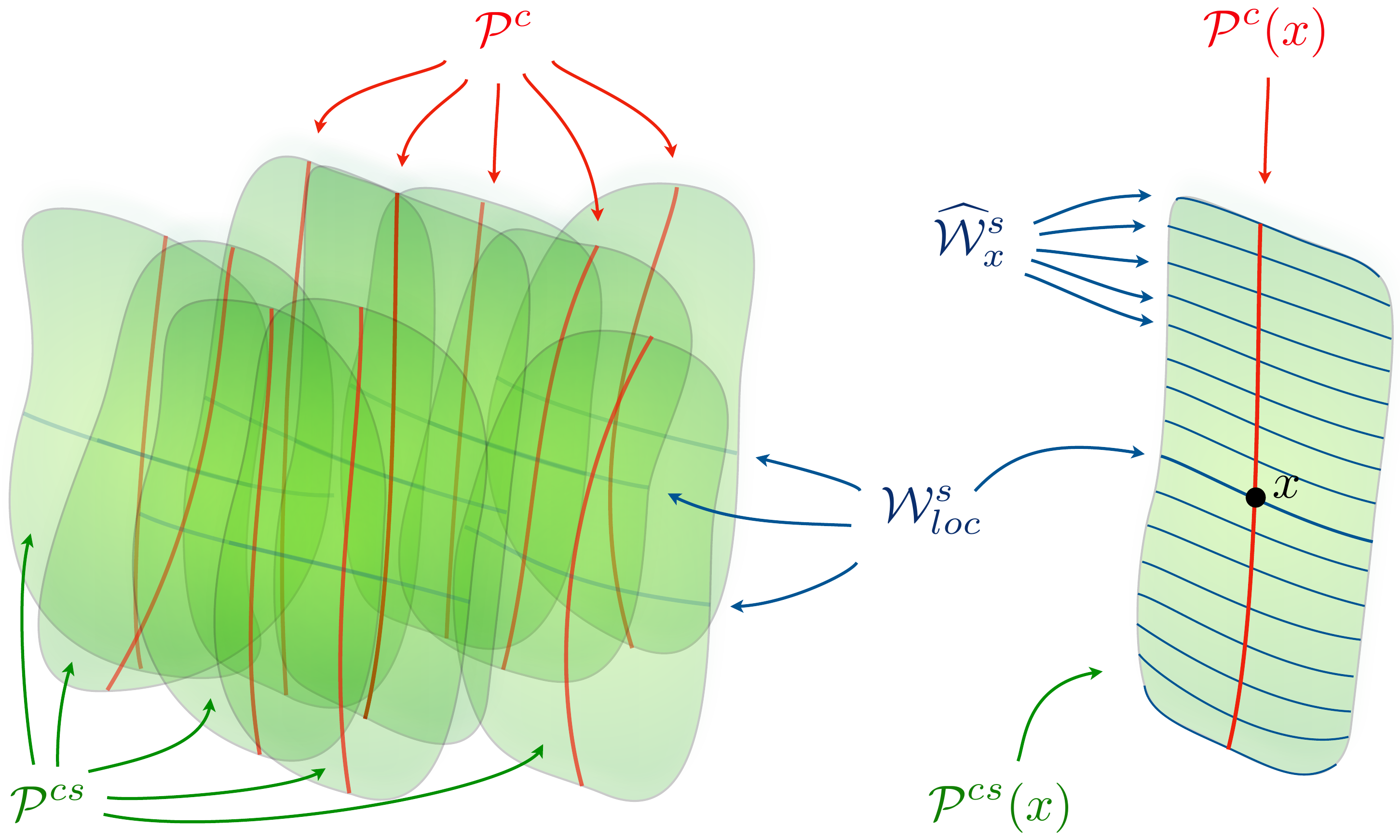}
\caption{Family of locally invariant $\cP^{cs}$ and $\cP^{c}$ plaques, with the $\widehat \W^s_x$ foliation of $\cP^{cs}(x)$ depicted on the right. The construction guarantees that  $\cP^{cs}$ plaques for points in the same $\W^s$ leaf coincide along that  $\W^s$ leaf, and $\widehat\W^s_x \subset \W^s_{loc}(x)$.}
\label{f=plaques}
\end{center}
\end{figure}

\begin{proposition}\label{p.plaque}
Let $f$ be a $C^{1}$ diffeomorphism with a partially hyperbolic set $\Lambda$ and functions  $\kappa, \mu^-, \mu^+ \colon U\to (0,\infty) \text{ with } \kappa<1$  and $\kappa < \mu^- < \mu^+$ as above, on a neighborhood $U$ of $\Lambda$.
Then there exists $\eps>0$, and for any diffeomorphism $g$ that is $C^1$ close to $f$, any $g$-invariant compact set $A\subset U$ and any point $x\in A$,
there exist:
\begin{itemize}
\item[--] $C^1$-embedded disks  (plaques)  $\cP^c_g(x)\subset \cP^{cs}_g(x)$ containing $x$, and
\item[--] a foliation $\widehat \W^s_{g,x}$ of $\cP^{cs}_g(x)$ by $C^1$ discs,
\end{itemize}
with the following properties.
\begin{enumerate}
\item The plaques $\cP^c_g(x)$, $\cP^{cs}_g(x)$ are tangent to $E_g^c(x)$ and $E_g^c(x)\oplus E_g^s(x)$
respectively and $\widehat \W^s_{g,x}(x)\subset \W^s_g(x)$.
\item  Local invariance. For any $x\in A$, any $*\in \{cs, c\}$ and any $y\in \cP^{cs}_g(x,\varepsilon)$, we have
\begin{align*}
g(\cP^{*}_g(x,\varepsilon)) &\subset \cP^{*}_g(g(x)),\quad &g^{-1}(\cP^{*}_g(x,\varepsilon)) &\subset \cP^{*}_g(g^{-1}(x)),\\
g(\widehat \W^{s}_{g,x}(y,\varepsilon)) &\subset \widehat \W^{s}_{g,g(x)}(g(y)),\quad
&g^{-1}(\widehat \W^{s}_{g,x}(y,\varepsilon)) &\subset \widehat \W^{s}_{g,g^{-1}(x)}(g^{-1}(y)),
\end{align*}
where $\cP^c_g(x,\varepsilon)$ denotes the $\varepsilon$-neighborhood of $x$ in $\cP^c_g(x)$.

\item Exponential growth bounds at local scales.  For $x\in M$ and $j\in\ZZ$, denote by $x_j$ the point $g^j(x)$.  Given $n\geq 1$, suppose that $y_j, y_j' \in B(x_j,\eps/2)$ for all $1 \leq j \leq n-1$.
Then the following holds.
\begin{itemize}
\item[--] If $y, y'\in \P^{cs}_g(x,\eps)$,  then $y_n, y_n' \in \P^{cs}_g(x_n)$, and
$$d(y_n,y_n') \leq \mu_{g,n}^+(x) d(y,y') .$$
\item[--] If $y, y'\in \cP^{cs}_g(x,\varepsilon)$, then $y_n, y_n' \in \cP^{cs}_g(x_n)$, and
\[d(y_n,y_n') \leq \mu_{g,n}^+(x) \, d(y,y').\]
\item[--] If $y, y'\in \P^{c}_g(x,\eps)$,  then $y_n, y'_n \in \P^{c}_g(x)$, and
\[\mu_{g,n}^-(x) d(y,y') \leq  d(y_n,y'_n) 
.\]
\end{itemize}
\item  $(g,x)\mapsto \cP^{cs}_g(x), \cP^c_g(x), \widehat \W^s_{g,x}$ vary continuously in the $C^1$ topology.
The spaces $T_q \widehat \W^s_{g,x}(q)$, $q\in \widehat \W^s_{g,x}$,
depend continuously on $(g,x,q)$.
 \item If $g$ is $C^{1+\alpha}$, then  there exists  $\theta\in (0,\alpha)$  such that the tangent bundles to the plaques $\cP^{cs}_g(x)$ and $\cP^c_g(x)$ 
 are uniformly $\theta$-H\"older continuous.
 \end{enumerate}
When there is no ambiguity about the diffeomorphism, we will simply denote the plaques by
$\cP^{cs}(x)$, $\cP^c(x)$, $\widehat \W^s_{x}(y)$.
\end{proposition}
\begin{proof}
The proof follows from the proof of Proposition 3.1 in \cite{BW}.  There it is proved that if $f\colon M\to M$ is $C^1$ and partially hyperbolic, then there exist $r > \varepsilon> 0$ such that, for every $x\in M$, 
the neighborhood $B(x,r)$ is foliated by local
foliations $\hW^s_x$, $\hW^c_x$, and $\hW^{cs}_x$ such that the leaves of  $\hW^s_x$ subfoliate the leaves of  $\hW^{cs}_x$ (from Part (iv) of the proposition).
In the language of Proposition~\ref{p.plaque}, if we set $\cP^c(x) = \hW^c_x(x)$ and $\cP^{cs}(x) = \hW^{cs}_x(x)$, then Properties (1)-(3) and (5) for $f$ (corresponding to Properties 
(i)-(vi) in   \cite[Proposition 3.1]{BW})  are satisfied. The arguments also apply to diffeomorphisms $g$ that are $C^1$-close to $f$. The $C^1$-leaves $\cP^{cs}_g(x), \cP^c_g(x), \widehat \W^s_{g,x}$ at a point $q$ are obtained as a fixed point of a $C^1$ map, hence vary continuously
with $(g,x,q)$ in the $C^1$-topology, which gives Property (4).

As in \cite{BW}, Property (5) follows from a standard application of the H\"older section theorem \cite[Theorem 3.2]{PSW} to the invariant bundles appearing in the proof of  \cite[Proposition 3.1]{BW}. \end{proof}

\subsection{More on the SH property}
We derive some basic properties of $uu$-laminations with the SH property.
\subsubsection{Robustness} We already mentioned that the SH property is robust.
\begin{proposition}[\cite{PujSam}]\label{p=SH laminations open} Let $\Lambda$ be a partially hyperbolic $uu$-lamination   for a diffeomorphism $f\colon M\to M$.  Assume that $\Lambda$ has the SH property. 
Then there exists a neighborhood $U$ of $\Lambda$ and $\cU$ of $f$ in $\Diff^1(M)$ such that for all $g\in \cU$,  every $uu$-lamination $\Lambda_g\subset U$ has the SH property.
\end{proposition}

The SH property is a mechanism for  robust minimality of the $\cW^{s}$ foliation.

\begin{proposition}[\cite{PujSam}]\label{p=ss minimal} Let $f\colon M\to M$ be a diffeomorphism that is partially hyperbolic and satisfies the SH property. If $\cW^s$ is minimal, then any diffeomorphism $C^1$ close to $f$ has a minimal foliation $\cW^s$.

In particular $f$ is $C^1$-robustly topologically mixing.
\end{proposition}

\subsubsection{$\cW^{uu}$-sections}\label{uu-section}
Consider a $uu$-lamination $\Lambda$ for $f$.

\begin{definition}
A set $K\subset \Lambda$ is a \emph{$\cW^{uu}$-section} of $\Lambda$ if it is compact, forward $f$-invariant,
and there exists $R>0$ such that $\cW_R^{uu}(x)\cap K\neq \emptyset$ for each $x\in \Lambda$.
\end{definition}

Note that if $\Lambda$ satisfies the SH property (Definition~\ref{d=SH}) with constants $C>0$ and $\lambda_{SH}>1$, then
the following set is a $\cW^{uu}$-section:
\begin{equation}\label{e=SHsectiondef}
K:=\left\{y \in \Lambda: \|Df^{-n}\mid_{E^c(f^{\ell+n}(y))}\| \leq C \lambda_{SH}^n \text{ for all } n, \ell\geq 1\right\}.
\end{equation}
We associate to $K$  its {\em hyperbolic core} $K_0:= \bigcap_{n\geq 0} f^{n}(K)$, which is an invariant hyperbolic set.
Note that $K$, $K_0$ depend on $C>0$ and $\lambda_{SH}>1$.

\part{Robust minimality via s-transversality}\label{part1}


Throughout this part, $f$ is a $C^{1+}$ diffeomorphism of a Riemannian manifold $M$
and $\Lambda$ is a partially hyperbolic, $f$-invariant set
with a splitting
\[T_\Lambda M = E^{uu}\oplus E^c\oplus E^s;\quad \hbox{dim}(E^c)=1.\]
We assume that $\Lambda$ is a $uu$-lamination, i.e. it is saturated by $\W^{uu}$ leaves.

This part of the paper is devoted to a proof of Theorem~\ref{t=s-transPartHyp}.  The general proof uses several technical arguments to deal with the possibilities:
\begin{itemize}
 \item[--] there might not exist center-stable and center-unstable manifolds that foliate a neighborhood of $\Lambda$, and 
 \item[--] the center bundle $E^c$ might not be uniformly expanded by $Df$.  
 \end{itemize}
 Without these technicalities, the argument is relatively simple, and so we present it first in a restricted setting.

\section{A sketched proof of Theorem~\ref{t=s-transPartHyp} in a simplified setting}\label{s=simplified proof}
We  sketch a proof of Theorem~\ref{t=s-transPartHyp} under these simplifying hypotheses:
\begin{enumerate}
\item $f$ is a $C^{1+}$ transitive Anosov diffeomorphism, \label{i=trans Anosov}
\item $f$ is partially hyperbolic and dynamically coherent,  \label{i=PHandDC}
\item the center bundle is expanding, $1$-dimensional and has an orientation preserved by $Df$. \label{i=1dexpandingcenter}
\end{enumerate}
\medskip

\noindent{\bf Theorem.} {\em  If the above hypotheses \eqref{i=trans Anosov}-\eqref{i=1dexpandingcenter}  and  s-transversality of $\W^{uu}$ hold, 
 then $\cW^{uu}$ is a minimal foliation.}
\smallskip

The  assumptions \eqref{i=trans Anosov} and \eqref{i=1dexpandingcenter} mean that there is a globally-defined  $Df$-invariant splitting  $TM= E^{uu}  \oplus E^c\oplus  E^s$, where $E^c$ is uniformly expanded, orientable and $1$-dimensional; it follows that  the center-unstable bundle $E^u:= E^{uu}\oplus E^c$ is uniquely integrable, tangent to the unstable foliation $\W^u$ of the Anosov diffeomorphism $f$.  The dynamical coherence assumption \eqref{i=PHandDC} means that the 
bundle $E^s\oplus E^c$ is also  integrable, tangent to an $f$-invariant foliation $\W^{cs}$; by intersecting leaves of the foliations $\W^{cs}$ and $\W^u$, one obtains in addition an $f$-invariant {\em center foliation} $\W^c$.  

These hypothesis hold, for example, if $f$ is the perturbation of a hyperbolic toral automorphism that is partially hyperbolic with $1$-dimensional expanding center. They also hold for any $f\in\cA^r(\TT^3)$, $r> 1$, by \cite{Hammerlindl, Potrie}.

\subsection{Preparation}
The leaves of  $\W^{cs}$ are jointly subfoliated by the leaves of $\W^s$ and $\W^c$; the leaves of $\W^{u}$ are jointly subfoliated by the leaves of  $\W^{uu}$ and $\W^c$.  Hence there are well-defined $\W^{uu}$ and  $\W^{s}$ holonomies bet\-ween center leaves.
We will use the following facts about these holonomy maps.
 \begin{itemize}
 \item[--] For any $x\in M$ and any $x'\in \W^{uu}_{loc}(x)$, the  (restricted) $\W^{uu}$-holonomy $h^{uu}\colon \W^{c}_{loc}(x) \to  \W^{c}_{loc}(x')$ is $C^1$, with derivative varying continuously in the choice of $x,x'$.
\item[--] For any $x\in M$ and any $x'\in \W^{s}_{loc}(x)$, the   (restricted)   $\W^{s}$-holonomy $h^{s}\colon \W^{c}_{loc}(x) \to  \W^{c}_{loc}(x')$ is $C^1$, with derivative varying continuously in the choice of $x,x'$.
\end{itemize}
The first item is a standard fact about expanding strong unstable subfoliations and does not use the fact that the dimension of $E^c$ is $1$. The second item uses  a center-bunching property (see \cite{BW}), which follows from the $1$-dimensionality of the bundle $E^c$.
\medskip

We recall the notation $a\asymp_\Delta b$ introduced in Section~\ref{ss.notation}.
We then say that $a$ and $b$ are \emph{$\Delta$-nearly equal}, or just \emph{nearly equal}, if $\Delta$ is understood to be a constant that may be taken arbitrarily close to $1$ in the context.

We fix an orientation on $E^c$, and using this orientation we define the {\em signed center distance} $\bar d^c$ between two points on the same $\W^c$ plaque: for $x\in M$ and $x'\in \W^c_{loc}(x)$, the signed distance $\bar d^c(x,x')$ is just distance between $x$ and $x'$ along $\W^c_{loc}$ multiplied by $\pm 1$ according to whether  the geodesic arc  in $\W^c$   from $x$ to $x'$  is positively oriented or not.

Then we have the following standard distortion estimate on the derivative $Df\vert_{E^c}$  (compare with Corollary~\ref{c.distortion} in Section~\ref{ss.distortion}).
 \begin{lemma}\label{l=baby distortion}
 For every  $\delta>0$, there is $\eps_\delta>0$ such that for every $x_0\in M$, $x_1, x_2\in \W^c_{loc}(x_0)$, and $k\geq 1$,
if $d(f^k(x_i),f^k(x_j))<\eps_\delta$ for all $i,j$, then
\[\frac{ \bar d^c(x_1,x_0)}{\bar d^c(x_2,x_0)}\asymp_{1+\delta} \frac{\bar d^{c}(f^k(x_1),f^k(x_0))}{\bar d^c (f^k(x_2),f^k(x_0))}.\]
\end{lemma}

When we say that a set of $N$ points lying in a $1$-dimensional $\W^c_{loc}$ plaque is {\em $\Delta$-nearly evenly spaced}, we mean that $d^c$-distances between any two neighboring pairs of points are $\Delta$-nearly equal.

By the lemma, given $\Delta>1$ there is $\widehat \Delta$ such that if $x_0, x_1, x_2$ are positively ordered and $\widehat \Delta$-nearly evenly spaced apart along a $\W^c$ plaque, then the iterates $f^k(x_0), f^k(x_1), f^k(x_2)$ will remain positively ordered and ($\Delta$-)nearly evenly spaced as long as the distances between them remain small ($<\eps_\delta$).

\subsection{An inductive property}
Fix $\tau>0$ small so that the $\W^s$-holonomies (denoted by $h^s$)
between $\W^{cu}_{loc}$ plaques at $\tau$-close points are well-defined.

Let us assume that $\cW^{uu}$ is s-transverse. There are $R,\chi>0$ such that for every  $x\in M$ there exist curves $\gamma,\gamma'\colon [0,1]\to \W^{uu}_R(x)$ satisfying:
\begin{enumerate}
\item[(a)] for all $t\in[0,1]$, $\gamma(t)\in \W^{cs}_\tau (\gamma'(t))$, and
\item[(b)]  $\bar d^c (h^{s}(\gamma'(0)), \gamma(0))< -\chi$ and $\bar d^c (h^{s}(\gamma'(1)), \gamma(1))>\chi$.
\end{enumerate}
Let $L>1$ be a constant depending only on $R$ that bounds the norm of $Dh^{uu}$ and $Dh^s$, where $h^{uu}$ and $h^s$ are holonomy maps  between $\W^c$ plaques along paths of length smaller than $2R$.

Fix an $f$-invariant $uu$-lamination $\Lambda$.
We argue inductively (compare with Proposition~\ref{p=evenspaced} in Section~\ref{ss=cucriterion}):

\medskip

\noindent {\bf Inductive statement $P(N)$.} {\em
For every $\Delta>1$ and $\eps\in (0,\eps_\Delta)$, there exists a set of $N$ points which is $\Delta$-nearly evenly-spaced,
which lie in a single plaque  $\W^c_{loc}(x)\cap \Lambda$, and whose diameter (measured along $\W^c_{loc}(x)$) belongs to $[\eps/\|Df\|, \eps]$.}

\medskip

The case where $N=1$ is trivial. We now suppose $P(N)$ holds for some $N$ and prove $P(N+1)$.
We are given a constant $\Delta>1$ and a scale $\eps>0$.

We fix some $1<\widehat \Delta_1, \widehat \Delta_2$ such that $\log \widehat \Delta_1\ll \log \widehat \Delta_2 \ll \log \Delta$.
Applying  $P(N)$ to some $\widehat \Delta_1$ and to some $\widehat \eps\ll \eps$ small,  we  obtain
$x_1, \ldots, x_{N}\in \Lambda$,  a collection of $\widehat\Delta_1$-nearly evenly spaced points on $\W^{c}_{loc}(x_N)$ with the maximum distance $d(x_1,x_{N}) \in [\widehat \eps/\|Df\|, \widehat \eps)$. The $\W^{uu}$ holonomy $h^{uu}$ between
$\W^c_{loc}(x_{N})$ and $\W^c(y)$, for any $y\in \W^{uu}_R(x_{N})\subset \Lambda$ is $C^1$, uniformly in $R$.  Thus for  $\rho>0$ sufficiently small, the derivative of the  holonomy between  $\W^c_\rho(x_{N})$ and $\W^c(y)$ is nearly constant,  which means that $\widehat \Delta_1$-nearly evenly-spaced points are sent to $\widehat \Delta_2$-nearly evenly-spaced points.  We assume that $\widehat \eps <\rho$.

The s-transversality property gives curves $\gamma,\gamma'$ in $\W^{uu}_R(x_{N})\subset \Lambda$ satisfying the properties (a) and (b) above, with the given values of $R,\tau, \chi$.  
Let  $ x_1'(t), \ldots, x_{N}'(t) = \gamma'(t) \in \Lambda$ be the images of  $x_1, \ldots, x_N$ under the holonomies $h^{uu}\colon \W^{c}_{\widehat \eps}(x_{N})\to  \W^{c}_{loc}(\gamma'(t))$.  These points vary continuously with $t$ while they remain very nearly evenly spaced, since $\eps_\ell <\rho$.  The distance between any successive points $x_i'(t),  x_{i+1}'(t)$ is at most $L \widehat \eps$.

Define a function $\phi\colon [0,1]\to\RR$ by
\[\phi(t) = \bar d^c(h^{s}_t(\gamma'(t)), \gamma(t) ),
\]
where $h^s_t\colon \cW^c_{loc}(\gamma'(t))  \to \cW^c_{loc}(\gamma(t))$ is the local $\W^s$ holonomy.
This function $\phi$ is continuous, and s-transversality implies that
 \[\phi(0)<-\chi< 0<\chi<\phi(1).\] 
 We now choose a special value of $t_0\in[0,1]$.  There are two cases:
\begin{itemize}
\item[--] For $N=1$, we select $t_0\in[0,1]$ such that $\phi(t_0)= \chi/2$.
\item[--] For $N\geq 2$,  since $0< d^{c}( x_{N} '(t),  x_{N-1}' (t)) < L\widehat \eps < \chi$, it follows that
 there exists $t_0\in [0,1]$ such that 
\[\bar d^c(h^{s}_{t_0}(\gamma'(t_0)), \gamma(t_0) ) = \bar d^{c}(x_{N}' (t_0), x_{N-1}' (t_0)).
\]
\end{itemize}

\begin{figure}[h]
\begin{center}
\includegraphics[scale=.22]{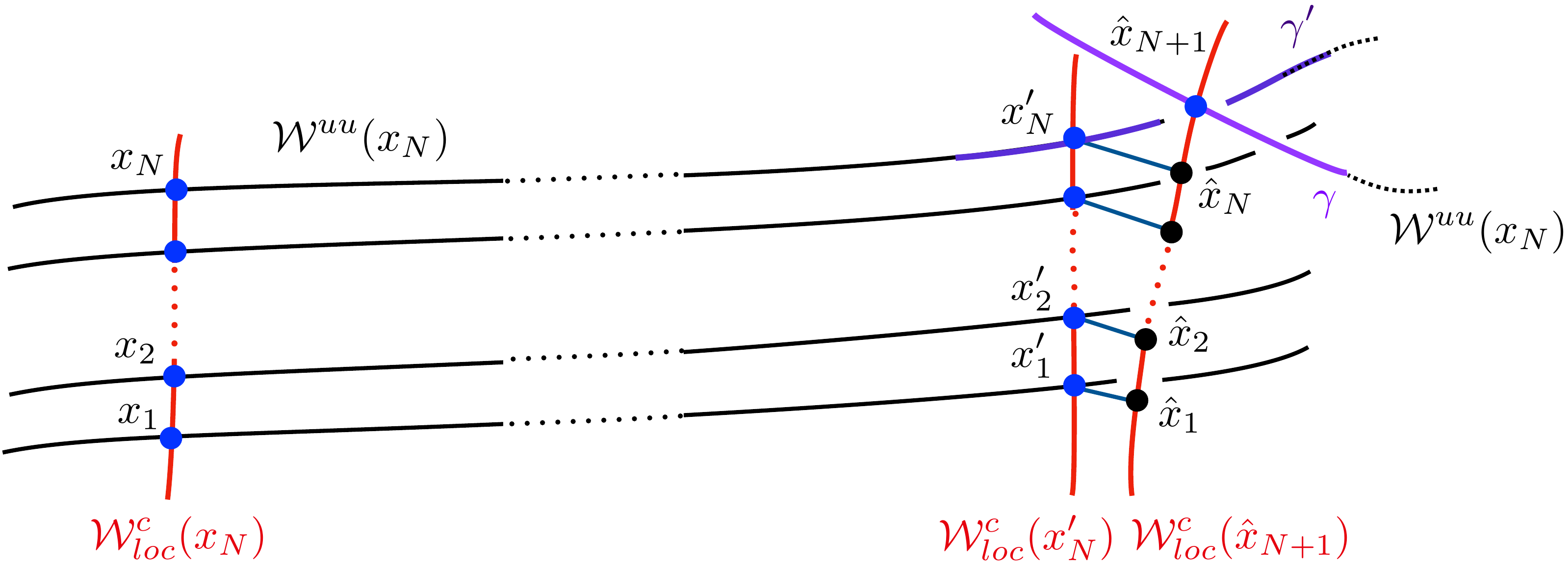}
\caption{Choice of the points $\hat x_1,\ldots, \hat x_{N+1}$. Points belonging to   $\Lambda_\ell(x)$ are colored in blue.}
\label{f=simplifiedproof}
\end{center}
\end{figure}

We set  $\hat  x_{N+1} = \gamma(t_0)\subset \Lambda$.  We also set
$\hat x_{i} =  h^s_{t_0}( x_i' (t_0))$, for $i=1,\ldots, N$.
Then $\hat x_{1}, \ldots, \hat x_{N+1}$ lie on $\W^{c}_{loc}(\hat  x_{N+1})$, are very nearly evenly spaced, with distance between these points bounded above by $L^2\widehat \eps$.  The point  $\hat x_{N+1}$ belongs to $\Lambda$, and  $\hat x_{1}, \ldots, \hat x_{N}$ belong to $\W^s_\tau (x_{i}'(t_0)) \subset \W^s_\tau(\Lambda)$.   See Figure~\ref{f=simplifiedproof}. If we then iterate the  points $\hat x_{1}, \ldots, \hat x_{N+1}$   forward by $f^k$, the distances between successive points  become larger but, since  $f$ is $C^{1+}$, the points remain $\Delta$-nearly evenly spaced as long as they remain within some fixed scale $\eps_\delta>0$ (as given by Lemma~\ref{l=baby distortion}). At the same time, the distance between $f^k(\hat x_{N+1})$ and $f^k(x_{N+1}')\in \Lambda$ goes to zero exponentially fast.
We take $k$ such that
$d(f^k(\widehat x_1), f^k(\widehat x_{N+1})) \in [\tfrac{\eps}{\|Df\|}, \eps]$.

We perform this construction for a sequence $\widehat \eps_\ell \to 0$
and get points \[ x_{\ell,1}:= f^{k_\ell}(\hat x_{\ell,1}), \ldots, x_{\ell,N+1}:= f^{k_\ell}(  \hat x_{\ell,N+1})\; \in \Lambda\] $\Delta$-nearly evenly spaced in $\W^c_{loc}(x_{\ell,N})$, with $d(x_{\ell,1}, x_{\ell,N+1}) \in [\tfrac{\eps}{\|Df\|}, \eps]$ and  $d( x_{\ell,N+1},\Lambda)\to 0$ as $\ell\to \infty$.  Extracting a subsequence (by compactness of $\Lambda$), we obtain the desired limiting points $x_1 \ldots, x_{N+1}\in \Lambda$, proving $P(N+1)$, and completing the induction.

\subsection{Conclusion of the proof}
We now fix $\Delta>1$ and $\eps>0$  sufficiently small. For each $N\geq 1$,  Property $P(N)$ gives an ordered sequence of points
\[x_{1}(N), x_{2}(N), \ldots, x_{N}(N)\in \Lambda\] $\Delta$-nearly evenly spaced  in $\W^c_{loc}(x_{N}(N))$ with
$d^c(x_{1}(N),x_{N}(N)) \in [\tfrac{\eps}{\|Df\|}, \eps)$.  Sending $N\to \infty$ and extracting a Hausdorff convergent subsequence of $\{x_{1}(N), \ldots, x_{N}(N)\}$,  we obtain a center disk $\W^c_{\eps'}(x)\subset \Lambda$, where $\eps'\in \bigl[\eps/\|Df\|, \eps\bigr]$. Since $\Lambda$ is $\W^{uu}$-saturated and $\W^u$ is bifoliated by $\W^c$ and $\W^{uu}$, the union of $\W^{uu}_{loc}$ plaques through points in  $\W^c_{\eps'}(x)$ gives a $\W^u$ disk in $\Lambda$.
Corollary~\ref{c=s-transPartHyp} then concludes that $\W^{uu}$ is minimal.
This completes the outline of the proof of Theorem~\ref{t=s-transPartHyp} under  the simplifying assumptions.

\section{Properties of foliated plaques and the details on  s-transversality}
\label{s.preliminaries-part2}
In this section, we consider a $C^1$ diffeomorphism $f$ preserving a partially hyperbolic $uu$-lamination $\Lambda$ with $1$-dimensional center. We develop some properties of foliated plaque families for the proof of   Theorem~\ref{t=s-transPartHyp}.  We also fill in the details about s-transversality (Propositions~\ref{p.s-transverse} and ~\ref{p=stransversecriterion}).

\subsection{Center and stable distances}\label{ss.plaque}

Let $d^{uu}$ be the induced Riemannian metric on $\W^{uu}$ leaves of $\Lambda$.  For $x\in \Lambda$ let $d^c_x, d^s_x$ be the induced Riemannian metrics on $\P^c_x$, and on the leaves of the local foliation  $\hW^s_x$, respectively.
Proposition~\ref{p.plaque} implies the following lemma (see Figure~\ref{f=pathinplaques}).
For  $x,y\in \Lambda$ close, this lemma gives a way to project $y$ onto $\cW^{uu}_{loc}(x)$ by intersecting with the plaque $\cP^{cs}(y)$.

In order to simplify the presentation, one will assume that the bundles $E^{uu}$, $E^c$ and $E^s$ are nearly orthogonal.
To handle the general case, one should replace explicit numbers in the estimates that follow by a constant that
only depends on the angles between these bundles.

\begin{lemma}\label{l.plaque}
If $\eps_0$ is small, then
for every $x\in \Lambda$ and $y\in B(x,\tfrac{\eps_0}{2})\cap\Lambda$, there is a unique 
path  $\sigma = \sigma_1\cdot\sigma_2\cdot \sigma_3$, $\sigma_i\colon [0,1]\to M$, from $y$ to $x$ with  $\sigma_1$ geodesic in $\cP^c(y)$, $\sigma_2$ geodesic in $\hW^{s}_y(\sigma_1(1))$, and $\sigma_3$ geodesic in $\W^{uu}_{loc}(x)$.

This path varies continuously in $(x,y)$.
\end{lemma}We set
\[\rho^c(y,x)  = d^{c}_y(y, \sigma_1(1)) = \hbox{Length}(\sigma_1),\]
\[\rho^s(y,x)  = d^{s}_{y} (\sigma_1(1),\sigma_3(0)) = \hbox{Length}(\sigma_2).\] 
While $\rho^c,\rho^s$ are not distances (they are not even symmetric), the maps $\rho^c$ and $\rho^s$ are continuous on $(B(z_0,\tfrac{\eps_0}{4})\cap \Lambda)^2$, for any $z_0\in\Lambda$.

\begin{figure}[h]
\begin{center}
\includegraphics[scale=.21]{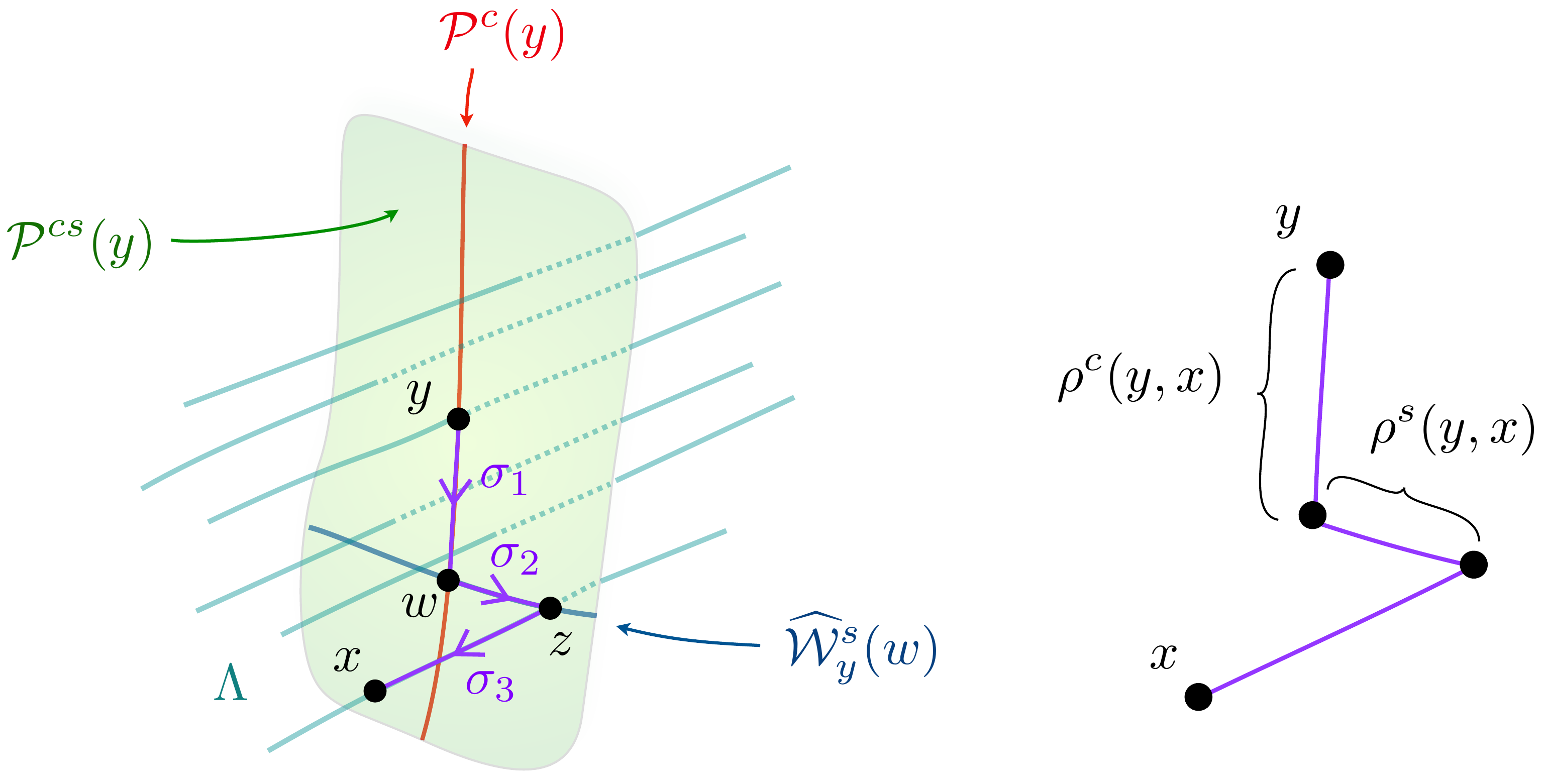}
\caption{The paths in Lemma~\ref{l.plaque}, which are used to define the ``distances"  $\rho^c(y,x), \rho^s(y,x)$. }
\label{f=pathinplaques}
\end{center}
\end{figure}

The next lemma shows that
given $x,y\in \Lambda$ close, $y$ is the projection (onto  $\cW^{uu}_{loc}(y)$)
of some point $x'\in \cW^{uu}_{loc}(x)$. Hence the union of center-stable plaques centered at points of the local manifolds
$\cW^{uu}_{loc}(x)$ is a uniform neighborhood of $x$.
\begin{lemma}\label{l.cover}
If $\eps_0$ is small enough,
for any $x\in \Lambda$ and $y\in B(x,\tfrac{\eps_0}{10})$, there exists $x'\in \W^{uu}_{loc}(x)$
such that $y\in \cP^{cs}(x)$.
In particular, there is $z\in \cP^c(x')$ such that
$y\in \widehat \W^s_{x'}(z)$.
\end{lemma}
\begin{proof}
The bundle $E^{cs}:= E^c\oplus E^s$ is locally trivial,
hence the spaces $E^{cs}(x')$ for $x'$ close to $x$
can be identified with $E^{cs}(x)$.
For $\rho>0$ small, we construct a continuous map $H$
from the product of balls $B^{uu}(x,\rho)\times B^{cs}(x,\rho)\subset T_{x}M$ to a
neighborhood of $x$ as follows. For any $(\zeta,\xi)$ in the product, let $x'=\exp_{x}(\zeta)$
and identify  $\xi\in E^{cs}(x)$ with some $\xi'\in E^{cs}(x')$. The point $H(\zeta,\xi)$ is then defined to be the image of $\xi'$ by the exponential map at $x'$ along the submanifold $\cP^{cs}(x')$.
The image of $H$ is thus contained in the union of the plaques $\cP^{cs}(x')$
for $x'\in \W^{uu}_{loc}(x)$.

Note that when $(\zeta,\xi)$ is close to $(0,0)$,
the distance $d(H(\zeta,\xi),\exp_{x}(\zeta+\xi))$ is smaller than $\|\zeta+\xi\|/100$.
In particular if $\eps_0$ is small enough,
the restriction of $H$ to the boundary of $B^{uu}(x,\eps_0)\times B^{cs}(x,\eps_0)$
around any point in $B(x,\tfrac{\eps_0}{10})$ has index $1$. This implies that
$B(x,\tfrac{\eps_0}{10})$ is contained in the image of $H$.
\end{proof}

The center-stable plaques $\cP^{cs}$ are in general not coherent (i.e. they don't foliate a neighborhood of $\Lambda$)  but the next lemma gives a level of control over their noncoherence.
\begin{lemma}\label{l=alternate 5.4}
If $\eps_0$ is small enough, then the following holds.
Consider any $x\in \Lambda$ and $x'\in \W^{uu}_{loc}(x)\setminus\{x\}$
such that $\cP^{cs}(x),\cP^{cs}(x')$ intersect at some point $y\in B(x,\eps_0)$,
and let $z\in \cP^c(x)\cap \widehat \cW^s_x(y)$, $z'\in \cP^c(x')\cap \widehat \cW^s_{x'}(y)$.
Then $$d(z,z')< \tfrac 1{10}\min (d(x,z), d(x',z')).$$
\end{lemma}
\begin{proof}
We choose $0<\eps_0<\eps_1$ small.
Working in exponential charts,
all the bundles are close to constant bundles. Hence it is enough to show:
$$d(x,x')\ll \min (d(x,z), d(x',z')).$$
This is proved as follows:

First iterate  the points $x,x',y$ forward under $f$ while $d(x,y)$ remains smaller than $\eps_1$.
By local invariance of the plaque families,
the geometry of the configuration remains bounded and controls the distance $d(f^n(x),f^n(x'))$.

The distance between $x$ and $x'$ is uniformly expanded; hence there is a last iterate $N$
satisfying our assumptions on the configuration. At that time
$d(f^N(x),f^N(z))$ and $d(f^N(x'),f^N(z'))$ are larger than a fixed number
and $d(f^N(x),f^N(x'))$ is smaller than a fixed number.
Hence there exists $C>0$ such that
$$d(f^N(x),f^N(x'))\leq C \min (d(f^N(x),f^N(z)), d(f^N(x'),f^N(z'))).$$

Now iterate backward to return to the initial configuration.
The integer $N$ can be assumed arbitrarily large if $\eps_1/\eps_0$ is large.
By domination of the splitting $E^{uu}\oplus E^{c}$, the ratio
$d(x,x')/d(f^N(x),f^N(x'))$ is much smaller than
$\min (d(x,z), d(x',z'))/\min (d(f^N(x),f^N(z)), d(f^N(x'),f^N(z')))$.
This gives the required estimate.
\end{proof}

\subsection{Center orientation}\label{ss.center-orientation}
We first prove that the neighborhood of any point $x_0\in \Lambda$ is separated by the brush
\[Br(x_0):= \bigcup_{y\in \W^{uu}_{loc} (x_0)}\W^{s}_{loc}(y)\]
into two components. This will allow us to define two sides of $x_0$, locally.

For any $0<\eps\ll \eps_0$
and $x_0\in \Lambda$ we introduce the set
\[U_{\eps}(x_0):=B(x_0,\eps)\setminus Br(x_0).\]
As defined in Section~\ref{ss.brush}, two points $y,y'\in U_{\eps}(x_0)$
are said to be {\em in the same component} if they can be joined by a continuous path $\gamma_{y,y'}$
in $U_{\eps_0/C_0}(x_0)$, where $C_0>1$ is a quantity that only depends on the angle between the bundles $E^{uu}, E^c, E^s$.

\begin{lemma}\label{l=twocomponents}
If $\eps_0>0$ is small enough, then the set $U_{\eps}(x_0)$
has two components in the above sense, denoted by $U^\pm_{\eps}(x_0)$.
Moreover one can choose $\gamma_{y,y'}$ to lie in $B(x_0,C_0\eps)$.
Hence $U^\pm_{\eps}(x_0)$ do not depend on $\eps_0$.
\end{lemma}
Note that $U_{\eps}(x_0),U^\pm_{\eps}(x_0)$ are defined independently of  the plaque families $\cP^c,\cP^{cs}$, and $\hW^s$.
\begin{proof}
For the proof and as in Section~\ref{ss.plaque} we will assume that the bundles $E^{uu}$, $E^c$ and $E^s$ are nearly orthogonal.
This allows us  to work with explicit numbers rather than with a constant $C_0$.

Using the local coordinates of the exponential map at $x_0$,
we construct  a $C^1$ disc $D$ with dimension $\dim(E^{uu})+1$ that contains $x_0$ and is separated by $\W^{uu}_{loc}(x_0)$, such that for any $y,x\in B(x_0,\eps_0/100)$ with $x\in\Lambda$ and $y\in \cP^{cs}(x)$,
the leaf $\widehat \W^s_x(y)$ intersects $D$ in a unique point.

\begin{figure}[h]
\begin{center}
\includegraphics[scale=.2]{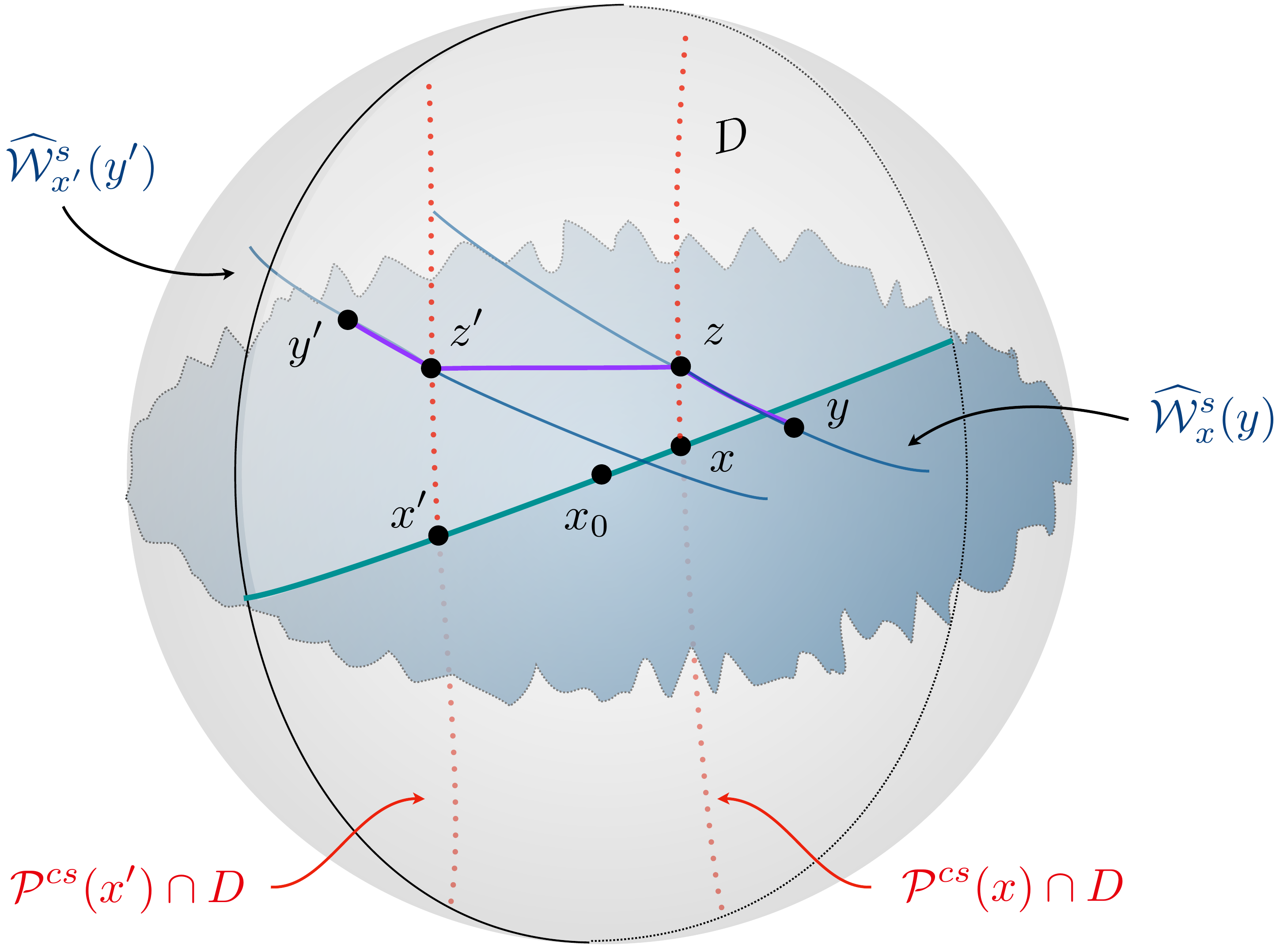}
\caption{First part of the proof of Lemma~\ref{l=twocomponents}.}
\label{f=twocomponents}
\end{center}
\end{figure}

By Lemma~\ref{l.cover}, any point $y\in U_{\eps}(x_0)$
belongs to $\cP^{cs}(x)$ for some $x\in W^{uu}_{loc}(x_0)$.
The leaf $\widehat \W^s_x(y)$ intersects $D$ at a point $z$
but does not meet $Br(x_0)$ by Lemma~\ref{l=alternate 5.4}.
Since $\eps$ is small and and since $y,z$ belong
to $B(x_0,\eps)$, the plaques are close to linear subspaces in the exponential chart at $x_0$. Hence there exists a path $\gamma_y\subset \widehat \W^s_x(y)$
that connects $y$ to $z$ in $U_{5\eps}(x_0)$.

To any other point $y'\in U_{\eps}(x_0)$  we associate in a similar way an arc
$\gamma_{y'}\subset U_{5\eps}(x_0)$ that connects $y'$ to some point $z'\in D$.
The set $D\setminus \W^{uu}_{loc}(x_0)$
has two connected components. If $z,z'$ belong to the same component,
they can be connected by an arc $\gamma_{z,z'}\subset U_{10\eps}(x_0)$ and $y,y'$ can be connected by an arc which is the concatenation of
$\gamma_y,\gamma_{z,z'},\gamma_{y'}$ and which is contained in $U_{10\eps}(x_0)$.
This shows that $U_{\eps}(x_0)$ has at most two components.

Now consider two points $y_1,y_2$ which belong to different connected components of
$D\setminus \W^{uu}_{loc}(x_0)$. We claim that they belong to different components
of $U_{\eps}(x_0)$; this will conclude the proof.
Assume by contradiction that there exists an arc $\gamma_{y_1,y_2}$
connecting  these two points and  contained in $U_{\eps_0/10}(x_0)$.
By Lemma~\ref{l.cover}, any point $y\in \gamma_{y_1,y_2}$ can be projected to $D$:
it belongs to some leaf $\widehat \W^s_x(y)$ with $x\in \W^{uu}_{loc}(x_0)$
intersecting $D\setminus \W^{uu}_{loc}(x_0)$ at some point $z$.
Moreover,  any point in a neighborhood of $y$ can also be projected to $D$~:
the projection is a priori not uniquely defined, but belongs to a connected neighborhood of $z$
in $D$, which by  by Lemma~\ref{l=alternate 5.4} is disjoint from $\W^{uu}_{loc}(x_0)$.
Using these projections, we thus define a continuous arc  inside 
$D\setminus \W^{uu}_{loc}(x_0)$ that connects $y_1$ to $y_2$,
contradicting the choice of the points.
\end{proof}

The next lemma proves that along any path $\gamma$ in a leaf $\W^{uu}(x_0)$,
one can continuously follow the two components $U_\eps^+(x),U_\eps^-(x)$.
This will allow us to compare the components of points that are not close.

\begin{lemma}\label{l.lift}
Consider a leaf $\W^{uu}(x_0)$ with $x_0\in \Lambda$
and an arc $\gamma\colon [0,1]\subset \W^{uu}(x_0)$.
For any component $U^\sigma_\eps(\gamma(0))$,
there exists a path $\overline \gamma$ satisfying
$\overline \gamma(0)\in U^\sigma_\eps(\gamma(0))$ and
$\overline \gamma(t)\in U_\eps(\gamma(t))$ for all $t\in [0,1]$.
The component of $U_\eps(\gamma(1))$ that contains
$\overline \gamma(1)$ does not depend on $\overline \gamma$.
\end{lemma}
Two points $z_0\in U_\eps(\gamma(0))$ and
$z_1\in U_\eps(\gamma(1))$ will be said \emph{on the same side} of $\W^{uu}(x)$
relative to $\gamma$ if there exists an arc $\overline \gamma$ as in Lemma~\ref{l.lift}
joining the components of $U_\eps(\gamma(0)), U_\eps(\gamma(1))$
and containing $z_0,z_1$ respectively.
\begin{proof}
Choose $z\in \cP^c_\eps(\gamma(0))$ in the component $U^\sigma_\eps(\gamma(0))$
and let $\delta\in (0,\eps)$ be the distance between $\gamma(0)$ and $z$ inside $\cP^c_\eps(\gamma(0))$.
Since the plaque family $\{\cP^c(x)\}$ is continuous,
one can continuously choose $\overline \gamma(t)$ in
$\cP^c_\eps(\gamma(t))$ at a distance $\delta$ from $\gamma(t)$
inside $\cP^c_\eps(\gamma(t))$.
By Lemma~\ref{l=alternate 5.4}, $\overline \gamma(t)$ belongs to $U_\eps(\gamma(t))$ for any $t$.

Consider another arc $\overline \gamma'$.
Since the components $U_\eps^+(x),U_\eps^-(x)$ are open in $U_\eps(x)$,
the set of parameters $t$ such that $\overline \gamma(t)$
and $\overline \gamma'(t)$ belong to the same (resp. different) component
of $U_\eps(\gamma(t))$ is open. By connectedness,
$\overline \gamma(t)$
and $\overline \gamma'(t)$ belong to the same component of $U_\eps(\gamma(t))$ for all $t$.
\end{proof}

\subsection{Robustness, minimality and s-transversality}\label{ss=cleanup}

We first establish the robustness of the s-transversality property stated in Proposition~\ref{p.s-transverse0}.
\begin{proposition}\label{p.s-transverse}
Let $\Lambda$ be an invariant s-transverse $uu$-lamination for $f$. For $\tau>0$ sufficiently small,
there exist  $R,\chi>0$, a neighborhood $U$ of $\Lambda$ in $M$ and a $C^1$ neighborhood $\cU$ of $f$, such that for any $g\in \cU$,
for any $g$-invariant $uu$-lamination $\Lambda_g\subset U$, and for any $x\in \Lambda_g$,
\begin{itemize}
\item[--] $\Lambda_g$ is s-transverse,
\item[--] Items (1)--(2) of Definition~\ref{d.transverse} are satisfied at scale $\tau$ by arcs
$\gamma,\gamma'\subset W_R^{uu}(x)$; moreover $\rho^c(\gamma(0),\gamma'(0))>\chi$ and $\rho^c(\gamma(1),\gamma'(1))>\chi$.
\end{itemize}
\end{proposition}

\begin{proof} Suppose $\Lambda$ is s-transverse, and let $\tau>0$ be given by  Definition~\ref{d.transverse}.  Given $x\in \Lambda$,   there exist $R_x ,\chi_x >0$ satisfying the conditions of Definition~\ref{d.transverse}  at the scale $\tau$ by arcs
$\gamma,\gamma'\subset W^{uu}_{R_x}(x)$, with  $\rho^c(\gamma(0),\gamma'(0))>\chi_x$ and
$\rho^c(\gamma(1),\gamma'(1))>\chi_x$.
The continuity of the plaque families and the constructions in the previous subsection imply that for $x'$ sufficiently close to $x$, one can choose $R_{x'}=R_x, \chi_{x'} = \chi_x$.  Compactness of $\Lambda$ implies that $\chi$ and $\rho$ can be chosen uniformly.
This property persists under $C^1$-small perturbations (in both the set $\Lambda$ and in $f$).
\end{proof}

We next justify Remark~\ref{r.single tau}.
\begin{lemma}\label{l=single tau}
If $\Lambda$ satisfies Items (1)--(2) of Definition~\ref{d.transverse} for a given $\tau$ small, then it is s-transverse.
\end{lemma}
\begin{proof}
Suppose then that $\gamma, \gamma'\subset \cW^{uu}(x)$ satisfy Definition~\ref{d.transverse} at scale $\tau$.  
 Lemma~\ref{l.plaque} implies that these paths can be continuously parametrized so that each point  $\gamma'(t)$ lies in $\cP^{cs}(\gamma(t))$, as in Figure~\ref{f=gammagammaprime}. Denote by $z(t)$ the unique point in $\cP^c(\gamma(t))$ such that $\gamma'(t) \in \widehat\W^s_{\gamma(t)}(z(t))$, and  consider the continuous function $\psi(t):= \bar d^c(z(t),\gamma(t))$, where $\bar d^c$ is the signed distance in  $\cP^c(\gamma(t))$ with respect to some fixed orientation of $E^c$ along $\gamma([0,1])$.

\begin{figure}[h]
\begin{center}
\includegraphics[scale=.22]{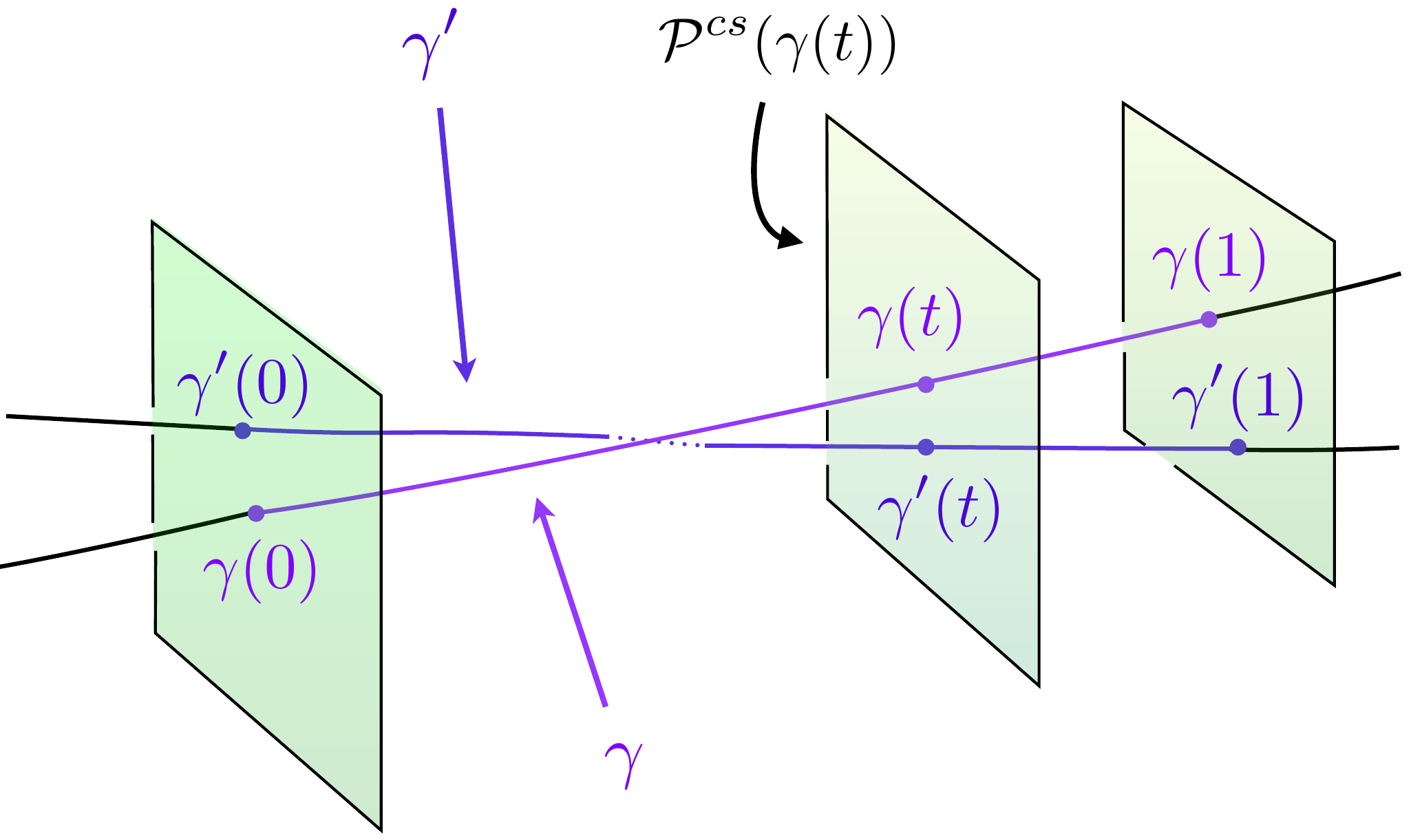}
\caption{Reparametrizing $\gamma'$.}
\label{f=gammagammaprime}
\end{center}
\end{figure}

Given any $\tau' \in (0,\tau)$, we fix $\eta>0$ small.
Since $\gamma'(0)$ and $\gamma'(1)$ lie on different sides of $\W^{uu}(x)$ relative to $\gamma$, for any $\eta>0$ there exists a closed interval $I=[a,b]\subset (0,1)$ such that
$\psi(a),\psi(b)\neq 0$ have opposite signs and $|\psi(t)|<\eta$ on $[a,b]$.
If $\eta$ has been chosen small enough, then  there exists $n\geq 1$ large such that
$\rho^c(f^n(\gamma(t)), f^n(\gamma'(t)))\leq \tau'/3$ for all $t\in [a,b]$,  and since the local stable manifolds are uniformly contracted
$\rho^s(f^n(\gamma(t)), f^n(\gamma'(t)))\leq \tau'/3$. This gives $d(f^n(\gamma(t)), f^n(\gamma'(t)))\leq \tau'$,
whereas $f^n(\gamma'(a))$ and $f^n(\gamma'(b))$ are on different sides of $\cW^{uu}(f^n(x))$ relative to $f^n(\gamma)$.
Hence Items (1)--(2) of Definition~\ref{d.transverse} are satisfied at scale $\tau'$ by arcs $f^n(\gamma),f^n(\gamma')$ at the point $f^n(x)$.
Since the point $x$ is chosen arbitrarily in the invariant set $\Lambda$ and since $n$ can be chosen independently from $x$,
s-transversality holds at scale $\tau'$.
 \end{proof}

We next prove Proposition~\ref{p=structure-minimal} on the structure of dynamically minimal $uu$-laminations.
\begin{proof}[Proof of Proposition~\ref{p=structure-minimal}]
The previous proof shows that if $\Lambda$ is s-transverse, then there is $R>0$ such that for any
$x\in \Lambda$ there exist two different points $y,z\in \cW^{uu}_{R}(x)$ such that
$y\in \cW^{s}_{loc}(z)$. By continuity of the lamination $\cW^{uu}$, if $x,x'$ are close enough, then there exist
$y\in \cW^{uu}_{R}(x)$, $z\in \cW^{uu}_{R}(x)$ such that $y\in \cW^{s}_{loc}(z)$.
Consequently there exists $\eta>0$ such that if $\Omega,\Omega'$ are two uu-sublaminations at distance less than $\eta$ from each other
then they intersect a common local stable leaf.

Fix $N_0\geq \operatorname{Diam}(M)/\eta$. We claim that any (non necessarily invariant) minimal $uu$-lamination $\Omega\subset \Lambda$ is fixed by some iterate $f^N$, $N\geq 1$.
Indeed for any $\ell\geq 1$, two iterates $f^{-i}(\Omega),\dots, f^{-j}(\Omega)$ with $\ell\leq i<j<\ell+N_0$ share a common local stable manifold.
Iterating by $f^j$, we obtain that $\Omega$ and $f^{j-i}(\Omega)$ have two points whose distance  apart  is exponentially small in $\ell$; moreover $1\leq j-i\leq N_0$.
There exists $N\geq 1$ and a sequence $\ell_k\to \infty$ such that the associated iterates satisfy $j_k-i_k=N$.
Passing to the limit, it follows that $\Omega$ and $f^N(\Omega)$ intersect,  and hence coincide, since they are minimal $uu$-laminations.

Let $\Omega$ be a minimal $uu$-sublamination inside $\Lambda$.
We have shown that it is periodic; denote by $N$ its smallest period.
Since the iterates $f^n(\Omega)$ for $n=0,\dots,N_1$ are distinct minimal $uu$-laminations, they are disjoint.
Hence $\Omega\sqcup f(\Omega) \sqcup\cdots \sqcup f^{N-1}(\Omega)$ is an invariant $uu$-lamination.
Since $\Lambda$ is dynamically minimal it coincides with this disjoint union.

Each set $f^n(\Omega)$ is connected since it is the closure of a leaf of $\cW^{uu}$.
Consequently it is a connected component of the disjoint union $\Lambda = \Omega\sqcup f(\Omega) \sqcup\cdots \sqcup f^{N-1}(\Omega)$.
\end{proof}

The next lemma characterizes the integrability of $E^{uu}\oplus E^s$ in the partially hyperbolic setting.

\begin{lemma}\label{l.non-joint-integrability}
Let $f$ be a diffeomorphism with a partially hyperbolic splitting
$TM=E^{uu}\oplus E^c\oplus E^s$, $\dim(E^c)=1$. The following properties are equivalent:
\begin{enumerate}
\item there exists an invariant foliation tangent to $E^{uu}\oplus E^s$;
\item there exists $\eta>0$ such that for any $x\in M$, $y\in Br(x)\cap B(x,\eta)$,
the local unstable manifold $\W^{uu}_\eta(y)$ is contained in $Br(x)$.
\end{enumerate}
\end{lemma}
\begin{proof}
(1)$\Rightarrow$(2) is clear. In the following we assume (2) and prove (1).

Fix a point $x_0$ and consider a chart  in which the vertical is close to $E^c$ and transverse to $E^{uu}\oplus E^s$.
In such a  chart, each stable and unstable manifold is a graph over the horizontal plane.

By assumption, for any $x$ and any $y \in Br(x)\cap B(x,\eta)$ we have $\W^{uu}_\eps(y)\subset Br(x)$, and hence $Br(y)\cap B(y,\eps)\subset Br(x)$.
Symmetrically,  $Br(x)\cap B(x,\eps)\subset Br(y)$.

We claim that if $y,z$ in $Br(x_0)$ are close to $x_0$ and project to the same point in the  horizontal plane, then $y=z$.
Indeed, since $z$ belongs to $Br(x_0)\cap B(x_0,\eta)$, it must be  is contained in $Br(y)$,
and so it belongs to a plaque $\W^s_{loc}(t)$ with $t\in\W^{uu}_{loc}(y)$.
Since $\W^s$, $\W^{uu}$ and the vertical are tangent to transverse cones,
the projection of $z$ to the horizontal belongs to the projection of $\W^{uu}_{loc}(y)$ only
if $z$ itself belongs to $\W^{uu}_{loc}(y)$. This implies the claim.

We have proved that $Br(x_0)\cap B(x_0,\eta)$ is a graph over the horizontal.
By invariance of  domain,  its projection to the horizontal is open.
Any points $z,y\in Br(x_0)\cap B(x_0,\eta)$ close to each other can be connected by a small path tangent to $E^{uu}\oplus E^s$.
By uniform continuity of $E^{uu}\oplus E^s$, this shows that the graph $Br(x_0)\cap B(x_0,\eta)$ is $C^1$ and tangent to $E^{uu}\oplus E^s$.

By Proposition 1.6 and Remark 1.10 in~\cite{Bonatti-Wilkinson},
this implies that there exists a continuous foliation with $C^1$-leaves tangent to $E^{uu}\oplus E^s$, proving (1).
\end{proof}

Finally we turn to the proof of Proposition~\ref{p=stransversecriterion}, a criterion for checking s-transversality.

\begin{proof}[Proof of Proposition~\ref{p=stransversecriterion}]
Fix  $\tau>0$ small.
By Lemma~\ref{l.non-joint-integrability} if $E^{uu}\oplus E^s$ is not integrable, then for any $0<\eta\ll \tau$
there are points $x$, $y\in Br(x)\cap B(x,\eta)$ such that
$\cW^{uu}(y)\cap B(x,\eta)$ is not contained in $Br(x)$. 
Thus there exists an arc  $\gamma'_0\subset \cW^{uu}_{2\eta}(y)$ connecting $y$ to some point
$z\in U_\eps(x)$.

We claim that close to $y$ there exists a point $y'$ belonging to the  component of $U_\eps(x)$ that does not contain $z$.
This can be proved by examining the end of the proof of Lemma~\ref{l=twocomponents}:
 consider again the disc $D$ through $x$ of dimension $\dim(E^{uu})+1$.  
 The local manifold  $\cW^{s}_{loc}(y)$ meets $D$ at a unique point, which belongs to $\cW^{uu}_{loc}(x)$.
Since the projection on $D$ by holonomy along $\cW^s_{loc}$-leaves is open on a neighborhood of $y$,
there exists $y'$ close to $y$ such that $\cW^{s}_{loc}(y)$ meets $D$ at a point belonging to any given component $U^\pm_\eps(x)$.
Let  $\gamma'\subset \cW^{uu}_{loc}(y')$ be an arc connecting  $y'$ to $z$ of  length smaller than $2\eta$ that is 
 $\tau$-close to $x$.  Define  $\gamma$ to be the arc supported on $\{x\}$.
Definition~\ref{d.transverse} is now satisfied at scale $\tau$. Since $\gamma,\gamma'$ have diameter smaller than $\tau$,
the foliation $\W^{uu}$ is locally s-transverse.
\end{proof}

\section{Preparation for the proof of  Theorem~\ref{t=s-transPartHyp}}\label{s.proof-minimality}
This section and the next are devoted to the  proof of Theorem~\ref{t=s-transPartHyp}. 
We thus fix a $C^{1+}$ diffeomorphism $f\colon M\to M$ and an $f$-invariant  $uu$-lamination $\Lambda\subset M$ that is $s$-transverse
and satisfies the property SH. We aim to show that  $\Lambda$ contains a $cu$-disk.  We continue to use the notation introduced in \S~\ref{s=notations}-\ref{s=simplified proof}-\ref{s.preliminaries-part2}.
In this section, we introduce the necessary tools to address the general case where $\Lambda$ is not necessarily equal to $M$ and $E^c$ is not necessarily expanding.

\subsection{Plaque holonomies and distortion}
\subsubsection{Signed center distance} In Section~\ref{s=simplified proof}, we introduced signed versions $\bar d^c$ and $\bar \rho^c$ of $d^c$ and $\rho^c$ assuming that the center bundle is orientable. When it is not orientable, signed distances can be introduced locally and can be extended continuously in the neighborhood of a simple arc.
There are two choices for the sign locally, but the ratio of two signed distances does not depend on this choice and is is thus well-defined.

\subsubsection{Basic distortion}\label{ss.distortion}
Item (5) of Proposition~\ref{p.plaque} implies that the tangent bundles of the center plaques $\cP^c(z)$ are $\theta$-H\"older continuous, for some $\theta\in(0,1)$. Let $C_\theta>0$ be a uniform bound on their H\"older constants.
For $x\in \Lambda$, $\eps>0$, $k\geq 0$, let 
\[ c_k(w,\eps) := \exp\left(C_\theta \sum_{i=0}^{k-1} \min \{\eps_0,  \text{\rm Length}(f^i(\P^c(w,\eps)))\}^\theta  \right).\]
We have the standard distortion estimates.
\begin{lemma}\label{l.basicdistort}
If the lengths of the first $k$ iterates of $\cP^c(x,\eps)$ are less than $\eps_0$,  then for all $y,z\in\P^c(x,\eps)$, we have
\[{\|D_yf^k|_{\cP^c(x,\eps)}\|}\asymp_{{c_k}(x,\eps)}{\|D_zf^k|_{\cP^c(x,\eps)}\|}.\]
\end{lemma}

\begin{corollary}\label{c.distortion} 
Let $K\subset \Lambda$ be a hyperbolic set such that $E^c|_K$ is uniformly expanded.
For any $\Delta>1$, there is $\eps_1 = \eps_1(\Delta)>0$ such that for any $z\in K$, $x_0,x_1,x_2\in \P^c(z)$ and any $k\geq 1$,
if $d(f^i(z),f^i(x_j))<\eps_1$ for all $j$ and $0\leq i\leq k$,  then
\[\frac{ \bar d^c_z(x_1,x_0)}{\bar d^c_z(x_2,x_0)}\asymp_\Delta \frac{\bar d^{c}_{f^k(z)}(f^k(x_1),f^k(x_0))}{\bar d^c_{f^k(z)}(f^k(x_2),f^k(x_0))}.\]
\end{corollary}

\subsubsection{Basic distance controls}
We will prove a sequence of lemmas controlling the  distortion  $c_k$ for very large $k$.  This will  require as a preliminary finding a scale at which the bundles in the partially hyperbolic splitting are nearly constant.  The next lemma addresses the choice of this scale.

\begin{lemma}\label{l.def-length}
For any $\delta,\eta>0$, there exists $\eps_2 = \eps_2(\delta,\eta)>0$ such that  for any points $z,z'\in \Lambda$ with $z'\in \W^{uu}_{loc}(z,\eps_2)$,  the following holds.

\begin{enumerate}
\item For any $k\geq 0$, if  $d^{uu}(f^i(z), f^{i}(z'))<\eps_2$ for $i=0,\ldots, k-1$, then \[d^{uu}(f^k(z), f^k(z'))\geq (1+\delta) \lambda_k(z) d^{uu}(z,z').\]

\item Consider any $\eps\in (0,\eps_2)$, and $x,y\in \P^c(z,\eps)$, $x',y'\in \P^c(z')$ such that
there are points $x''\in \hW_z^s(x,\eps_2)\cap \Lambda$, $y''\in \hW_{z}^s(y,\eps_2)\cap \Lambda$,
$x'''\in \hW_{z'}^s(x',\eps_2)\cap \W^{uu}_{loc}(x'')$, and $y'''\in \hW_{z'}^s(y',\eps_2)\cap \W^{uu}_{loc}(y'')$ (see Figure~\ref{f=deltascalecomp}).
If $d^c_{z'}(x',y') \geq \eta \eps$
and $\eps>\max\{d^s_{x''}(x'',x), d^{uu}(z,z')\}$, then:
 $$d^c_z(x,y)\asymp_{1+\delta} d^c_{z'}(x',y').$$

 \end{enumerate}
\end{lemma}

\begin{figure}[h]
\begin{center}
\includegraphics[scale=.22]{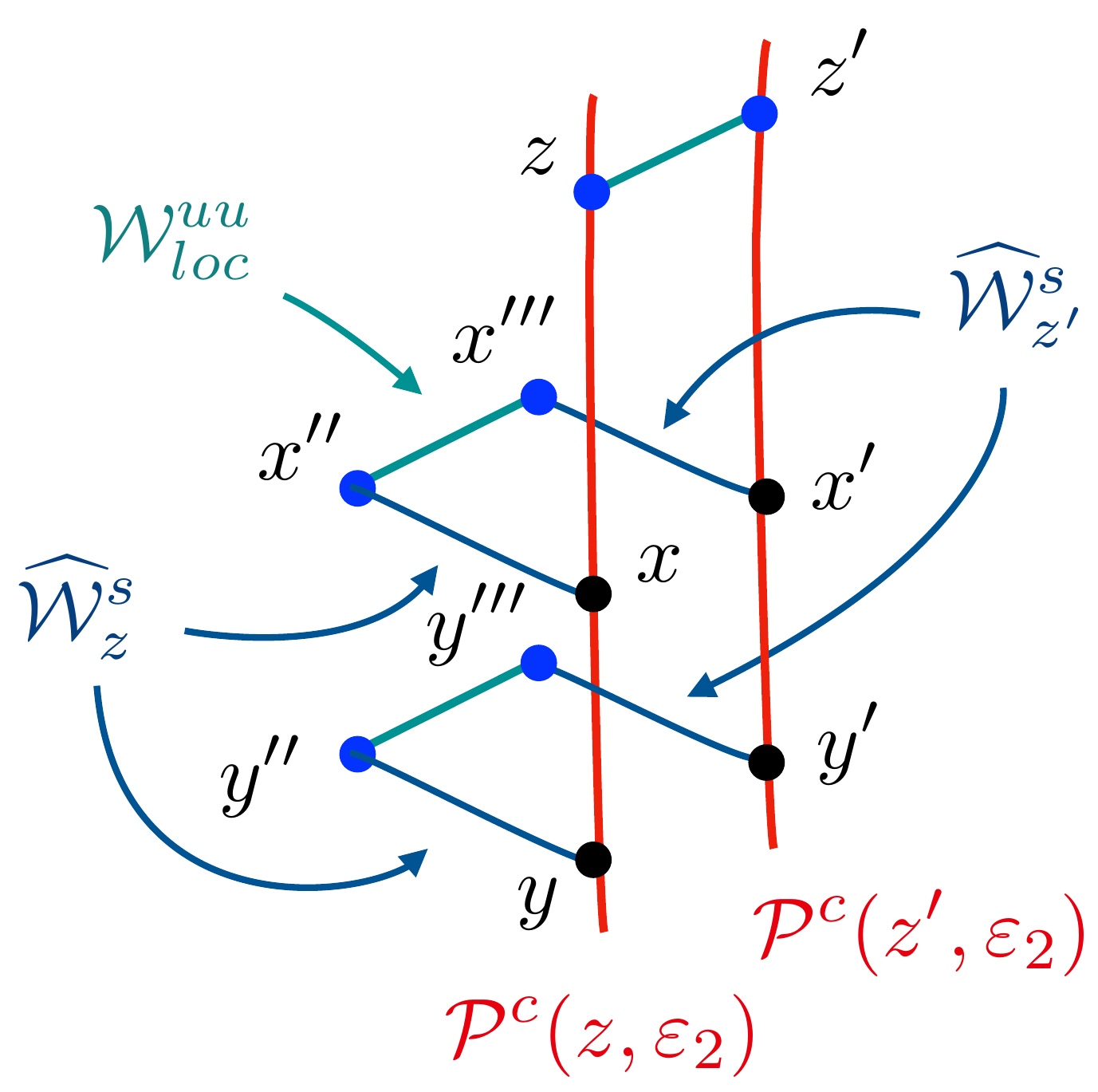}
\caption{Arrangement of points in Lemma~\ref{l.def-length}. Points in $\Lambda$ are colored blue.}
\label{f=deltascalecomp}
\end{center}
\end{figure}

\begin{proof}
Part (1) uses an elementary hyperbolic distortion estimate.
At small scale the bundles are almost constant.
In particular, one may consider local coordinates such that the center bundle is
arbitrarily close to the vertical direction.
By our assumptions, all stable and unstable distances are controlled
from above by the center distance $d^c_z(x,y)$.
This implies that moving $x$ to $x'$ along the ``suus-path'' through $x''$ and $x'''$, the difference of the vertical coordinates of $x$ and $x'$ is small with respect to
$d^c_z(x,y)$. The same property holds when one moves $y$ to $y'$.
Consequently the ratio between
$d^c_z(x,y)$ and $d^c_{z'}(x',y')$ is arbitrarily close to $1$ if the scale is chosen small enough.
\end{proof}

\subsubsection{$\W^{uu}$-holonomy between center plaques}\label{ss.uu-holonomy}
We need to introduce a type of $\W^{uu}$-holonomy between center plaques centered at points inside a common $\W^{uu}$-leaf.
Since the system is not dynamically coherent, we have to add a projection along a (fake) stable foliation, as in Figure~\ref{f=Lipholonomy}.

Fixing $R>0$,  if $\eps>0$ is sufficiently small, then for each 
 $x\in \Lambda$ and $x'\in \W^{uu}(x,R)$, we define as follows the $\W^{uu}$-holonomy
\[h^{uu}\colon \P^c(x, \eps)\cap\Lambda  \to \P^c(x').\]
First, there exists $\delta>0$ such that for $\eps>0$ sufficiently small and for any  $x\in \Lambda$,  $x'\in \W^{uu}(x,R)$, and $y\in  \P^c(x,\eps)\cap\Lambda$,  the intersection  $D'(y):=\W^{uu}(y,R)\cap B(x',\delta)$ is  connected and nonempty.  We fix any $x''(y)\in D'(y)$ and define $y''$ to be the unique point of intersection between $\W^{uu}_{loc}(x''(y))$ and   $\P^{cs}(x')$.
Then $h^{uu}(y)$ is the unique point in $\cP^c(x')$ such that
$\W^{ss}_{loc}(h^{uu}(y))$ contains $y''$. This does not depend on the choice of $x''(y)$.
 
It is well-known that for a dynamically coherent, partially hyperbolic diffeomorphism with $1$-dimensional center, unstable holonomies are Lipschitz  between center manifolds. The following is a refined  analogue of this statement in the context of $\W^{uu}$-laminated partially hyperbolic sets, expressing as well  the fact that the Lipschitz norm of the holonomy is nearly constant on small scales.

\begin{figure}[h]
\begin{center}
\includegraphics[scale=.23]{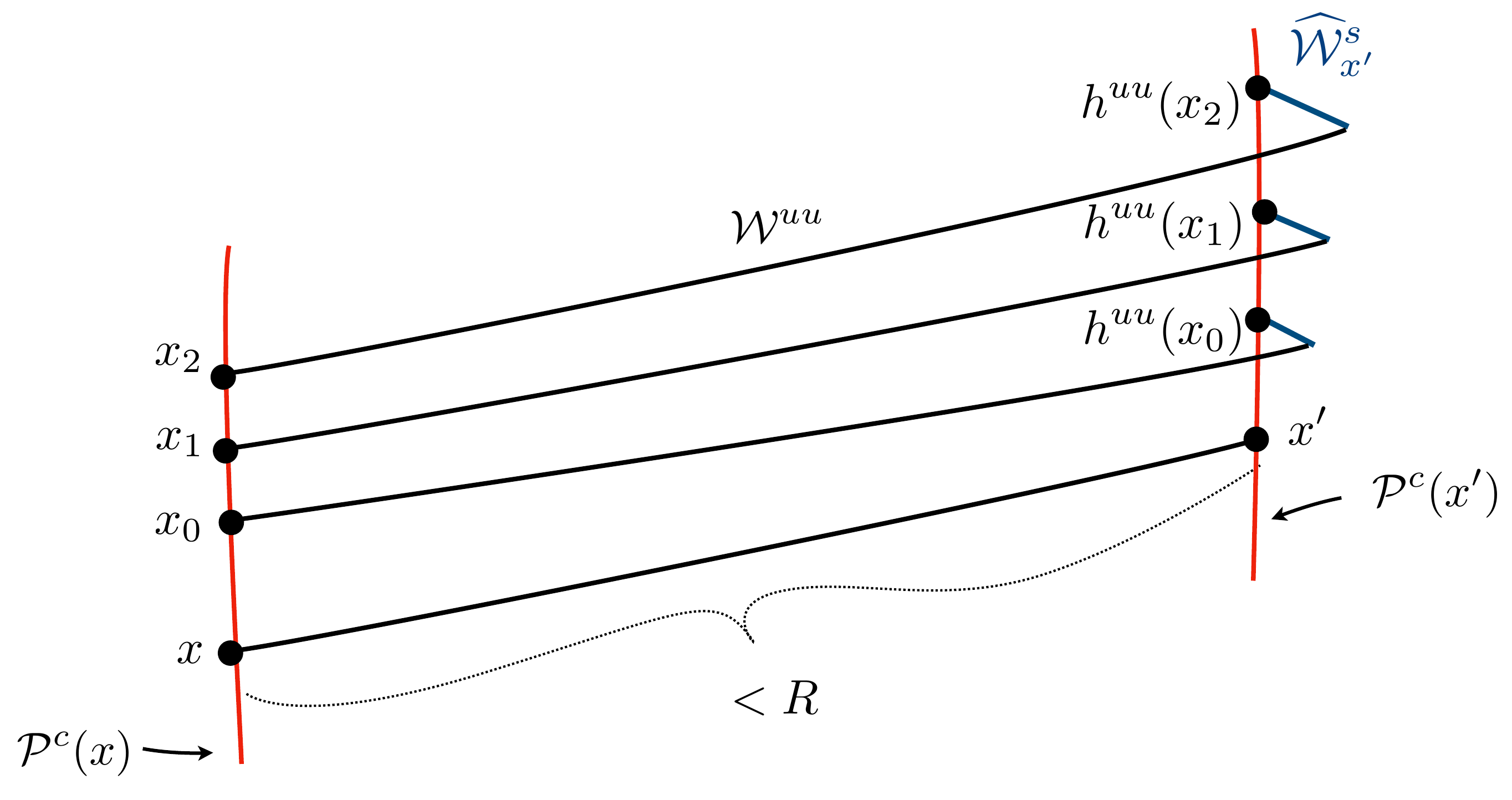}
\caption{$\W^{uu}$ holonomy is Lipschitz on suitable scales.}
\label{f=Lipholonomy}
\end{center}
\end{figure}

\begin{lemma}\label{l.holonomy}
Given any $R,\eta>0$ and $L>1$ there exists $\eps_3 = \eps_3({R,\eta,L})$
such that for any $\eps\in (0, \eps_3)$, any $x\in\Lambda $ and any $x'\in \W^{uu}(x,R)$, the holonomy  $h^{uu}\colon \P^c(x, \eps)\cap\Lambda  \to \P^c(x')$ has the following properties:
\begin{enumerate}
\item For any $x_0\in \P^c(x, \eps) \cap\Lambda $, if $x_0\neq x$ then $h^{uu}(x_0)\neq x'$.
\item For any $x_0, x_1, x_2 \in  \P^c(x, \eps) \cap\Lambda $,  if $d^c_{x}(x_i, x_j)>\eta\eps$ for $i\neq j$, then
\[ \frac{\bar d^c_{x} (x_1,x_0)}{\bar d^c_{x}  (x_2,x_0)} \asymp_{L}  \frac{\bar d^c_{x'} (h^{uu}(x_1),h^{uu}(x_0))}{\bar d^c_{x'}  (h^{uu}(x_2),h^{uu}(x_0))}.\]
\end{enumerate}
\end{lemma}

\begin{proof}
Given such a configuration of points with $x\neq x_0$, iterate by $f^{-n}$ so that the strong unstable distances become shorter than the center distance between $f^{-n}(x)$ and $f^{-n}(x_0)$.
Lemma~\ref{l.def-length} implies that $f^{-n}(x')\neq f^{-n}(x_0)$. This gives Item (1).

Let us now fix $1<L'<L^{1/3}$ and assume that $\eps_3$ is small enough.
The distortion of $f^{-n}$ along $\cP^c(x)$ is controlled by Lemma~\ref{c.distortion},
and by the lengths of the arcs $f^{-k}(\cP^c(x,\eps_3))$ for $0\leq k\leq n$:
having chosen $\eps_3$ small enough, $c_{-k}(x,\eps_3)$
is close to $1$ for $k\geq 0$ smaller than some fixed large integer $N$.
For $N\leq k\leq n$, the curve $f^{-k}(\cP^c(x,\eps_3))$
has  length smaller than the strong unstable distance between $f^{-k}(x)$
and $f^{-k}(x')$ (by definition of $n$), which is exponentially small with respect to $k$.
This bounds $c_{-n}(x,\eps_3)$ and gives:
\[ \frac{\bar d^c_{x} (x_1,x_0)}{\bar d^c_{x}  (x_3,x_0)} \asymp_{L'}
\frac{\bar d^c_{f^{-n}(x)} (f^{-n}(x_1),f^{-n}(x_0))}{\bar d^c_{f^{-n}(x)}  (f^{-n}(x_3),f^{-n}(x_0))}.\]
The local geometry ensures that the strong stable curves do not get too large.
We compare distances at these scales using Lemma~\ref{l.def-length}:
\[ \frac{\bar d^c_{f^{-n}(x')} (f^{-n}(h^{uu}(x_1)),f^{-n}(h^{uu}(x_0)))}{\bar d^c_{f^{-n}(x')}  (f^{-n}(h^{uu}(x_3)),f^{-n}(h^{uu}(x_0)))} \asymp_{L'}
\frac{\bar d^c_{f^{-n}(x)} (f^{-n}(x_1),f^{-n}(x_0))}{\bar d^c_{f^{-n}(x)}  (f^{-n}(x_3),f^{-n}(x_0))}.\]
Going back by $f^n$, controlling distortion along $\cP^c(f^{-n}(x'))$ as before, we obtain:
\[ \frac{\bar d^c_{f^{-n}(x')} (f^{-n}(h^{uu}(x_1)),f^{-n}(h^{uu}(x_0)))}{\bar d^c_{f^{-n}(x')}  (f^{-n}(h^{uu}(x_3)),f^{-n}(h^{uu}(x_0)))} \asymp_{L'}
\frac{\bar d^c_{x'} (h^{uu}(x_1),h^{uu}(x_0))}{\bar d^c_{x'}  (h^{uu}(x_3),h^{uu}(x_0))}.\]
This gives Item (2). 
\end{proof}

\begin{figure}[h]
\begin{center}
\includegraphics[scale=.22]{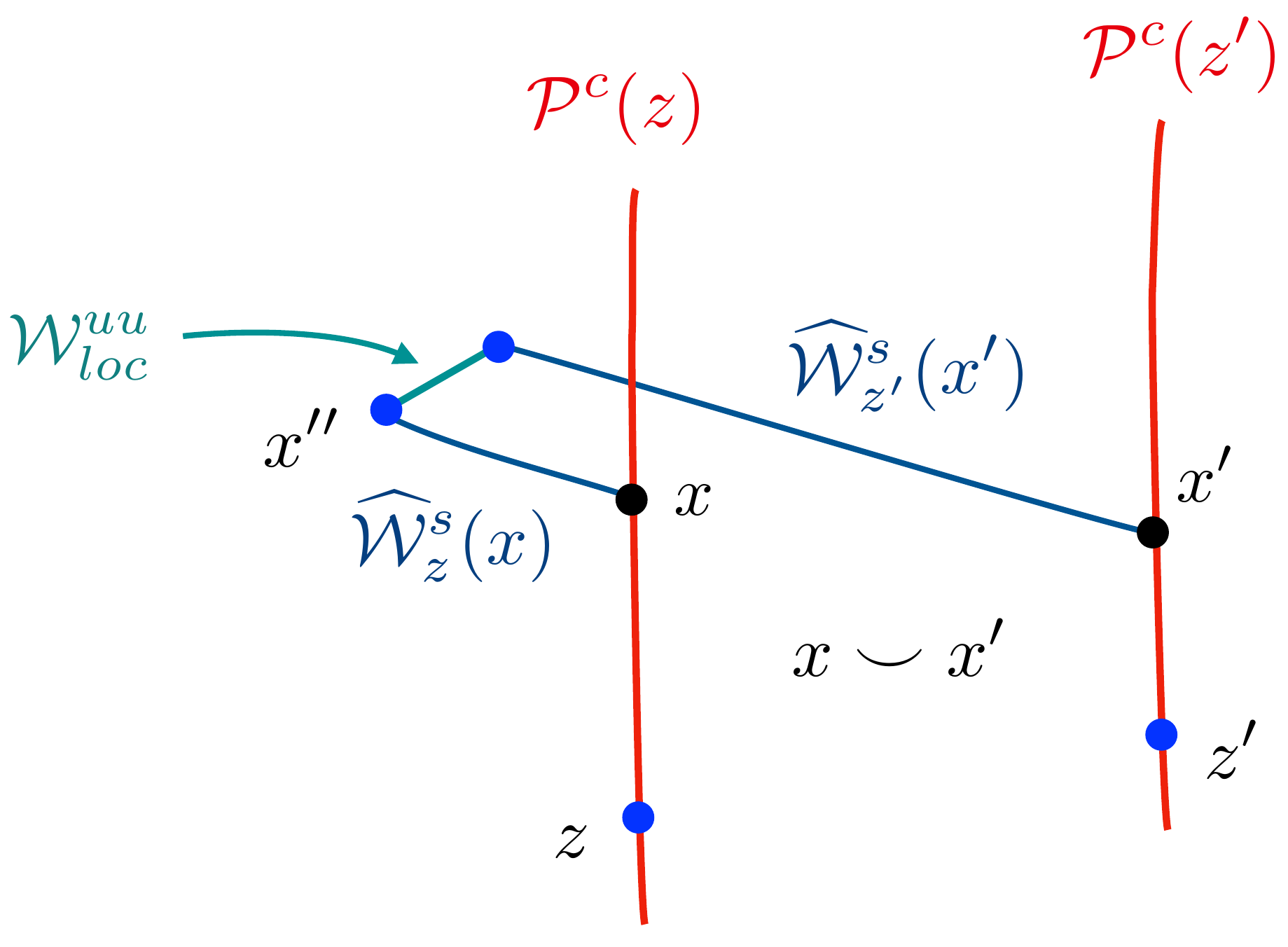}
\caption{The $suus$ relation. Points in $\Lambda$ are colored blue.}
\label{f=suus}
\end{center}
\end{figure}

\subsubsection{The $suus$ relation $\smile$}

For $z,z'\in \Lambda$, $x\in \cP^c(z)$ and   $x'\in \cP^c(z')$, we say that $x,x'$ are {\em  $suus$-related} and write $x\smile_{z,z'} x'$ (or just $x\smile x'$) if there exists $x''\in \hW^{s}_z(x)\cap \Lambda$
such that 
 $\hW^s_{z'}(x')\cap   \W^{uu}_{loc}(x'')\neq \emptyset$.
 See Figure~\ref{f=suus}. Observe that $\smile$ is symmetric.

The next lemma gives a similar comparison between distances of $suus$-related points lying on adjacent center plaques.  This identification describes something like a stable holonomy (for any  $z,w\in \Lambda$
and $x\in\P^c(z)$, there may exist several $x'\in\P^c(w)$, with $x\smile x'$; hence it is a relation between $\P^c(z)$ and $\P^c(w)$ rather than a function).

\begin{lemma}\label{l.sholonomy} Given $ L>1, \eta>0$, there is $\eps_4= \eps_4({\eta,L})>0$ such that for all $\eps\in (0,\eps_4)$, the following holds. 
Let $ z, x_0,x_1,x_2, x_0',x_1',x_2'$ such that $z\in \Lambda$, $x_0\in \Lambda\cap \P^{cs}(z)$,
$x_i \in \P^c(x_0,\eps)$, $x_i'\in \P^c(z,\eps)$ with $x_i\smile x_i'$ (see Figure~\ref{f=LocalsHolonomy}).
\begin{enumerate}
\item If $x_0\neq x_1$ then $x'_0\neq x'_1$.
\item If $d^c_{x_0}(x_i, x_j)>\eta\eps$ for $i\neq j$, then
$ \displaystyle \frac{\bar d^c_{x_0} (x_1,x_0)}{\bar d^c_{x_0}  (x_2,x_0)} \asymp_{L}  \frac{\bar d^c_{z} (x_1',x_0')}{\bar d^c_{z}  (x_2',x_0')}.$
\end{enumerate}
\end{lemma}

\begin{figure}[h]
\begin{center}
\includegraphics[scale=.22]{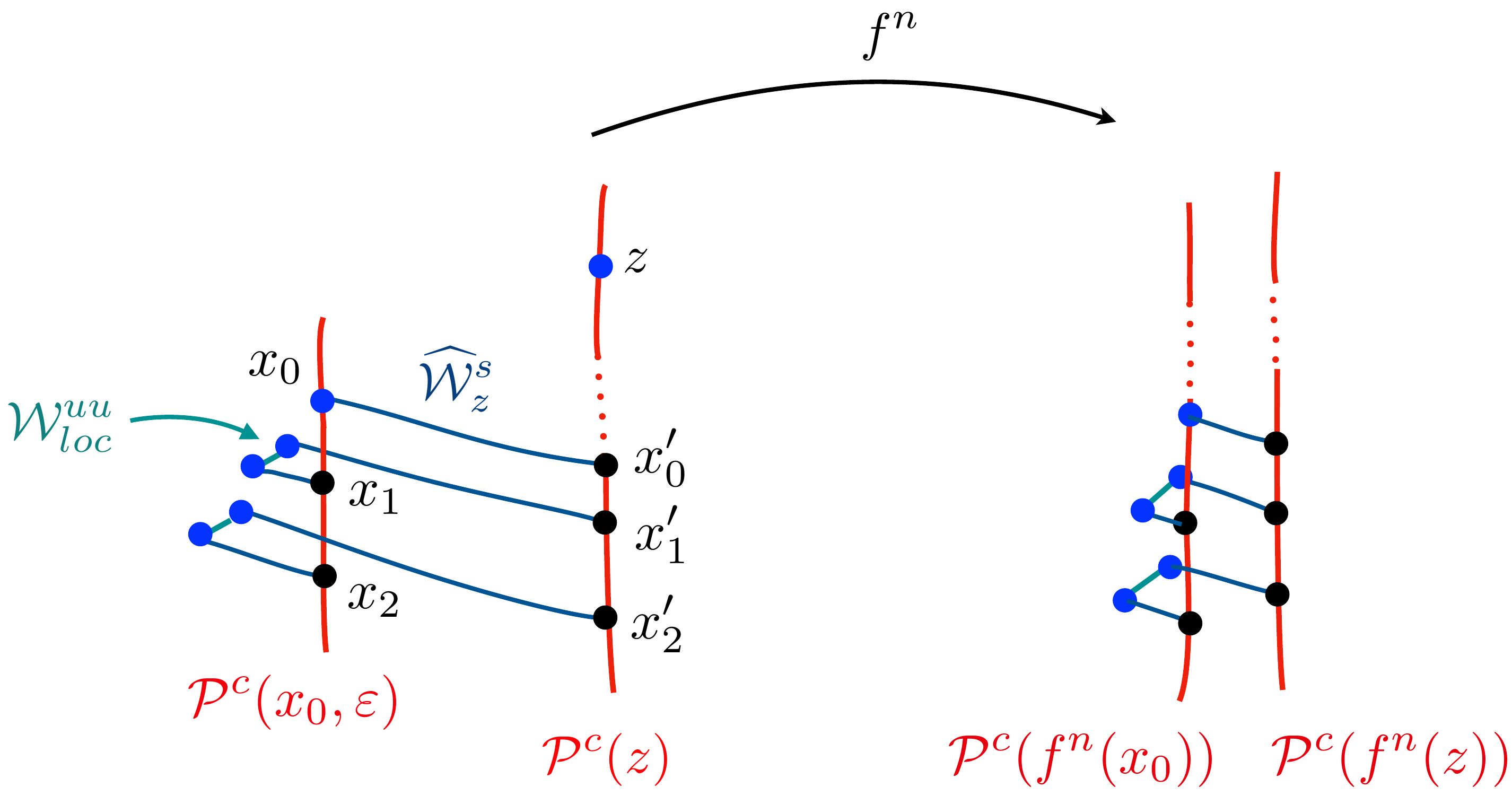}
\caption{Proof of Lemma~\ref{l.sholonomy}.  Points in $\Lambda$ are colored blue.}
\label{f=LocalsHolonomy}
\end{center}
\end{figure}

\begin{proof}[Proof of Lemma~\ref{l.sholonomy}]  The proof is very similar to the proof of Lemma~\ref{l.holonomy}. Given such a configuration of points,  we iterate by $f^{n}$ so that the stable distances become shorter than the center distance between $f^{n}(x)$ and $f^{n}(y)$, while this center distance remains small
(less than $\eps$) and the points remain in the center plaque of $f^{n}(w)$.  This is depicted in Figure~\ref{f=LocalsHolonomy}.  One then compares distances along the center using Lemma~\ref{l.def-length}  and iterates backwards, using distortion control.
\end{proof}

\subsection{A key proposition for the SH case}
The following proposition is fairly technical and is needed to handle the case where $\Lambda$ is not hyperbolic but merely satisfies the SH property. It gives a method to transfer configurations between points $x_i$ in a center plaque $\cP^c(z)$,
to another center plaque $\cP^c(z')$ with $z'\in \W^{uu}(z)$.
Note that the SH property is not used in its proof.

\begin{proposition}\label{p.waitingdisk} Given $\eps, R, \eta >0$ and $\Delta,L>1$, there exists \\$\eps_5 = \eps_5(\eps, R, \eta, \Delta,L)>0$ such that for all $\widetilde \eps\in (0,\eps_5)$, the following holds.

For any $z\in \Lambda$  there exists  $m\in \NN$ such that 
for any \[z'\in D:= f^{-m}\left( \W^{uu}(f^m(z),R)\right),\]
for  any $x_0',x_1',x_2'\in \P^c(z',\widetilde \eps)$ such that
$d^c_{z'}(x_i', x_j')> \eta\, \widetilde \eps$ when $i\neq j$,
and for any $x_0,x_1, x_2\in \P^c(z,\widetilde \eps)$ with  $x_i\smile x_i'$ for $i=0,1,2$,
 we have:
\begin{enumerate}
\item   $d^c_{f^k(z')}(f^k(z'),f^k(x'_i))<\eps$ for all $0\leq k\leq m$ and $i=0,1,2$;
\item  $\displaystyle \frac{\bar d^c_{z'} (x_1',x_0')}{\bar d^c_{z'}  (x_2',x_0')}\asymp_\Delta \frac{\bar d^c_{f^m(z')}(f^m(x_1'),f^m(x_0'))}{\bar d^c_{f^m(z')}(f^m(x_2'),f^m(x_0'))};$
\item  $\displaystyle \frac{\bar d^c_{z} (x_1,x_0)}{\bar d^c_{z}  (x_2,x_0)} \asymp_{L} \frac{\bar d^c_{z'} (x_1',x_0')}{\bar d^c_{z'}  (x_2',x_0')}$.

\end{enumerate}
\end{proposition}

\begin{figure}[h]
\begin{center}
\includegraphics[scale=.22]{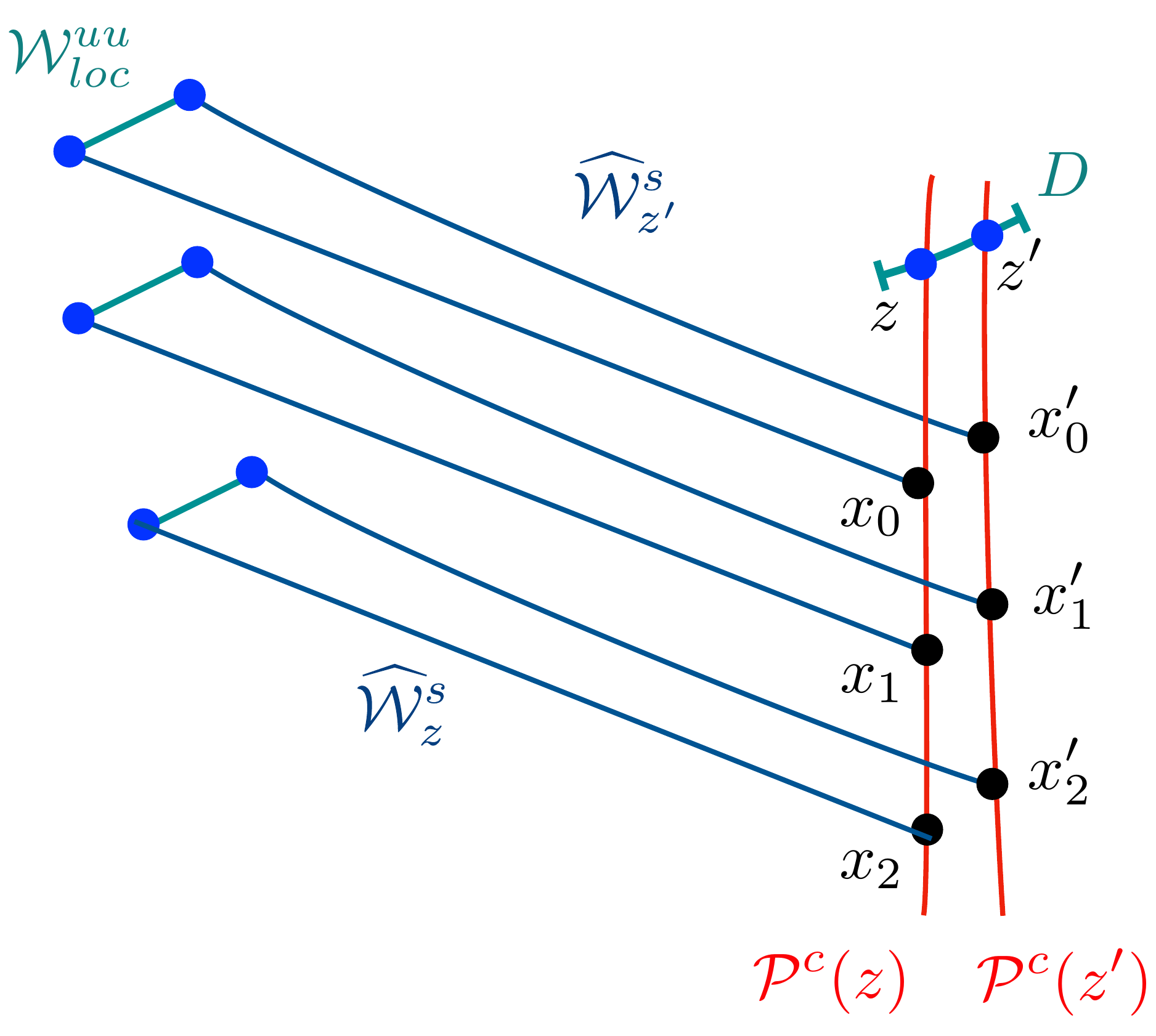}
\caption{Configuration in Conclusion (3) of Proposition~\ref{p.waitingdisk}.}
\label{f=waitingdisk}
\end{center}
\end{figure}

\begin{proof}
Recall the functions $\kappa<\mu<\lambda$ that were introduced in \S~\ref{ss.partial-hyperbolicity}, and let
$$\kappa_0=\sup \kappa,\;\; \lambda_0=\inf \lambda,\;\;  \mu_0=\sup \mu,\;\; \nu_0=\sup \tfrac \mu\lambda,$$
$$\beta \in (\max\{\lambda_0^{-1}, \nu_0, \kappa_0\},1).$$
Item (5) of Proposition~\ref{p.plaque} states that the tangent bundles of the center plaques $\cP^c(w)$ are $\theta$-H\"older continuous, for some $\theta\in(0,1)$; we denote by $C_\theta>0$ a H\"older constant.
The center bundle $E^c$ on $\Lambda$ is also H\"older continuous, with some exponent $\alpha>0$ and H\"older constant
$C_\alpha>0$.

The following choices of constants could be compressed, but to make things easier to read, we have chosen to leave the choices as closely aligned as possible with their application in the proof.  The proof is broken into 4 steps; we indicate where appropriate in which step each choice is used.

Let $\eps, \eta > 0$, $\Delta,L>1$ be given.  We choose $\widetilde\Delta>1$, $\widetilde\delta,\widetilde\eta\in (0,1)$,   satisfying: 
\begin{itemize}
\item[$\bullet$] $\widetilde\Delta^4(1+\widetilde\delta)^2\leq  \min\{L,\Delta\}$ (Steps 1 and 3),
\item[$\bullet$] $\widetilde \Delta^2 \exp\big(C_\alpha \widetilde \eta^\alpha/(1-\lambda_0^{-\alpha})\big)<2$ (Steps 1 and 2), and
\item[$\bullet$]  $\widetilde\eta< \eta /4$ (Step 1).
\end{itemize}
Let $\eps_2=\eps_2(\widetilde\delta, \widetilde \eta)$ be given by Lemma~\ref{l.def-length}, and $r_0 <\min\{ \eps_2,\widetilde\eta,1\}$ such that
\begin{itemize}
\item[$\bullet$]  $\exp(C_\theta(2\widetilde \eta^{-1}r_0)^\theta/(1-\beta^\theta))<\widetilde\Delta^{1/2}$ (Steps 1 and 2).
\end{itemize}
As in \S~\ref{ss.brushes}, the size of local manifolds is $\eps_0$.
Fix $\ell_0\geq 1$ such that
\begin{itemize}
\item[$\bullet$]  $\kappa_0^{\ell_0}\eps_0 < r_0$ (Step 0),
\end{itemize}
and  choose $\ell_3\geq 1$ such that for every $w\in\Lambda$, 
\begin{itemize}
\item[$\bullet$]  $f^{\ell_3}\big(\W^{uu}(w, \frac{r_0}{\lambda_0(1+\widetilde\delta)})\big)\supseteq \W^{uu}(f^{\ell_3}(w),R)$ (Step 3).
\end{itemize}
Now choose integers $N_1, N_2\geq 1$ such that 
\begin{itemize}
\item[$\bullet$] $\widetilde\eta^{-1} \max\{\nu_0,\lambda_0^{-1}\}^{N_2}<\beta^{N_2}$ (Step 2),
\item[$\bullet$] $5r_0 \mu_0^{\ell_3} \beta^{N_2} < \eps$ (Step 3),
\item[$\bullet$] $\exp(C_\theta\ell_3( 5 \mu_0^{\ell_3}\beta^{N_2})^{\theta}r_0) < \widetilde\Delta$ (Step 3),
\item[$\bullet$] $\beta^{N_1}\leq \left(\widetilde\eta \lambda_0 (1+\widetilde\delta)\|Df\|^{N_2}\right)^{-1}$ (Step 2), and
\item[$\bullet$]  $2\mu_0{\widetilde\eta}^{-1} \kappa_0^{n_1-1} <  \beta^{n_1}$ for any integer $n_1\geq N_1$ (Step 1).
\end{itemize}
Finally we choose $\eps_5 = \eps_5({\eps, R, \Delta,L,\eta})\in (0,1)$ satisfying
\begin{itemize}
\item[$\bullet$] $\eps_5 \mu_0^{i} < \eps_0$ for all $0\leq i\leq \ell_0$ (Step 1),
\item[$\bullet$] $\exp(C_\theta\ell_0(2\eps_5\mu_0^{\ell_0})^{\theta})\leq \widetilde\Delta^{1/2}$  (Steps 1 and 2), and
\item[$\bullet$] $2\mu_0^{\ell_0+N_1-1}\eps_5 < \widetilde\eta^{-1} \|Df^{-1}\|^{-(N_1-1)} r_0$ (Step 1).
\end{itemize}

Having chosen the constants, we continue to the proof.  Fix $\widetilde\eps\in (0, \eps_5)$ and
$z\in \Lambda$.  We will choose integers $n_1, n_2\geq 0$ and consider the points:
$$z_0:=f^{m_0}(z),\quad z_1:=f^{m_1}(z), \quad z_2:= f^{m_2}(z), \text{ and } z_3:=  f^{m}(z), \text{ where}$$
$$m_0=\ell_0,\quad m_1=\ell_0 + n_1,\quad m_2=\ell_0 + n_1+n_2, \quad m=m_3=\ell_0 + n_1+n_2 + \ell_3.$$
We denote $\widetilde \eps_i=\hbox{Length}
(f^{m_i}(\cP^c(z,\widetilde\eps)))$ for $i\in\{0,1,2\}$.

We fix  round unstable disks $D_1, D_2, D_3$ centered at $z_1$, $z_2$, $z_3$:
$$D_1:= \W^{uu}(z_1, \widetilde\eta \widetilde\eps_1), \quad
D_2 = \W^{uu}\big(z_2, \tfrac{r_0}{\lambda_0(1+\widetilde\delta)}\big),\quad
D_3 =\W^{uu}(z_3,R).$$
Iterations at Step 0 are a preparation, ensuring that all distances are small during the two next steps.
The disk $D_1$ will ensure Conclusion (3) holds at Step 1. The disk $D_2$ will be shown to be included in $f^{n_2}(D_1)$ and will ensure Conclusion (2) holds at Step 2. 
Finally iterating at Step 3 will be used to recover the unstable size $R$:
the disk $D_3$ gives $D=f^{-m}(D_3)$. Conclusion (1) is obtained in this last step.
\bigskip

\noindent
{\bf Step 0:} {\em Iterate $\ell_0$ times. Prove that the images of local stable manifolds through any  $w\in \Lambda$ have diameter less than $r_0$ and that
$\widetilde\eps_0\leq 2\widetilde\eps \mu_0^{\ell_0}$.}

The size of the stable manifolds is controlled by the condition $\kappa_0^{\ell_0}\eps_0 < r_0$.
Then observe that
$\widetilde\eps_0= \hbox{Length}(f^{\ell_0}(\cP^c(z,\widetilde\eps))) \leq 2\widetilde\eps \mu_0^{\ell_0}$.
\bigskip

\begin{figure}[h]
\begin{center}
\includegraphics[scale=.21]{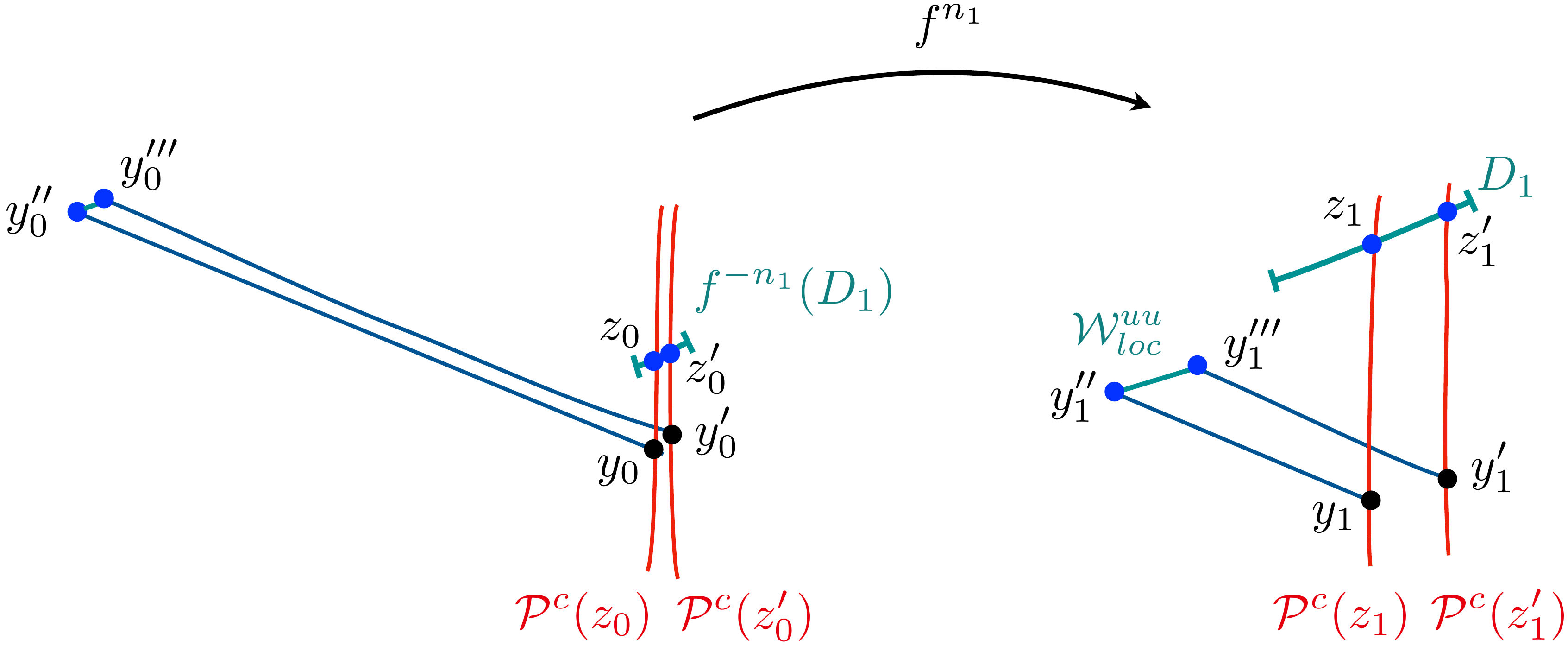}
\caption{Choice of $n_1$ and $D_1$ in the proof of Proposition~\ref{p.waitingdisk}: after $n_1$ iterations, the stable and unstable distances are smaller than $\widetilde \eta$ times the center lengths.}
\label{f=n1choice}
\end{center}
\end{figure}
\noindent
{\bf Step 1:} {\em Iterate  until the first time $n_1$ that $\operatorname{Diam}(f^{n_1}(\W^s(z_0,r_0)))$ is smaller than
${\widetilde\eta}\operatorname{Length}(f^{m_0+n_1}(\cP^c(z,\widetilde\eps)))$.  Prove $\widetilde\eps_1<\beta^{n_1}r_0$
and for all $z'\in f^{-m_1}(D_1)$,
$$c_{m_1}(z', \widetilde\eps)<\widetilde\Delta \text{ and } \tfrac 1 2 \widetilde \eps_1 \leq \operatorname{Length}(f^{m_1}(\cP^c(z',\widetilde\eps)))\leq 2\widetilde \eps_1.$$
Then prove Conclusion (3) for all $z'\in f^{-m_1}(D_1)$.}

Here are the details. For $k\geq 0$, $\hbox{Length}(f^{m_0+k}(\cP^c(z,\widetilde\eps)))< 2\mu_0^{m_0+k}\widetilde\eps$, and the diameter of $f^k(\W^s(z_0,r_0))$ is at least
$\|Df^{-1}\|^{-k}  r_0$.  Our choices of $\widetilde \eps, \widetilde\eps_5$  give $2\mu_0^{m_0+N_1-1}\widetilde\eps < \widetilde\eta^{-1} \|Df^{-1}\|^{-(N_1-1)} r_0$, which implies $n_1\geq  N_1$.

The diameter of  $f^{n_1-1}(\W^s(z_0,r_0))$ is at most $2\kappa_0^{n_1-1}r_0$, and is greater than $\widetilde\eta\cdot\hbox{Length}(f^{m_0+n_1-1}(\P^c(z,\widetilde\eps )))$, by our choice of $n_1$. It follows that $\widetilde\eps_1 < 2\mu_0 {\widetilde\eta}^{-1}\kappa_0^{n_1-1}r_0$.
Our choice of $N_1$ and the fact that $n_1\geq N_1$ imply  that $2\mu_0{\widetilde\eta}^{-1} \kappa_0^{n_1-1} <  \beta^{n_1}$.  Hence $\widetilde\eps_1< \beta^{n_1}r_0$.
 
 Now we do the distortion  and comparison estimates for $z'\in f^{-m_1}(D_1)$.

 \begin{claim}\label{c=5arcs}  For all  $z'\in f^{-m_1}(D_1)$, we have
 $c_{m_1}(z', \widetilde\eps) <\widetilde\Delta$ and
$$\tfrac 1 2 \widetilde \eps_1 < \operatorname{Length}(f^{m_1}(\cP^c(z',\widetilde\eps))) < 2\widetilde \eps_1.$$
\end{claim}
  
\begin{proof}
First we note that for any $0\leq k<m_0=\ell_0$,
$$\operatorname{Length}(f^{k}(\cP^c(z,\widetilde\eps)))\leq
2\widetilde \eps\mu_0^k<\eps_5\mu_0^k<\eps_0.$$
By definition of $n_1$, each $0\leq k<n_1$ satisfies
$$\operatorname{Length}(f^{m_0+k}(\cP^c(z,\widetilde\eps)))\leq
{\widetilde\eta}^{-1}\operatorname{Diam}(f^{k}(\cP^s(z_0,r_0)))\leq 2{\widetilde\eta}^{-1} \kappa_0^{k}r_0<{\widetilde\eta}^{-1} \beta^{k}r_0.$$

Fix an arbitrary $z'\in f^{-m_1}(D_1)$.
An inductive argument gives that  for all $0\leq k\leq m_1$:
$$\tfrac 1 2 \hbox{Length}(f^{k}(\P^c(z,\widetilde\eps ))) < \hbox{Length}(f^{k}(\P^c(z',\widetilde\eps ))) < 2 \hbox{Length}(f^{k}(\P^c(z,\widetilde\eps ))).$$
With the definition in \S\ref{ss.distortion}, these inequalities imply
$$c_{k+1}(z', \widetilde\eps) \leq
\exp(C_\theta \ell_0 (2 \eps_5\mu_0^{\ell_0})^\theta) \exp(C_\theta (2{\widetilde\eta}^{-1}r_0)^\theta/(1-\beta^\theta))
< \widetilde \Delta.$$
The distortion estimate in Lemma~\ref{l.basicdistort} then gives
$$\hbox{Length}(f^{k+1}(\P^c(z',\widetilde\eps)))\asymp_{\widetilde \Delta} \|Df^{k+1}(z')|_{E^c}\| 2\widetilde\eps.$$
Note that
$d(f^i(z),f^i(z'))\leq \widetilde \eta \lambda_0^{i-m_1}\widetilde \eps_1 < \widetilde \eta\lambda_0^{i-m_1}$.
In particular
$$\big|\log\tfrac {\|Df^{k+1}(z')|_{E^c}\|}{\|Df^{k+1}(z|_{E^c})\|}\big| <C_\alpha \sum_{i=0}^{k}
( \widetilde \eta\lambda_0^{i-m_1})^\alpha< \tfrac{C_{\alpha} \widetilde \eta^\alpha }{1-\lambda_0^\alpha}.$$
We thus obtain
$$\bigg|\log\tfrac{\hbox{Length}(f^{k}(\P^c(z',\widetilde\eps )))}{\hbox{Length}(f^{k}(\P^c(z,\widetilde\eps )))}\bigg|\leq \tfrac{C_{\alpha} \widetilde \eta^\alpha}{1-\lambda_0^\alpha} + 2 \log \widetilde \Delta < \log 2,$$
with our choices of $r_0,\widetilde \Delta$. This completes the induction.
As explained before, this implies
$c_{m_1}(z', \widetilde\eps)< \widetilde \Delta$.
\end{proof}
   
To prove Conclusion (3) of the proposition, we  fix $z'\in f^{-m_1}(D_1)$ and points  $x_{0},x_{1}, x_2 \in \P^c(z, \widetilde\eps)$ and $x_0',x_1',x_2'\in \P^c(z', \widetilde\eps)$ with $x_i \smile x_i'$ and  $d^c_{z'}(x_i', x_j')>\eta\, \widetilde \eps$ for $i\neq j$. 
 
Let $x_0''\in \W^s_{loc}(x_0)\cap\Lambda$ be the point  satisfying $\W^s_{loc}(x_0') \cap \W^{uu}_{loc}(x_0'')\neq \emptyset$, guaranteed by $x_0\smile x_0'$.
Since  $2\widetilde\eps_1 < 2\beta^{n_1}r_0 < 2\eps_2$, we are in the position to apply Lemma~\ref{l.def-length} to the points $f^{m_1}(x_i) , f^{m_1}(x_i')$; we just need to show that 
the center distances between the points $f^{m_1}(x_i'), f^{m_1}(x_j')$ are
larger than $\widetilde \eta \max_i d^c_{f^{m_1}(z')}(f^{m_1}(z'), f^{m_1}(x_i'))$
and that $\max_i d^c_{f^{m_1}(z')}(f^{m_1}(z'), f^{m_1}(x_i'))$ is larger than
the stable distance between $f^{m_1}(x), f^{m_1}(x'')$ and the unstable distance between $f^{m_1}(z)$ and $f^{m_1}(z')$.
 
For $i\neq j$,  since by assumption $d^{c}_{z'}(x_i',x_j') > \tfrac \eta 2 {\hbox{Length}(f^{k}(\P^c(z',\widetilde\eps )))}$
and as we have controlled the distortion
$c_{m_1}(z',\widetilde\eps)<\widetilde\Delta$,
we obtain
$$d^{c}_{f^{m_1}(z')}(f^{m_1}(x_i'), f^{m_1}(x_j')) >
\tfrac {\eta\widetilde\Delta} 2 {\hbox{Length}(f^{m_1}(\P^c(z',\widetilde\eps )))}
> \tfrac {\eta\widetilde\Delta\widetilde \eps_1} 4
> \widetilde \eta\widetilde\Delta\widetilde \eps_1 .$$
Since $\max_i d^c_{z'}(z', x_i')<\tfrac 1 2 {\hbox{Length}(\P^c(z',\widetilde\eps ))}$, we also get
$$\max_i d^c_{f^{m_1}(z')}(f^{m_1}(z'), f^{m_1}(x_i')) < \tfrac{\widetilde \Delta} 2 {\hbox{Length}(f^{m_1}(\P^c(z',\widetilde\eps )))}
< \widetilde \Delta\widetilde \eps_1.$$

On the other hand,  our choice of $n_1$ implies that   \[d^{s}_{f^{m_1}(z)}(f^{m_1}(x), f^{m_1}(x'')) < {\widetilde\eta} \,\hbox{Length}f^{n_1}(\P^c(z,\widetilde\eps))  = \widetilde\eta\widetilde\eps_1 < {\widetilde \Delta\widetilde \eps_1}.\]   Finally, by our choice of $D_1$, we get $d^{uu}(f^{m_1}(z),f^{m_1}(z'))<\widetilde\eta\widetilde\eps_1 < {\widetilde \Delta\widetilde \eps_1}$. Thus the hypotheses of  Lemma~\ref{l.def-length}  are satisfied. Item (2) gives for $i\neq j$
\[d^c_{f^{m_1}(z)}(f^{m_1}(x_i),f^{m_1}(x_j) )\asymp_{1+\widetilde\delta} d^c_{f^{m_1}(z')}(f^{m_1}(x_i'),f^{m_1}(x_j') ).\]
Taking ratios, iterating by $f^{-m_1}$, and using the distortion controls, we get
\[ \frac{\bar d^c_{z} (x_1,x_0)}{\bar d^c_{z}  (x_2,x_0)} \asymp_{\widetilde\Delta^2(1+\widetilde\delta)^2} \frac{\bar d^c_{z'} (x_1',x_0')}{\bar d^c_{z'}  (x_2',x_0')};\]
 which gives Conclusion  (3), since $\widetilde\Delta^2(1+\widetilde\delta)^2\leq  L$.

\bigskip

\noindent
{\bf Step 2:} {\em Iterate  until  the first time $n_2$ that the inner radius of $f^{n_2}(D_1)$ is greater than  $\tfrac{r_0}{\lambda_0(1+\widetilde\delta)}$.
Hence $D_2\subset f^{n_2}(D_1)$. Prove $n_2\leq N_2$, $\widetilde\eps_2 <\beta^{n_2}r_0$ and
$$\operatorname{Length}(f^{n_2}(\cP^c(w,2\widetilde\eps_1))) < 5\widetilde \eps_2 \text{ and } c_{n_2}(w,2\widetilde\eps_1)<\widetilde\Delta \text{ for all $w\in D_1$.}$$}

\noindent  Here are the details. Since the inner radius of $f^{k}(D_1)$ is at most $\|Df\|^{k}\widetilde\eta \widetilde\eps_1 < \|Df\|^{k}\widetilde\eta \beta^{N_1}r_0$, and since
$\beta^{N_1}\leq \left(\widetilde\eta \lambda_0 (1+\widetilde\delta)\|Df\|^{N_2}\right)^{-1}$, we obtain $n_2\geq N_2$.
  
The definition of $n_2$ and Item (1) of Lemma~\ref{l.def-length} imply, for $0\leq k\leq  n_2$,
\begin{equation*}   \text{\rm InnerRadius}(f^k(D_1)) \leq  \lambda_0^{k-n_2} r_0.
\end{equation*}
Write $\gamma :=  f^{m_1}(\P^c(z ,\widetilde\eps ))$.
The definition of $\nu_0$ gives
$$\frac{\text{\rm Length}(f^k(\gamma))/\text{\rm Length}(\gamma)}{\text{\rm InnerRadius}(f^k(D_1))/ \text{\rm Radius}(D_1)}<\nu_0^k.$$
Since $\text{\rm Radius}(D_1)=\widetilde \eta \;\text{\rm Length}(\gamma)$, we obtain
\begin{eqnarray*}
\text{\rm Length}(f^k(\gamma))& \leq &\frac{\widetilde \eta^{-1}\text{\rm Length}(f^k(\gamma))/\text{\rm Length}(\gamma)}{\text{\rm InnerRadius}(f^k(D_1))/ \text{\rm Radius}(D_1)} \text{\rm InnerRadius}(f^k(D_1))\\
&\leq& \widetilde \eta^{-1} \nu_0^k \lambda_0^{k-n_2} r_0
< \beta^{n_2}r_0 \leq \beta^kr_0,
\end{eqnarray*}
since $n_2\geq N_2$, and $N_2$ was chosen so that  $\widetilde\eta^{-1} \max\{\nu_0,\lambda_0^{-1}\}^{N_2}<\beta^{N_2}$.
 In particular $\widetilde\eps_2 < r_0\beta^{n_2}$ and $c_{n_2}(z_1,2\widetilde\eps_1)< \widetilde\Delta$.

We next need to show that $c_{n_2}(w,2\widetilde\eps_1)<\widetilde\Delta$ for all $w\in D_2$. 

\begin{claim}
For all  $w\in f^{-n_2}(D_2)$, we have
 $c_{n_2}(w, 2\widetilde\eps_1) <\widetilde\Delta$ and
$$\operatorname{Length}(f^{n_2}(\cP^c(w,2\widetilde\eps_1))) < 5\widetilde \eps_2.$$
\end{claim}
\begin{proof}
The proof is identical to the proof of the claim in Step 1.
\end{proof}
\bigskip

\noindent
{\bf Step 3:} {\em Iterate $\ell_3$ times. Prove $f^{\ell_3}(D_2)\supseteq \W^{uu}(f^{\ell_3}(z_2), R)=D_3$ and
$$c_{\ell_3}(w,5\widetilde\eps_2)<\widetilde\Delta\quad \text{  for all $w\in D_2$.}$$
Finish the proof of the proposition.}

\medskip

Our choice of $\ell_3$ implies that  $f^{\ell_3}(D_2)\supseteq \W^{uu}(f^{\ell_3}(z_2),R)$.  
By Steps 1 and 2, the length of $f^{m}(\P^c(z',\widetilde\eps))$ for $z'\in f^{-m}( \W^{uu}(f^m(z),R))$ is less than
$5\mu_0^{\ell_3}\widetilde \eps_2 < 5\mu_0^{\ell_3} \beta^{n_2}r_0$ which is smaller than $ \eps$, by our choice of $N_2$. This gives Conclusion (1) of  Proposition~\ref{p.waitingdisk}. We also get,  for all $w\in D_2$,
 \[c_{\ell_3}(w, 5\widetilde\eps_2)<\exp(C_\theta\ell_3( 5 \mu_0^{\ell_3}\beta^{n_2}r_0)^{\theta}) < \widetilde\Delta.\]
Since $D \subset f^{-m_2}(D_2) \subset f^{-m_1}(D_1)$, and using for any $z'\in D$
the distortion controls
$c_{m_1}(z',\widetilde \eps')$,
$c_{n_2}(f^{m_1}(z'),2\widetilde \eps_1)$, $c_{\ell_3}(f^{m_2}(z'),5\widetilde \eps_2)$,
we get from the previous steps $c_m(z',\widetilde \eps)<\widetilde \Delta^4<\Delta$. Together with Lemma~\ref{l.basicdistort}
this implies that Conclusion (2) of Proposition~\ref{p.waitingdisk} holds.

Using again $D \subset  f^{-m_1}(D_1)$, we deduce from Step 1 that Conclusion (3) holds.
This completes the proof of Proposition~\ref{p.waitingdisk}.
\end{proof}

\section{A criterion for existence of $cu$-discs: proof of  Theorem~\ref{t=s-transPartHyp}}\label{ss=cucriterion}
We continue to fix a $C^{1+}$ diffeomorphism $f\colon M\to M$ and an $f$-invariant  $uu$-lamination $\Lambda\subset M$ that is $s$-transverse
and satisfies the property SH, retaining the notation of the previous sections in this part of the paper.  In this section, we prove that $\Lambda$ contains a $cu$-disk, establishing Theorem~\ref{t=s-transPartHyp}.

We fix $\tau>0$ small. By Proposition~\ref{p.s-transverse},
the s-transversality property of $\Lambda$ is satisfied at scale $\tau$ for any $x\in \Lambda$, between arcs $\gamma,\gamma'$ contained in the strong unstable plaque $\W^{uu}(x,R)$. Moreover, there exists $\chi>0$
such that
\begin{equation}\label{e.eps0}
\rho^c_{\gamma(0)}(\gamma(0),\gamma'(0))>\chi, \text{ and }
\rho^c_{\gamma(1)}(\gamma(1),\gamma'(1))>\chi.
\end{equation}

Let $K$ be the $\W^{uu}$-section of $\Lambda$ given by the SH property as in \S~\ref{uu-section}.  By enlarging $R$ if necessary, for any $x\in\Lambda$, the  set  $\W^{uu}(x,R)$  meets some point of $K$.
Let $K_0 = \cap_{n\geq 0} f^n(K)$ be its hyperbolic core.
\medskip

We prove by induction the following property:

\begin{proposition}\label{p=evenspaced}
For any $\delta>0$ sufficiently small, there exists $\eps_6=\eps_6(\delta)$
such that for any $\eps\in(0,\eps_6)$ and any $N\geq 0$,
there are  points $x_0,\dots,x_{N}\in \Lambda$  that  satisfy:
 \begin{enumerate}
\item  $x_{N}\in K_0$  and   $x_0, \ldots, x_{N-1}\in \P^c(x_{N})$,
\item  $d^c_{x_N}(x_0,x_{N})\in [\eps/\|Df\|,\eps]$,
 and  
 \item for $0\leq i\leq N$, we have  
${\bar d^c_{x_N} (x_{0},x_{i})}\asymp_{(1+\delta)} { i \;  \bar d^c_{x_N}(x_0,x_{1})}$.
\end{enumerate}

\end{proposition}

\begin{proof}[Proof of Proposition~\ref{p=evenspaced}]Fix $\delta>0$ arbitrarily.
We choose some numbers $L,\Delta,\widetilde \delta>1$ close to $1$, which  will control the  Lipschitz constant of holonomies and  distortion under iteration, such that
\[\Delta^2 L^3(1+\widetilde \delta)^2<1+\delta.\]
We fix $\eps>0$ smaller than $\eps_6:=\min(\chi/\|Df\|,\eps_1)$ where $\eps_1(\Delta)$ is given by Corollary~\ref{c.distortion} applied to the hyperbolic set $K$.

We then prove the property inductively on $N$.
The case $N=0$ is trivial. The case $N=1$ will be explained at the end of the proof.
We now assume the property for $N-1\geq 2$ and prove it for $N$.

Let $\eta = 1/[(1+\widetilde\delta)L{(N+1)}]$, let $\eps_3({R,\eta,L})$, $\eps_4(\eta/L,L)>0$ be given by Lemmas~\ref{l.holonomy}  and~\ref{l.sholonomy} and let $\eps_5(\eps/\|Df\|, R, \eta, \Delta,L)>0$ be given by Proposition~\ref{p.waitingdisk}.
We fix $\widetilde \eps>0$ smaller than $\min\{\eps/\|Df\|, \eps_3,\eps_4,\eps_5,\eps_6(\widetilde \delta)\}$. 
We apply the inductive assumption for $\widetilde \delta,\widetilde \eps$ to get $N$ points
 $\widetilde x_0,\dots, \widetilde x_{N-1}$. They all belong to $\P^c(\widetilde x_{N-1})$, with $\widetilde x_{N-1}\in K_0$.

The s-transversality of $\Lambda$ gives arcs $\gamma,\gamma'\subset \W^{uu}(\widetilde x_{N-1},R)$.
Lemma~\ref{l.plaque} shows that one can parametrize them so that
$\gamma'(t)\in \cP^{cs}(\gamma(t))$ for all $t$.

For each $t$, one can define points $h^{uu}_t(\widetilde x_0),\dots,h^{uu}_t(\widetilde x_{N-1})$ in $\cP^{c}(\gamma'(t))$
as in \S~\ref{ss.uu-holonomy}: one first moves $\widetilde x_i$ by uu-holonomy to some point $\widetilde x'_i(t)\in \W^{uu}(\widetilde x_i)\cap \P^{cs}(\gamma'(t))$;
one then projects along the leaves of $\widehat \W^s_{\gamma'(t)}$ so that $h^{uu}_t(\widetilde x_i)$ belongs to $\widehat \W^s_{\gamma'(t)}(\widetilde x'_i(t))$.

\begin{figure}[h]
\begin{center}
\includegraphics[scale=.23]{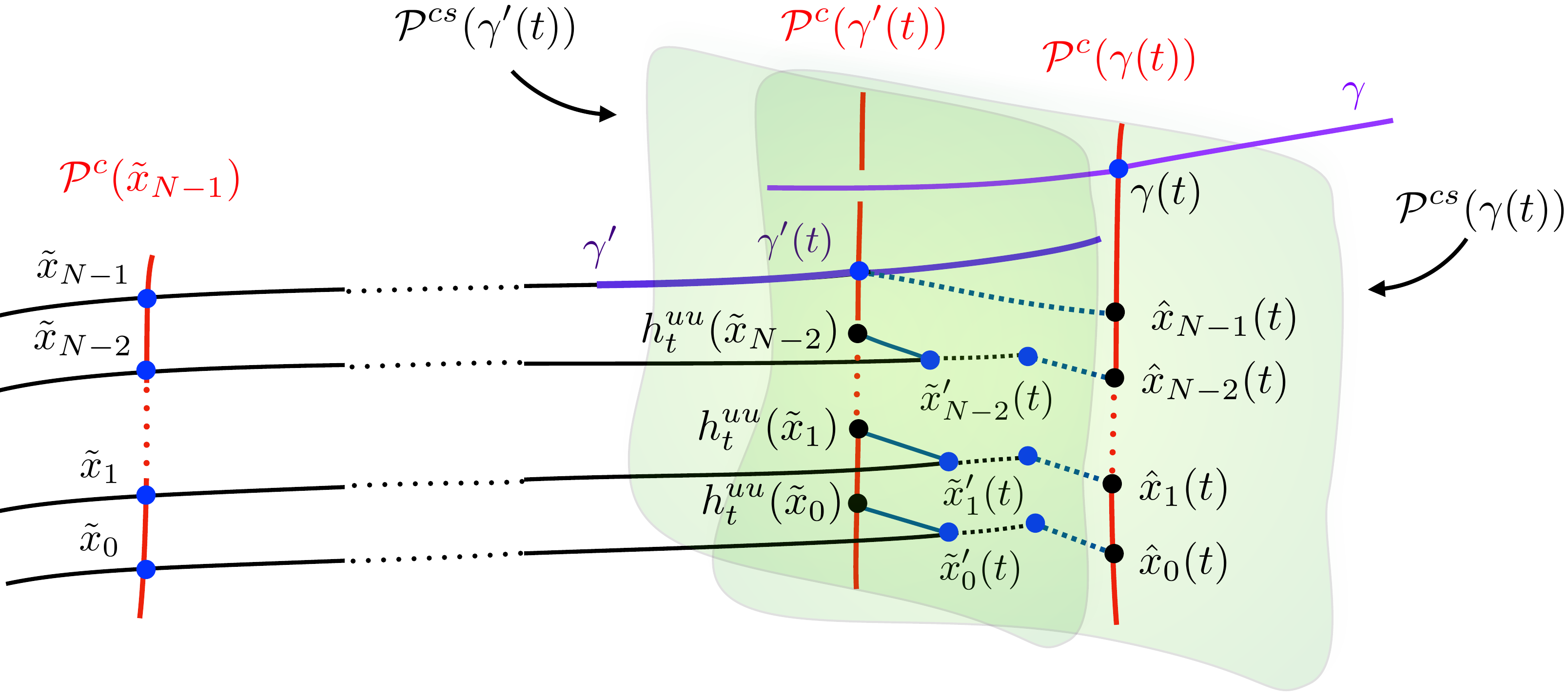}
\caption{Holonomy to two different $\P^{c}$-plaques.}
\label{f=doubleholonomy}
\end{center}
\end{figure}

By Lemma~\ref{l.plaque}, one can also project each $\widetilde x'_i(t)$ to $\P^c(\gamma(t))$ and define a point
$\widehat x_{i}(t)$: this is the unique point in $\P^c(\gamma(t))$ such that
$\widehat \W^s_{\gamma(t)}(\widehat x_{i}(t))$ intersects $\W^{uu}_{loc}(\widetilde x_i'(t))$.
In this way, $h^{uu}_t(\widetilde x_i)$ and $\widehat x_{i}(t)$ are suus-related.
Each point $\widehat x_{i}(t)$ is the image of $\widetilde x_i$ under a composition of two operations:  first the strong unstable holonomy  $h^{uu}_t\colon  \P^c(\widetilde x_{0}) \to  \P^c(\gamma'(t))$, and second, the $suus$-holonomy relation between $\P^c(\gamma'(t))$ and  $\P^c(\gamma(t))$.

Remember that we have assumed $N\geq 2$, so that $\widehat x_{0}(t),\widehat x_{1}(t)$ are defined.
By Items (1) in Lemmas~\ref{l.holonomy} and~\ref{l.sholonomy}, the points $\widehat x_{0}(t)$ and $\widehat x_{1}(t)$ are distinct.
We then consider the function
\[\phi(t) := \frac{\bar d^c_{\gamma(t)}(\widehat x_{N-1}(t),\gamma(t))}{  \bar d^c_{\gamma(t)}(\widehat x_0(t), \widehat x_1(t))}.
\]
Note that it does not depend on the choice of a center orientation and that it is continuous.
The point $\widehat x_{N-1}(t)$ is the projection of $\gamma'(t)=\widetilde x'_{N-1}(t)$ to $\P^c(\gamma(t))$.
We recall that $\gamma'(0)$ and $\gamma'(1)$ are on different sides of $\W^{uu}(\widetilde x_{N-1},R)$ relative to $\gamma$;
hence $\bar d^c_{\gamma(0)}(\widehat x_{N-1}(0),\gamma(0))$, $\bar d^c_{\gamma(1)}(\widehat x_{N-1}(1),\gamma(1))$
have different signs and satisfy
$$|\bar d^c_{\gamma(0)}(\widehat x_{N-1}(0),\gamma(0))|>\chi \text{ and }
|\bar d^c_{\gamma(1)}(\widehat x_{N-1}(1),\gamma(1))|>\chi.$$ 
On the other hand $|\bar d^c_{\gamma(t)}(\widehat x_{0}(t), \widehat x_{1}(t))|$ does not vanish and is smaller than $\chi$.
One deduces that $\phi(0),\phi(1)$ have different signs and have modulus larger than $1$.
In particular there exists $\hat t$ such that $\phi(\hat t)=1$.
We set $\widehat z := \gamma(\hat t)$.

\begin{figure}[h]
\begin{center}
\includegraphics[scale=.23]{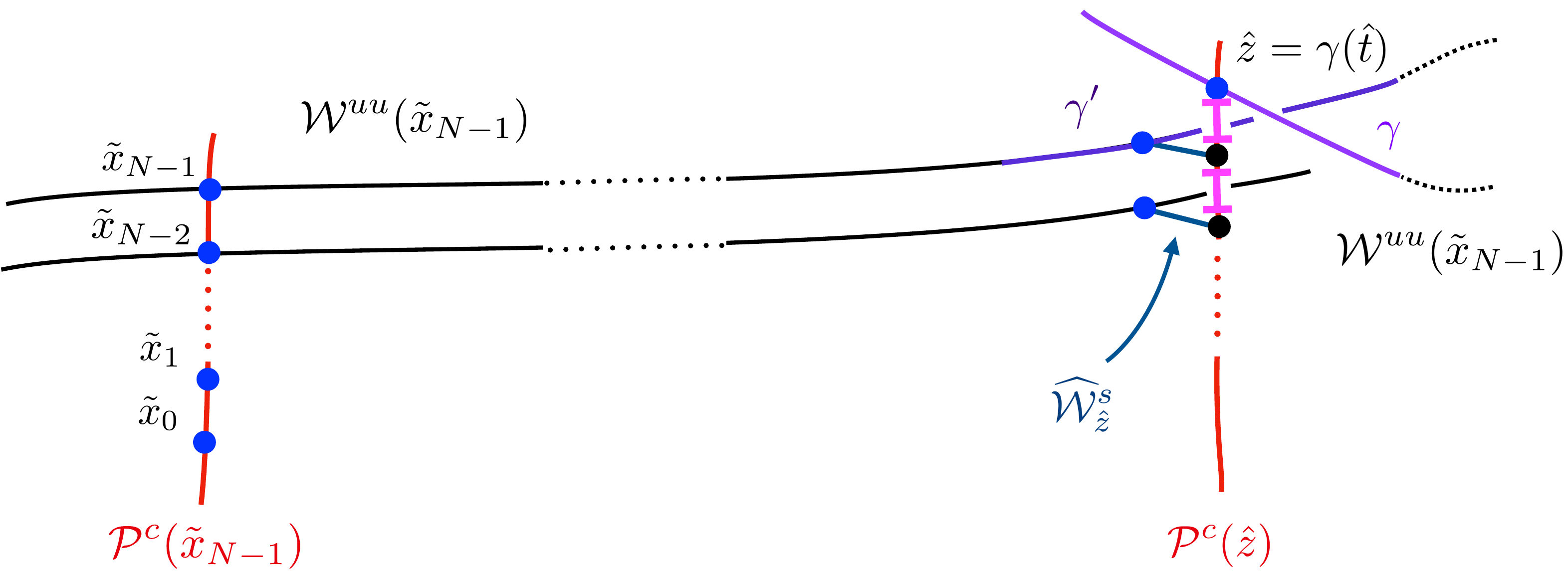}
\caption{Proof of Proposition~\ref{p=evenspaced}: choice of $\widehat z$.}
\label{f=z0choice}
\end{center}
\end{figure}

Consider the points
$\widehat x_N:=\widehat z$ and $\widehat x_i:=\widehat x_i(\widehat t)$.
All the points $h^{uu}_t(\widetilde x_i)$ belong to $\cP^c(\gamma'(t))$, all the points $\widehat x_i$ belong to $\cP^c(\gamma(t))$
and $\gamma'(t)$ belongs to $\cP^{cs}(\gamma(t))$.
By our choice of $\eta$, Items (2) in  Lemmas~\ref{l.holonomy} and~\ref{l.sholonomy}  hold,  and we obtain
for $0\leq i\neq j\neq k< N$
\[
\frac{\bar d^c_{\widehat z}(\widehat x_{i},\widehat x_{j})}{\bar d^c_{\widehat z}( \widehat x_{j},\widetilde x_{k})}
\asymp_{L^2} \frac{ \bar d^c_{\widetilde x_{N-1}}(\widetilde x_{i}, \widetilde x_{j})}{\bar d^c_{\widetilde x_{N-1}}(\widetilde x_{j},\widetilde x_{k})}.
\]
This, our inductive assumption and
$\bar d^c_{\widehat z}(\widehat x_0,\widehat x_1)= \bar d^c_{\widehat z}(\widehat x_{N-1}, \widehat x_N)$
give \begin{equation}\label{e.homogeneity-at-widehat-z}
{ \bar d^c_{\widehat z}(\widehat x_{0},\widehat x_{i})}\asymp_{L^2(1+\widetilde \delta)^2} {i\; \bar d^c_{\widehat z}( \widehat x_{0},\widehat x_{1})}
\text{  for $i=1,\ldots, N$.}
\end{equation}

When $E^c$ is not uniformly expanded we need to move the point $\widehat z=\widehat x_0$.
Proposition~\ref{p.waitingdisk}  implies  that   there exist an integer $m\geq 1$ and a disk $D\subset  \W^{uu}(\widetilde x_{0},R)$ such that  $f^m(D) = \W^{uu}(f^{m}(\widehat z),R)$ and  that for any $z\in D$,  the distortion of $f^m$ on $\cP^c(\widehat z)$ is controlled by $\Delta$.

\begin{claim} For every $\bar x_N:=z\in D$, there exist points $\bar x_0,\dots,\bar x_{N-1}$ in $\P^c(z,\widetilde\eps) \cap \W^s_{loc}(\Lambda)$,  such that
for any $0< i \leq N$, we have
\[{\bar d^c_z( \bar x_0 ,\bar x_{i})} \asymp_{ L^3(1+\widetilde \delta)^2} {i\; \bar d^c_z(\bar x_0, \bar x_{1})} .\]
\end{claim}
\begin{proof}
Fix $z\in D$.
The points $\widehat x_{i} \in \cP^c(\widehat z)$ have been obtained by projecting
points $\widetilde x'_i$ close to $\widehat z$. The points $\bar x_{i}$
are obtained similarly by projecting $\widetilde x'_i$ to $\cP(z)$ using Lemma~\ref{l.plaque},
so that $\widehat \W^s_z(\bar x_{i})$ intersects $\W^u_{loc}(\widetilde x'_i)$.

In this way the points $\widehat x_i$ and $\bar x_i$ are suus-related.
Combining Item (3) in Proposition~\ref{p.waitingdisk} with \eqref{e.homogeneity-at-widehat-z}, one gets the desired inequalities of the claim.
\end{proof}

Next, we select $z\in D$ such that $f^m(z)\in K$ and set $\bar x_N = z$.
We also fix $\bar x_0,\dots,\bar x_{N-1}\in \P^c(\bar x_N,\widetilde\eps) \cap \W^s_{loc}(\Lambda)$ given by the previous claim.
Finally, we choose $n$ so that  $d^c_{f^n(z)}(f^n(\bar x_N), f^{n}(\bar x_{0}))\in [\eps/\|Df\|,\eps)$.

Note that $n > m$ by Item (1) in Proposition~\ref{p.waitingdisk}.
Item (2) in Proposition~\ref{p.waitingdisk} and Corollary~\ref{c.distortion} then imply
\begin{equation*}
\begin{split}
\frac {\bar d^c_{f^n(z)}(f^n\bar x_0,f^n\bar x_{i})} {\bar d^c_{f^n(z)}(f^n\bar x_0,f^n\bar x_{1})} \asymp_{\Delta^2}
\frac { \;\bar d^c_z(\bar x_0, \bar x_{i})}{\bar d^c_z(\bar x_0, \bar x_{1})}.  
 \end{split}\end{equation*}
With our choice of numbers $L,\Delta,\widetilde\delta $, this ensures for any $0< i \leq N$
\[ {\bar d_{f^n(z)}(f^n\bar x_0,f^n\bar x_{i})}\asymp_{(1+\delta)}{i\;\bar d^c_{f^n(z)}(f^n\bar x_0,f^n\bar x_{1})}.\]

For a sequence $\widetilde \eps_k\to 0$, we repeat the construction of points $\bar x_i(k)$  and forward times $n_k\to\infty$.
Extracting convergent subsequences $f^{n_k}(\bar x_i(k))\to x_i$, we obtain the points $x_0,\dots x_N\in \Lambda $ satisfying the Item (3) in Proposition~\ref{p=evenspaced}.
The points $\bar x_0(k),\ldots, \bar x_{N-1}(k) $ belong to the center plaque $\P^c(\bar x_N(k))$ and to the union of local stable manifolds of $\Lambda$,  with $\bar x_{0}(k)\in K$, and so the limits $x_0,\dots,x_N$ belong to $\Lambda$,  with
$x_0\in K_0$,
and lie an interval in  $\P^c(x_0)$ whose length is contained in $[\eps/\|Df\|, \eps]$.
We have thus checked all the inductive assumptions when $N\geq 2$.

In the case $N=1$, the proof is similar but simpler: we select the points $\widehat z$ and then $\bar x_1$ in a similar way as $\bar x_N$, but we do not consider
points $\widetilde x_i$, $i\geq 1$. The projection of some point $\tilde x'_0\in \W^{uu}(\widetilde x_0)$ to $\cP^c(\bar x_0,\widetilde \eps)$
defines the point $\bar x_1\neq \bar x_0$ but we do not need to check the estimates of the claim.
As before the points $x_0,x_1$ are obtained as limits of sequences $f^{n_k}(\bar x_i(k))$.
The proposition is now proved.\end{proof}

Having proved Proposition~\ref{p=evenspaced}, we fix $\delta,\eps>0$ arbitrarily and let $N\to \infty$. Extracting a subsequence of center plaques produced by Proposition~\ref{p=evenspaced}, we obtain a center plaque  $\P^c(x,\eps)\subset\Lambda$, with $x\in K_0$.
Recall that $K_0$ is a compact hyperbolic set so that $\P^c(x,\eps)$ is contained in the unstable set of $x$. In particular, the union $\cD^{cu}(y)$ of  the local $\W^{uu}$ plaques through the points $y\in \P^c(x,\eps)$ is contained in the unstable manifold of $x$, which is tangent at points of $\Lambda$ to $E^{c}\oplus E^{uu}$, and which is contained in $\Lambda$ by construction. This proves that $\Lambda$ contains a $cu$-disk, completing the proof of Theorem~\ref{t=s-transPartHyp}.
\qed


\begin{thebibliography}{BPSW}

\bibitem{ABV}
J. Alves,  C. Bonatti, M. Viana,
{S{RB} measures for partially hyperbolic systems whose central direction is mostly expanding.}
\emph{Invent. Math.} \textbf{140} (2000), 351--398.

\bibitem{ALOS} S. Alvarez,  M. Leguil, Martin, D.  Obata, B. Santiago,
\emph{Rigidity of $u$-Gibbs measures near conservative Anosov diffeomorphisms on ${\TT}^3$.}
arXiv:2208.00126.

\bibitem{A1} A. Avila, S. Crovisier, A.  Wilkinson,
{$C^1$ density of stable ergodicity.}
\emph{Adv. Math.} {\bf 379}  (2021), 107496, 68 pp.

\bibitem{A5} A. Avila, Artur, S. Crovisier, A. Eskin, R.  Potrie, A. Wilkinson, Z. Zhang,
\emph{$uu$-states for Anosov diffeomorphisms of ${\bf T}^3$} (working title).
\emph{In preparation}.

\bibitem{BQ} Y. Benoist, J-F.  Quint,
{Mesures stationnaires et ferm\'es invariants des espaces homog\`enes}.
\emph{Ann. of Math.} {\bf 174} (2011), 1111--1162.

\bibitem{BDU} C. Bonatti, L. D\'iaz, R. Ures,
{Minimality of strong stable and unstable foliations for partially hyperbolic diffeomorphisms.}
\emph{J. Inst. Math. Jussieu} {\bf 1} (2002),  513--541.

\bibitem{Bonatti-Wilkinson} C. Bonatti, A. Wilkinson,
Transitive partially hyperbolic diffeomorphisms on 3-manifolds.
\emph{Topology} \textbf{44} (2005), 475--508.


\bibitem{BRH} A. Brown, F. Rodriguez Hertz,
{Measure rigidity for random dynamics on surfaces and related skew products}.
\emph{J. Amer. Math. Soc.} {\bf 30} (2017),  1055--1132.

\bibitem{brown} M. Brown,
{Locally flat imbeddings of topological manifolds}.
\emph{Annals of Mathematics} {\bf 75} (1962), 331--341.

\bibitem{BW} K. Burns, A. Wilkinson,
{On the ergodicity of partially hyperbolic systems}.
\emph{Annals of Mathematics} {\bf 171} (2010),  451--489.

\bibitem{CR} P. Carrasco, F. Rodriguez-Hertz,
{Equilibrium states for center isometries}.
\emph{J. Inst. Math. Jussieu} {\bf 23} (2024), 1295--1355.

\bibitem{CPS} S. Crovisier, R. Potrie, M. Sambarino,
{Finiteness of partially hyperbolic attractors with one-dimensional center.}
\emph{Ann. Sci. ENS} {\bf 53} (2020), 559--588. 

\bibitem{CDP} S. Crovisier, D. Obata, M. Poletti,
{Uniqueness of $u$-Gibbs Measures for Hyperbolic Skew Products on $\TT^4$}.
\emph{Comm. Math. Phys.} {\bf 405} (2024), Paper 215.

\bibitem{Dani} S. G. Dani,  Invariant measures and minimal sets of horoshperical flows. {\em  Invent. Math. } {\bf 64} (1981), 357--385.

\bibitem{DaniMargulis} S. G. Dani and G. A. Margulis,  Orbit closures of generic unipotent flows on homogeneous spaces of $\mathrm{SL}(3, \mathbb{R})$. {\em Math. Ann.} {\bf 286} (1990), 101--128.

\bibitem{Dolgustates} D. Dolgopyat,
{Limit theorems for partially hyperbolic systems.}
\emph{Trans. Amer. Math. Soc.} {\bf 356}  (2004), 1637--1689.

\bibitem{EL} A. Eskin, E. Lindenstrauss,
\emph{Random Walks on Locally Homogenous Spaces}.
Preprint.

\bibitem{EM} A. Eskin, M. Mirzakhani,
{Invariant and stationary measures for the $\SL(2,\RR)$  action on moduli space}.
\emph{Publ. Math. Inst. Hautes \'Etudes Sci.} {\bf127} (2018), 95--324.

\bibitem{EPZ} A. Eskin, R. Potrie, Z. Zhang,
\emph{Geometric properties of partially hyperbolic measures and applications to measure rigidity}.
arXiv:2302.12981.

\bibitem{GYYZ} S. Gan, F. Yang, J. Yang, R. Zheng,
{Statistical properties of physical-like measures}.
\emph{Nonlinearity} {\bf 34}  (2021)  1014--1029.

\bibitem{GMK} A. Gogolev, I. Maimon, A. Kolmogorov,
{A numerical study of Gibbs $u$-measures for partially hyperbolic diffeomorphisms on ${\mathbb T}^3$}.
\emph{Exp. Math.} {\bf 28} (2019), 271--283.

\bibitem{Hammerlindl} A. Hammerlindl,
{Leaf conjugacies on the torus.}
\emph{Ergod. Th. Dynam. Syst.} \textbf{33} (2013), 896--933.


\bibitem{Margulis1} G. A. Margulis, {Formes quadratiques ind\'efinies et flots unipotents sur les spaces homog\`enes}. {\em C. R. Acad. Sci. Paris Ser. I} {\bf 304} (1987), 247--253.

\bibitem{Margulis2} G. A. Margulis,  Discrete subgroups and ergodic theory, in Number Theory. \emph{Trace Formulas and Discrete Subgroups} (a symposium in honor of A. Selberg), pages 377--398, Academic Press, Boston, MA, 1989.

\bibitem{Margulis3} G. A. Margulis,  Indefinite quadratic forms and unipotent flows on homogeneous spaces. In {\em Dynamical Systems and Ergodic Theory} (Warsaw 1983), 399--409. Banach Center Publ. \textbf{23}. PWN - Polish Scientific Publ., Warsaw, 1989.

\bibitem{katz} A. Katz,
{Measure rigidity of Anosov flows via the factorization method}.
\emph{Geom. Funct. Anal.} {\bf 33} (2023),  468--540.


\bibitem{NH} G. N\'u\~nez, J. Rodriguez Hertz,
{Minimality and stable Bernoulliness in dimension 3.}
\emph{Discrete Contin. Dyn. Syst.} {\bf 40} (2020), 1879--1887.

\bibitem{NOH} G. N\'u\~nez, D. Obata, J. Rodriguez-Hertz,
{New examples of stably ergodic diffeomorphisms in dimension 3.}
\emph{Nonlinearity} {\bf 34} (2021) 1352--1365.

\bibitem{pesinsinai} Y. Pesin, Y. Sina\u\i,
{Gibbs measures for partially hyperbolic attractors}.
\emph{Ergod. Th. Dynam. Syst.} {\bf 2} (1982), 417--438.

\bibitem{Potrie} R. Potrie,
{Partial hyperbolicity and foliations in $\mathbb{T}^3$.}
\emph{J. Mod. Dyn.} \textbf{9} (2015), 81--121.

\bibitem{PSW} C. Pugh, M. Shub, A. Wilkinson, H\"older foliations.
{\em Duke Math. J.} {\bf  86} (1997),  517--546.

\bibitem{PujSam} E. Pujals, M. Sambarino,
{A sufficient condition for robustly minimal foliations}.
\emph{Ergod. Th. Dynam. Syst.} {\bf 26} (2006),  281--289. 

\bibitem{Ratner1} M. Ratner, Strict measure rigidity for unipotent subgroups of solvable groups. {\em  Invent. Math.} {\bf 101} (1990), 449--82.

\bibitem{Ratner2} M. Ratner,  On measure rigidity of unipotent subgroups of semisimple groups. {\em  Acta. Math.} {\bf 165} (1990), 229--309.

\bibitem{Ratner3} M. Ratner,  On Raghunathan's measure conjecture, {\em Ann. of Math.} {\bf 134} (1991), 545--607.

\bibitem{Ratner4} M. Ratner, Raghunathan's topological conjecture and distributions of unipotent flows. {\em Duke Math. J.} {\bf 63} (1991), 235--290.

\bibitem{HHU07} F. Rodriguez-Hertz, J. Rodriguez-Hertz, R. Ures,
{Some results on the integrability of the center bundle for partially hyperbolic diffeomorphisms}.
\emph{Partially hyperbolic dynamics, laminations, and Teichm\"uller flow}, 103--109.
\emph{Fields Inst. Commun.} {\bf 51}
American Mathematical Society, Providence, RI, 2007.

\bibitem{HUY22} J. Rodriguez-Hertz, R. Ures, J. Yang,
{Robust minimality of strong foliations for DA diffeomorphisms: $cu$-volume expansion and new examples}.
\emph{Trans. Amer. Math. Soc.} {\bf 375} (2022), 4333--4367.

\bibitem{Shah} N. Shah,  Ph.D. thesis, Tata Institute for Fundamental Research.

\bibitem{Starkov1} A. N. Starkov,  Solvable homogeneous flows. {\em Math. Sbornik}  {\bf 176} (1987), 242--259.

\bibitem{Starkov2} A. N. Starkov, The ergodic decomposition of flows on homogenous spaces of finite volume. {\em Math. Sbornik} {\bf 180}  (1989), 1614--1633.
\end{thebibliography}
\end{document}